\documentclass[11pt,letterpaper]{article}
\usepackage[margin=1in]{geometry}
\usepackage{amsmath, amssymb, amsthm}
\usepackage{mathtools}
\usepackage{hyperref}
\hypersetup{colorlinks=true, linkcolor=blue, urlcolor=blue, citecolor=blue,
  pdftitle={The Cone Projection f(z)=z/(1+|z|/R): Geometric structure and the Self-Directrix Theorem},
  pdfauthor={George M. Georgiou}}
\usepackage{enumitem}
\usepackage{tikz}
\usetikzlibrary{calc, arrows.meta, decorations.pathreplacing, positioning}
\usepackage{caption}
\usepackage{placeins}
\usepackage{tabularx}
\newcolumntype{L}{>{\raggedright\arraybackslash}X}

\numberwithin{equation}{section}

\theoremstyle{plain}
\newtheorem{theorem}{Theorem}[section]
\newtheorem{proposition}[theorem]{Proposition}
\newtheorem{corollary}[theorem]{Corollary}
\newtheorem{lemma}[theorem]{Lemma}
\theoremstyle{definition}
\newtheorem{definition}[theorem]{Definition}
\theoremstyle{remark}
\newtheorem*{remark}{Remark}

\title{The Cone Projection $f(z)=\dfrac{z}{1+|z|/R}$ \\[2pt]
\large Geometric structure and the Self-Directrix Theorem}
\author{\small George M. Georgiou\thanks{\small School of Computer Science and Engineering, California State University, San Bernardino}\\[3pt]
\small \href{mailto:georgiou@csusb.edu}{georgiou@csusb.edu}}
\date{}

\begin{document}
\maketitle

\begin{abstract}
\noindent
The cone projection $f_R(z)=z/(1+|z|/R)$ is a radial homeomorphism from $\mathbb{C}$ onto the open disk $D_R$ of radius $R$, obtained by an elementary cone-and-perpendicular construction (independent of the cone's height) and governed by the reciprocal lens identity $1/|f_R(z)|=1/|z|+1/R$. Its main Euclidean feature is the \emph{Self-Directrix Theorem}: every line $\ell$ not through the origin maps to the focus-side arc of the conic with focus $O$, directrix $\ell$ \emph{itself}, eccentricity $R/d$, and semi-latus rectum $R$, so the single distance $d=\operatorname{dist}(O,\ell)$ fixes the ellipse/parabola/hyperbola trichotomy. The \emph{Confocal--Codirectrix Theorem} extends this from lines to every focal polar locus of a fixed focus--directrix pencil, keeping the focus and directrix while lowering the eccentricity by $1/e\mapsto1/e+\delta/R$; the image of a circle, by contrast, is generally a circular quartic rather than a conic. The same lens identity organizes the remaining structure: a curvature-additive composition law and its flow, a raywise cross-ratio structure, and higher-dimensional, metric, and axiomatic results.
\end{abstract}

\paragraph{Main contribution.}
The radial formula and its one-dimensional M\"obius algebra are elementary and serve here as organizing structure. The main geometric contribution is the Self-Directrix Theorem, in which each input line reappears as the directrix of its image conic, together with its generalization, the Confocal--Codirectrix Theorem, which extends this focus- and directrix-fixing behavior to every focal polar locus in a fixed focus--directrix pencil. Here ``confocal'' means sharing the distinguished focus $O$, not both foci in the classical two-focus sense (see Section~\ref{sec:confocal}). The remaining sections place these results in an algebraic, projective, dynamical, and metric framework.

\paragraph{Prior appearance in neural networks.}
The same radial map $f(z)=z/(1+|z|/R)$ has appeared earlier, in a different context, as a complex-valued neuron activation function. It was introduced by Georgiou~\cite{georgiou1992diss} and by Georgiou and Koutsougeras~\cite{georgiou1992tcs}, and is discussed in this role, for example, in the textbook chapter by Adali and Li~\cite[Eq.~(1.45)]{adali2010}, as a bounded, ray-preserving activation function for complex-domain backpropagation, where it plays the role of a complex analogue of the real sigmoid: the modulus $|z|$ is squashed into the open disk of radius $R$ while the phase $\arg z$ is preserved. The dissertation~\cite{georgiou1992diss} records the same cone-and-perpendicular geometric interpretation that organizes the present paper, and establishes the magnitude formula $|f(z)|=|z|/(1+|z|/R)$ from a similar-triangles argument in the cone's axial cross-section. The present paper proceeds from this construction in a purely geometric direction (composition and curvature addition, the partial group and its flow, the cross-ratio and harmonic-conjugate structure on rays, the Self-Directrix Theorem, the higher-dimensional extension, and the flat pullback metric), which is independent of the neural-network setting in which the function was originally proposed.

\paragraph{Notation.}
Throughout, $R>0$, $D_R=\{z\in\mathbb{C}:|z|<R\}$, and
\[
  f_R(z)=\frac{z}{1+|z|/R}.
\]
When polar coordinates are used, we write $z=\rho e^{i\theta}$ with $\rho=|z|$. The point $O$ denotes the origin.

\section{The construction and the height-independence theorem}\label{sec:setup}

Fix $R,h>0$. Let $\mathcal{K}_{R,h}\subset\mathbb{R}^3$ be the cone with apex at the origin $O=(0,0,0)$, axis along the third coordinate, height $h$, and circular base of radius $R$ centered at $(0,0,h)$. Identify the complex plane $\mathbb{C}$ with the horizontal plane $\{(x,y,0):x,y\in\mathbb{R}\}$; when no confusion can arise, the same symbol $z$ denotes both a complex number and the corresponding point $(\operatorname{Re}z,\operatorname{Im}z,0)$.

\paragraph{The construction.}
Set $f(0):=0$. (For $z=0$ the construction below degenerates: the segment $L_0$ runs along the cone axis and meets the lateral surface only at the apex $O$, whose foot is again $O$; so this convention agrees with the construction and makes $f$ continuous at the origin.) For $z\in\mathbb{C}\setminus\{0\}$:
\begin{enumerate}[label=(\arabic*)]
  \item Draw the line segment $L_z$ from $z$ to $(0,0,h)$, the center of the base of the cone.
  \item Let $P\in L_z$ be the (unique) point at which $L_z$ meets the lateral surface of $\mathcal{K}_{R,h}$.
  \item Drop the perpendicular from $P$ onto $\mathbb{C}$ and call its foot $f(z)\in\mathbb{C}$.
\end{enumerate}
(The point $P$ in step~(2) is unique: in the axial cross-section through $z$, the lateral surface meets the half-plane on the side of $z$ in a single edge segment from $O$ to the rim, and $L_z$ crosses the line of that edge exactly once, since its endpoints $z$ and $(0,0,h)$ lie on opposite sides of that line; the explicit computation is in the proof of Proposition~\ref{prop:p10}.) The three-dimensional picture is shown in Figure~\ref{fig:3d}; the side view, in which $z$, $f(z)$, $P$, $O$ and the cone axis are coplanar, is shown in Figure~\ref{fig:side}.

\begin{figure}[htbp]
\centering
\begin{tikzpicture}[
  x={(1cm, 0cm)},
  y={(0.42cm, 0.32cm)},
  z={(0cm, 1cm)},
  scale=1.0,
  >={Latex[length=2mm]},
  line cap=round, line join=round, font=\small
]
  \def\h{4.0}\def\r{1.5}\def\Zx{4.5}\def\Zy{-0.5}
  \pgfmathsetmacro{\zmag}{sqrt(\Zx*\Zx + \Zy*\Zy)}
  \pgfmathsetmacro{\tt}{\zmag/(\zmag + \r)}
  \pgfmathsetmacro{\Px}{\Zx*(1-\tt)}
  \pgfmathsetmacro{\Py}{\Zy*(1-\tt)}
  \pgfmathsetmacro{\Pz}{\h*\tt}

  \def\xmin{-2.6}\def\xmax{5.5}\def\ymin{-2.0}\def\ymax{2.5}
  \draw[thin] (\xmin,\ymin,0) -- (\xmax,\ymin,0)
              -- (\xmax,\ymax,0) -- (\xmin,\ymax,0) -- cycle;

  \draw[dashed, thin] plot[domain=0:360, samples=72, smooth cycle]
        ({\r*cos(\x)},{\r*sin(\x)},0);

  \draw[thick] plot[domain=0:360, samples=72, smooth cycle]
        ({\r*cos(\x)},{\r*sin(\x)},\h);

  \draw[thick] (0,0,0) -- (-\r,0,\h);
  \draw[thick] (0,0,0) -- ( \r,0,\h);

  \draw[thin]  (0,0,0) -- (0,0,\h);

  \draw[->, thin] (0,0,\h) -- (\r,0,\h);

  \draw[thin] (0,0,0) -- (\Zx,\Zy,0);
  \draw[very thick] (0,0,\h) -- (\Zx,\Zy,0);

  \draw[thin] (\Px,\Py,\Pz) -- (\Px,\Py,0);

  \fill (0,0,0)         circle (1.4pt);
  \fill (0,0,\h)        circle (1.4pt);
  \fill (\Zx,\Zy,0)     circle (1.4pt);
  \fill (\Px,\Py,\Pz)   circle (1.4pt);
  \fill (\Px,\Py,0)     circle (1.4pt);

  \node[below left, yshift=6pt]    at (0,0,0)            {$(0,0,0)$};
  \node[above left, yshift=-9pt]    at (0,0,\h)           {$(0,0,h)$};
  \node[above]                      at ({0.5*\r},0,\h)    {$R$};
  \node[below right]                at (\Zx,\Zy,0)        {$z$};
  \node[above right]                at (\Px,\Py,\Pz)      {$P$};
  \node[below right]                at (\Px,\Py,0)        {$f(z)$};
  \node                             at (4.15,1.8,0)       {\textsc{complex plane}};
\end{tikzpicture}
\caption{Geometric interpretation of $f(z)$. The cone $\mathcal{K}_{R,h}$ has its apex on the complex plane at the origin and opens upward to its base of radius $R$ centered at $(0,0,h)$. The line from $z$ to $(0,0,h)$ meets the lateral surface at $P$; the orthogonal projection of $P$ onto $\mathbb{C}$ is $f(z)$. The dashed circle in the plane is the projection of the cone's rim and forms the boundary of
\(D_R=\{z\in\mathbb C:\ |z|<R\}\).}
\label{fig:3d}
\end{figure}
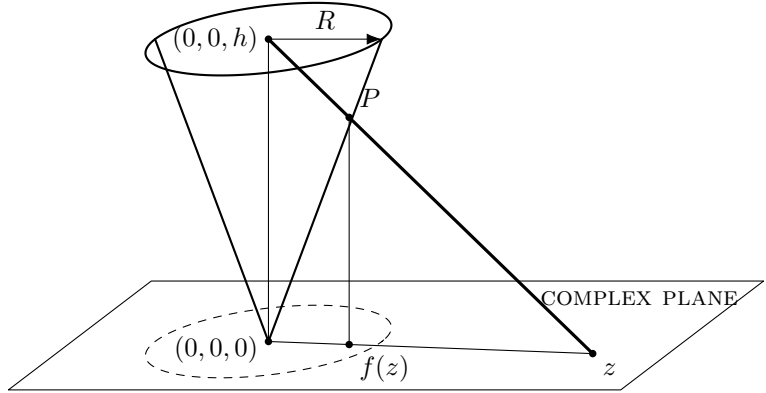

\begin{figure}[htbp]
\centering
\begin{tikzpicture}[
  scale=1.05,
  >={Latex[length=2.2mm]},
  line cap=round, line join=round, font=\small
]
  \def\h{4.0}\def\r{1.5}\def\Zx{6.0}
  \pgfmathsetmacro{\tt}{\Zx/(\Zx+\r)}
  \pgfmathsetmacro{\Px}{\r*\tt}
  \pgfmathsetmacro{\Py}{\h*\tt}
  \pgfmathsetmacro{\rt}{1.0*\r}
  \draw[thin] (-\r,\h) -- (0,0);
  \draw[thin] (-\r,\h) -- (\r,\h);
  \draw[thin] (0,0)  -- (0,\h);
  \draw[very thick] (0,0) -- (\r,\h);
  \draw[very thick] (0,0) -- (\Zx,0);
  \draw[very thick] (\Zx,0) -- (0,\h);
  \draw[thick] (\Px,\Py) -- (\Px,0);
  \draw[thin]  (0,\Py)   -- (\Px,\Py);
  \draw[->, thin] (0,\h+0.3) -- (\rt,\h+0.3);
  \node[above] at ({0.5*\r},\h+0.3) {$R$};
  \node[left]  at (0,\h/2)         {$h$};
  \node[right] at (\Px,{0.5*\Py})  {$l$};
  \fill (0,0)     circle (1.6pt);
  \node[below]       at (0,0)     {$O$};
  \fill (\Zx,0)   circle (1.6pt);
  \node[below]       at (\Zx,0)   {$z$};
  \fill (\Px,0)   circle (1.6pt);
  \node[below]       at (\Px,0)   {$f(z)$};
  \fill (\Px,\Py) circle (1.6pt);
  \node[above right] at (\Px,\Py) {$P$};
  \fill (0,\h)    circle (1.6pt);
  \fill ( \r,\h)  circle (1.6pt);
  \fill (-\r,\h)  circle (1.6pt);
\end{tikzpicture}
\caption{Side view of the construction. The right cone edge $O$--$(R,h)$ contains $P$; the construction line $z$--$(0,h)$ also contains $P$. The perpendicular from $P$ to the baseline has length $l$ and meets the baseline at $f(z)$. Two pairs of similar right triangles are visible in this diagram, formed by the perpendiculars at $f(z)$ and at $z$ together with the axis and the right cone edge; the proof of Proposition~\ref{prop:p10} uses the two resulting ratios $h/|z|=l/(|z|-|f(z)|)$ and $R/h=|f(z)|/l$ to eliminate $l$ and $h$.}
\label{fig:side}
\end{figure}

We first verify that the construction is well posed and produces a point that does not depend on the height $h$ of the cone, only on the base radius $R$. This is the \emph{height-independence theorem}.

\begin{proposition}[Height independence]\label{prop:p10}
The point $f(z)$ produced by the construction above is
\begin{equation}\label{eq:fz}
  f(z) \;=\; \frac{z}{\,1+\tfrac{1}{R}|z|\,},
\end{equation}
and is therefore independent of the height $h$ of the cone.
\end{proposition}

\begin{proof}
The case $z=0$ holds by the definition $f(0)=0$. For $z\ne 0$, the construction line $L_z$ and the cone axis meet at $(0,0,h)$, and the cone axis contains the origin. Hence the two lines together with $z$ determine a single $2$-dimensional plane $\Pi$ containing $O$, $z$, $(0,0,h)$, and the cone axis (see Figure~\ref{fig:side}). We set up coordinates in $\Pi$ such that the cone axis is the positive vertical axis and $z$ lies on the positive horizontal axis at $(|z|, 0)$.

In this planar cross-section, the apex is $(0,0)$, the center of the base is $(0,h)$, and the relevant rim point of the cone is $(R,h)$. The right cone edge forms a line segment from $(0,0)$ to $(R,h)$, while the construction line $L_z$ runs from $(|z|,0)$ to $(0,h)$.

The intersection $P$ of these two segments projects orthogonally onto the horizontal axis at $f(z)$, lying between $0$ and $|z|$. Let $l$ be the height of $P$ above the baseline. By similar triangles dropping to the baseline (for the first ratio below, the right triangles with hypotenuses along $L_z$ and legs $h,|z|$ and $l,|z|-|f(z)|$; for the second, the right triangles with hypotenuses along the cone edge through the apex and legs $R,h$ and $|f(z)|,l$), we have two immediate ratios:
\[
  \frac{h}{|z|} = \frac{l}{|z|-|f(z)|} \quad \text{and} \quad \frac{R}{h} = \frac{|f(z)|}{l}.
\]
Multiplying these equations eliminates $l$ and $h$, yielding $\frac{R}{|z|} = \frac{|f(z)|}{|z|-|f(z)|}$. Rearranging for $|f(z)|$ gives:
\begin{equation}\label{eq:mag}
  |f(z)| = \frac{R\,|z|}{R+|z|} = \frac{|z|}{1+\frac{1}{R}|z|}.
\end{equation}
Since $P$ lies in the same half-plane of $\Pi$ as $z$, $f(z)$ lies on the ray $Oz$, so multiplying by $z/|z|$ recovers the vector formula \eqref{eq:fz}, independent of $h$.
\end{proof}

The map $f_R:\mathbb{C}\to D_R$, $D_R:=\{w\in\mathbb{C}\,:\,|w|<R\}$, defined by \eqref{eq:fz}, is a homeomorphism (indeed a $C^1$-diffeomorphism, smooth on $\mathbb{C}\setminus\{0\}$ and only $C^1$ at $O$, as the regularity discussion below records), with inverse $f_R^{-1}(w)=w/(1-|w|/R)$ on $D_R$ (solve $|w|=R|z|/(R+|z|)$ for $|z|=R|w|/(R-|w|)$, then use that $f_R$ fixes each ray); we write $f_R$ to emphasize the dependence on the base radius $R$. Figure~\ref{fig:global} shows this map globally, as the compression of the whole plane into the disk $D_R$.

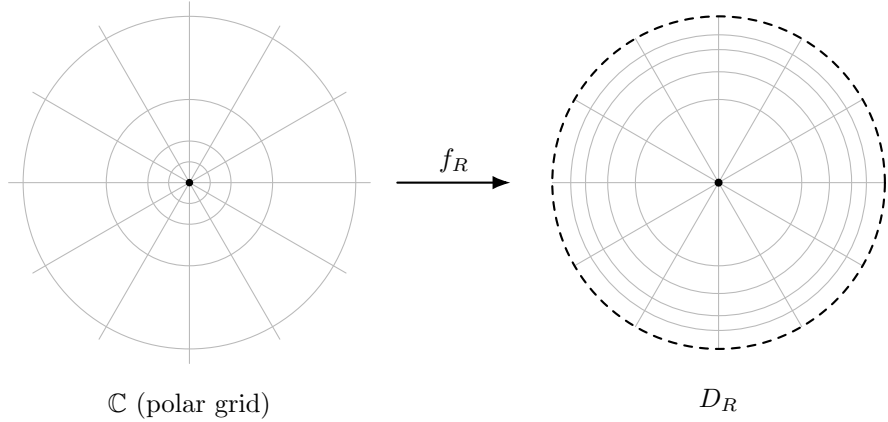
\begin{figure}[htbp]
\centering
\begin{tikzpicture}[line cap=round, line join=round, font=\small, >={Latex[length=2.4mm]}]
  \begin{scope}[scale=0.275]
    \foreach \rr in {1,2,4,8}{\draw[gray!55] (0,0) circle (\rr);}
    \foreach \a in {0,30,...,330}{\draw[gray!55] (0,0) -- (\a:8.7);}
    \fill (0,0) circle (0.18);
    \node[below] at (0,-9.6) {$\mathbb{C}$ (polar grid)};
  \end{scope}
  \draw[->, thick] (2.75,0) -- (4.25,0) node[midway, above] {$f_R$};
  \begin{scope}[xshift=7.0cm, scale=2.2]
    \foreach \s in {0.5,0.66667,0.8,0.88889}{\draw[gray!55] (0,0) circle (\s);}
    \foreach \a in {0,30,...,330}{\draw[gray!55] (0,0) -- (\a:1);}
    \draw[dashed, thick] (0,0) circle (1);
    \fill (0,0) circle (0.024);
    \node[below] at (0,-1.18) {$D_R$};
  \end{scope}
\end{tikzpicture}
\caption{The cone projection as a global map. Left: a polar grid of the plane, with circles $|z|=R,2R,4R,8R$ and rays from $O$. Right: its image under $f_R$. Every ray is fixed, and the circle $|z|=\rho$ is carried to the circle $|w|=R\rho/(R+\rho)$. The image circles crowd against the boundary $\partial D_R$ (dashed), which is the common limit of the images as $|z|\to\infty$. This is the homeomorphism $f_R:\mathbb{C}\to D_R$: the whole plane is compressed into the open disk of radius $R$, and the crowding of the circles is a visual form of the lens-equation identity $1/|f_R(z)|=1/|z|+1/R$.}
\label{fig:global}
\end{figure}

\paragraph{Regularity of $f_R$.}

Because the formula for $f_R$ involves $|z|$, the map is smooth on $\mathbb{C}\setminus\{0\}$ but not at the origin. To make this precise, we pass to real coordinates $\mathbf{x}\in\mathbb{R}^2$, in which
\[
  f_R(\mathbf{x})=\frac{\mathbf{x}}{1+|\mathbf{x}|/R}=\frac{R}{R+|\mathbf{x}|}\,\mathbf{x}.
\]
We claim that $f_R$ is globally $C^1$ but fails to be $C^2$ at $0$. Away from the origin $f_R$ is a composition of smooth maps, and differentiating gives its total derivative, the Jacobian matrix $Df_R(\mathbf{x})$ with entries $\bigl(Df_R\bigr)_{ij}=\partial (f_R)_i/\partial x_j$; using $\partial|\mathbf{x}|/\partial x_j=x_j/|\mathbf{x}|$ for $\mathbf{x}\ne0$ and the quotient rule,
\[
  Df_R(\mathbf{x})=\frac{R}{R+|\mathbf{x}|}\,I_2
  -\frac{R}{(R+|\mathbf{x}|)^2}\,\frac{\mathbf{x}\otimes \mathbf{x}}{|\mathbf{x}|}
  \qquad(\mathbf{x}\ne0),
\]
where $I_2$ is the $2\times 2$ identity matrix and $\mathbf{x}\otimes\mathbf{x}=\mathbf{x}\mathbf{x}^{\mathsf T}$ denotes the rank-one matrix with entries $x_i x_j$. To see that the second term vanishes at the origin, write $\mathbf{x}=|\mathbf{x}|\,\mathbf{u}$ with $\mathbf{u}$ a unit vector: the numerator $\mathbf{x}\otimes\mathbf{x}$ is quadratic in $\mathbf{x}$ while only one power of $|\mathbf{x}|$ is divided out, so
\[
  \frac{\mathbf{x}\otimes\mathbf{x}}{|\mathbf{x}|}=|\mathbf{x}|\,(\mathbf{u}\otimes\mathbf{u}).
\]
Since the projection $\mathbf{u}\otimes\mathbf{u}$ has operator norm $1$, the second term is bounded in norm by $R|\mathbf{x}|/(R+|\mathbf{x}|)^2\le|\mathbf{x}|/R$, a bound independent of the direction $\mathbf{u}$; it therefore tends to $0$ as $\mathbf{x}\to0$. Hence $Df_R(\mathbf{x})\to I_2$ as $\mathbf{x}\to0$. Since $f_R$ is continuous on $\mathbb{R}^2$, differentiable on $\mathbb{R}^2\setminus\{0\}$, and $Df_R$ has limit $I_2$ at the origin, the mean value theorem gives differentiability at $0$ with $Df_R(0)=I_2$; the derivative therefore extends continuously across the origin and $f_R$ is $C^1$ on all of $\mathbb{R}^2$.

It is not, however, twice differentiable at the origin. Along any unit vector $\mathbf{u}$, the scalar function
\[
  t\longmapsto \mathbf{u}\cdot f_R(t\mathbf{u})=\frac{t}{1+|t|/R}
\]
has one-sided second derivatives $-2/R$ as $t\to0^+$ and $+2/R$ as $t\to0^-$; these disagree, so $f_R$ is not $C^2$ there. The inverse map on $D_R$ has the same regularity: $f_R^{-1}(w)=w/(1-|w|/R)$ has the same radial form with $-1/R$ in place of $1/R$, so the same analysis applies on $D_R$. Accordingly, the word ``diffeomorphism'' below is to be read in the $C^1$ sense, unless smoothness away from the origin is stated explicitly.

\paragraph{Roadmap.}
The remainder of the paper is organized as follows. Section~\ref{sec:lens} states the reciprocal lens-equation identity, the algebraic engine of most of what follows. Section~\ref{sec:geometry} then gives the focus--directrix recap, proves the Self-Directrix Theorem with its one-parameter trichotomy, identifies the image arc, characterizes $f_R$ by the Self-Directrix property, and establishes the Confocal--Codirectrix Theorem that generalizes it from lines to every focal polar locus of a fixed focus--directrix pencil, closing with a chord-level rederivation of its parameter laws from a classical focal-chord harmonic-mean theorem of Askwith; these are the central geometric results of the paper. Section~\ref{sec:circle-image} then determines the image of a circle, which is in general not a conic but a circular quartic, with a trichotomy governed by the circle's position relative to $O$. Sections~\ref{sec:semigroup}--\ref{sec:iteration} develop the algebraic structure of the family: the composition law, the one-parameter partial group of radial M\"obius maps, and the flow that interpolates them. Section~\ref{sec:cross-ratio} treats the projective structure on rays, and Section~\ref{sec:alt-constructions} collects further constructions and interpretations. Sections~\ref{sec:higher-d}--\ref{sec:characterization} give the higher-dimensional extension, the pullback length metric, and a normalized characterization.

\section{The reciprocal lens identity}\label{sec:lens}

Inverting \eqref{eq:fz} produces an identity that complements the height-independence theorem.

\begin{proposition}[Lens-equation identity]\label{prop:lens}
For every $z\in\mathbb{C}\setminus\{0\}$,
\[
  \frac{1}{|f_R(z)|} \;=\; \frac{1}{|z|} \;+\; \frac{1}{R}.
\]
\end{proposition}

\begin{proof}
By \eqref{eq:mag}, $|f_R(z)|=R|z|/(R+|z|)$. Since $z\ne0$,
\[
  \frac{1}{|f_R(z)|}
  =\frac{R+|z|}{R|z|}
  =\frac{1}{|z|}+\frac{1}{R}. \qedhere
\]
\end{proof}

This identity is the same reciprocal-addition law that appears in the rule for the total resistance of two resistors in parallel, and in the formulas for capacitors in series, springs in series, and reduced mass. It is also formally analogous to a thin-lens formula after a choice of sign convention; this analogy is used below only at the level of reciprocal-distance algebra, not as a claim about the usual real-image sign convention of elementary optics. The cone parameter $R$ plays the role of the fixed reciprocal-addition parameter.

\begin{definition}[Cone curvature]
The \emph{cone curvature} associated with $f_R$ is $\kappa(f_R):=1/R$.
\end{definition}

\noindent Here ``curvature'' names only the reciprocal-radius parameter $1/R$ that governs the radial projection law; it is \emph{not} the intrinsic (Gaussian) curvature of the cone's lateral surface, which vanishes away from the apex.

In these terms, Proposition~\ref{prop:lens} reads: \emph{the cone projection $f_R$ adds the cone curvature $1/R$ to the inverse magnitude of any nonzero point.} Section~\ref{sec:group} below extends $\kappa$ to a signed parameter taking arbitrary real values, with $\kappa>0$ corresponding to cone projections, $\kappa=0$ to the identity (the cone of vanishing curvature), and $\kappa<0$ to inverse maps on the disk of radius $1/|\kappa|$; on this extended range, composition adds $\kappa$ exactly.

\section{Conic geometry and the Self-Directrix Theorem}\label{sec:geometry}

\subsection{Classical focus--directrix recap}\label{sec:conic-recap}

We turn directly to the Euclidean geometry of $f_R$, the central result of the paper. We first recall the two facts the argument rests on.

First, by Proposition~\ref{prop:p10} the cone projection is radial: it fixes the direction of every nonzero $z$ and rescales the modulus by the positive factor $R/(R+|z|)\in(0,1)$, so that $f_R(z)=\frac{R}{R+|z|}\,z$, and, by \eqref{eq:mag},
\begin{equation}\label{eq:recap-mag}
  |f_R(z)|=\frac{R\,|z|}{R+|z|}.
\end{equation}

Second, given a point $F$ (the \emph{focus}), a line $\ell_0$ not through $F$ (the \emph{directrix}), and a number $e>0$ (the \emph{eccentricity}), the associated \emph{conic} is the locus
\[
  \mathcal{C}(F,\ell_0,e)=\bigl\{\,P:\ |PF|=e\cdot\operatorname{dist}(P,\ell_0)\,\bigr\},
\]
an ellipse, parabola, or hyperbola according as $e<1$, $e=1$, or $e>1$. This is the standard focus--directrix description of the conic sections; see, for example, \cite{akopyan2007,brannan2012}. In polar coordinates centred at the focus $F$, with $\alpha$ the direction of the axis (the ray from $F$ perpendicular to $\ell_0$, pointing toward it) and $\delta=\operatorname{dist}(F,\ell_0)$, the locus takes the familiar form
\begin{equation}\label{eq:recap-polar}
  \rho(\theta)=\frac{\lambda}{1+e\cos(\theta-\alpha)},
  \qquad \lambda=e\,\delta,
\end{equation}
valid over the directions with $1+e\cos(\theta-\alpha)>0$ (all of them for an ellipse; for $e>1$ this parametrizes the focus-side branch of the hyperbola). Here $\lambda$ is the \emph{semi-latus rectum}, the focal radius in the two directions $\theta=\alpha\pm\tfrac{\pi}{2}$ perpendicular to the axis. Equation~\eqref{eq:recap-polar} is derived in the present setting in \eqref{eq:polar-conic} below and is the form the argument of Section~\ref{sec:confocal} works with.

\subsection{The Self-Directrix Theorem}\label{sec:directrix}

Because $f_R$ depends on $|z|$ it is not holomorphic; nor is it conformal, since away from the origin its radial and tangential stretch factors, $R^2/(R+\rho)^2$ and $R/(R+\rho)$, are unequal for every $\rho>0$, coinciding (both tending to $1$) only in the limit $\rho\to0$, where the conformal distortion vanishes and $f_R$ is tangent to the identity at the origin. \emph{A priori} it therefore has no reason to send any distinguished family of curves to another. It nevertheless has a precise relation to one classical family.

\begin{theorem}[Self-Directrix Theorem]\label{thm:directrix}
Let $\ell$ be a line that does not pass through the origin $O$, and let $d=\operatorname{dist}(O,\ell)>0$. Then $f_R$ carries $\ell$ onto the open arc, between the two latus-rectum endpoints (the subarc through the near vertex, identified precisely in Proposition~\ref{prop:image-arc}), of the conic
\[
  \Gamma=\mathcal{C}\!\left(O,\ \ell,\ \tfrac{R}{d}\right):
\]
the origin $O$ is a focus of $\Gamma$, the line $\ell$ itself is the directrix of $\Gamma$ belonging to that focus, and the eccentricity is $e=R/d$. The semi-latus rectum of $\Gamma$ equals $R$, the same value for every line. Consequently $\Gamma$ is
\[
  \text{an ellipse if } d>R,\qquad
  \text{a parabola if } d=R,\qquad
  \text{a hyperbola if } d<R.
\]
Equivalently, writing $\Gamma_+$ for the focus-side branch of $\Gamma$ (all of $\Gamma$ when $d\ge R$, since an ellipse or parabola lies entirely on the focus side of its directrix), the image is exactly
\[
  f_R(\ell)=\Gamma_+\cap D_R,
\]
the open arc cut from $\Gamma_+$ by the disk $D_R$ (its two endpoints, the latus-rectum points, lie on $\partial D_R$ and are not attained).
\end{theorem}

The line one starts with reappears, unchanged, as the directrix of the curve one ends with. This is the sense in which the result is ``self-directrix'': the same geometric object has two roles. At the beginning it is the input line whose points are being projected; at the end it is the directrix that defines the image conic. Throughout, the image lies on the \emph{focus side} of the directrix: each image point $w=f_R(z)$ lies strictly between $O$ and $z$ on the ray $Oz$, hence in the open half-plane bounded by $\ell$ that contains $O$. This fixes the unsigned distance formula $\operatorname{dist}(w,\ell)=d-x_w$ used in the proof and, in the hyperbolic case $d<R$, places the image on the focus-side branch. Indeed, for $d<R$ the full focus--directrix locus has two branches, and $f_R(\ell)$ reaches only the central portion of the focus-side branch (Proposition~\ref{prop:image-arc}). The proof below establishes the focus--directrix equation and the classification; Proposition~\ref{prop:image-arc} in Section~\ref{sec:image-arc} then identifies exactly which arc of $\Gamma$ is swept out. Figure~\ref{fig:directrix} shows the pointwise mechanism in the elliptic case: several points $z_j\in\ell$ move inward along their rays to $w_j=f_R(z_j)$, and a representative point $w$ displays the focus--directrix relation.

\begin{figure}[!ht]
\centering
\begin{tikzpicture}[
  >=Latex,
  scale=1.35,
  every node/.style={font=\small},
  ray/.style={thin,gray!60},
  maparrow/.style={-{Latex[length=2.4mm]},thick,blue!65!black},
  directrix/.style={thick,gray!75,dashed},
  imagearc/.style={very thick,blue!70!black},
  helper/.style={thin,densely dashed,gray!65},
  zright/.style={right},
  zbelowright/.style={below right},
  wbelow/.style={below left},
  wabove/.style={above left}
]

\pgfmathsetmacro{\R}{2.0}
\pgfmathsetmacro{\d}{3.25}
\pgfmathsetmacro{\Ymax}{2.75}
\pgfmathsetmacro{\ecc}{\R/\d}

\draw[->,thin,gray!70] (-0.35,0) -- (4.15,0) node[below left] {$x$};
\draw[->,thin,gray!70] (0,-3.20) -- (0,3.05) node[left] {$y$};
\draw[thick,densely dashed,gray!90] (0,0) circle[radius=\R];
\node[gray!100] at (-0.58,1.55) {$\partial D_R$};

\fill (0,0) circle (1.45pt);
\node[below left] at (0,0) {$O$};
\node[anchor=north west,xshift=-25pt,yshift=-7pt] at (0,0) {focus};

\draw[directrix] (\d,-\Ymax) -- (\d,\Ymax);
\node[above,align=center] at (\d,\Ymax)
  {$\ell=\{x=d\}$\\[-1pt]$\ell$ = directrix};

\draw[imagearc,domain=-88:88,samples=220,smooth,variable=\t]
  plot ({(\R/(1+\ecc*cos(\t)))*cos(\t)},
        {(\R/(1+\ecc*cos(\t)))*sin(\t)});
\node[blue!70!black,align=center] at (0.95,-1.58)
  {$f_R(\ell)$\\[-1pt]conic arc};

\draw[blue!70!black,fill=white,thick] (0,\R) circle (2.0pt);
\draw[blue!70!black,fill=white,thick] (0,-\R) circle (2.0pt);
\node[left=2pt,yshift=9pt] at (0,\R) {$iR$};
\node[left=2pt,yshift=-9pt] at (0,-\R) {$-iR$};

\foreach \ty/\lab/\zpos/\wpos in {-2.20/1/zright/wbelow,-0.75/2/zbelowright/wbelow,1.25/3/zright/wabove,2.35/4/zright/wabove}{
  \pgfmathsetmacro{\s}{sqrt(\d*\d+\ty*\ty)}
  \pgfmathsetmacro{\xw}{\R*\d/(\R+\s)}
  \pgfmathsetmacro{\yw}{\R*\ty/(\R+\s)}

  \draw[ray] (0,0) -- (\d,\ty);
  \draw[maparrow] (\d,\ty) -- (\xw,\yw);

  \fill[black] (\d,\ty) circle (1.4pt);
  \fill[blue!70!black] (\xw,\yw) circle (1.6pt);

  \node[\zpos] at (\d,\ty) {$z_{\lab}$};
  \ifnum\lab=3\relax
  \else
    \node[\wpos,blue!70!black] at (\xw,\yw) {$w_{\lab}$};
  \fi
}

\pgfmathsetmacro{\tstar}{1.25}
\pgfmathsetmacro{\sstar}{sqrt(\d*\d+\tstar*\tstar)}
\pgfmathsetmacro{\xstar}{\R*\d/(\R+\sstar)}
\pgfmathsetmacro{\ystar}{\R*\tstar/(\R+\sstar)}

\draw[red!70!black,thick] (0,0) -- (\xstar,\ystar)
  node[midway,below,xshift=11pt,yshift=4pt] {$|Ow|$};
\draw[red!70!black,densely dashed,thick] (\xstar,\ystar) -- (\d,\ystar)
  node[midway, above, xshift=10pt, yshift=-2pt] {$\operatorname{dist}(w,\ell)$};
\draw[red!70!black,thin] (\d,\ystar) ++(-0.10,0) -- ++(0,-0.10) -- ++(0.10,0);
\node[blue!70!black,anchor=south east,xshift=-6pt,yshift=5pt,fill=white,inner sep=1pt] at (\xstar,\ystar) {$w=w_3$};
\draw[<->,thin] (0,-2.92) -- (\d,-2.92)
  node[midway,fill=white,inner sep=1pt] {$d$};
\node[align=center] at (3.66,-3.02) {$d>R$\\[-1pt]elliptic case};

\end{tikzpicture}
\caption{The pointwise mechanism of the Self-Directrix Theorem in the elliptic case ($d>R$), drawn in normalized coordinates $\ell=\{x=d\}$. Points $z_j\in\ell$ move inward along their rays from the focus $O$ to $w_j=f_R(z_j)$ on the conic arc $f_R(\ell)$. For the representative point $w=w_3$, the red segments show the focus distance $|Ow|$ and the directrix distance $\operatorname{dist}(w,\ell)$, illustrating the relation $|Ow|=(R/d)\operatorname{dist}(w,\ell)$. The open arc has limiting endpoints $\pm iR$, which are not attained.}
\label{fig:directrix}
\end{figure}
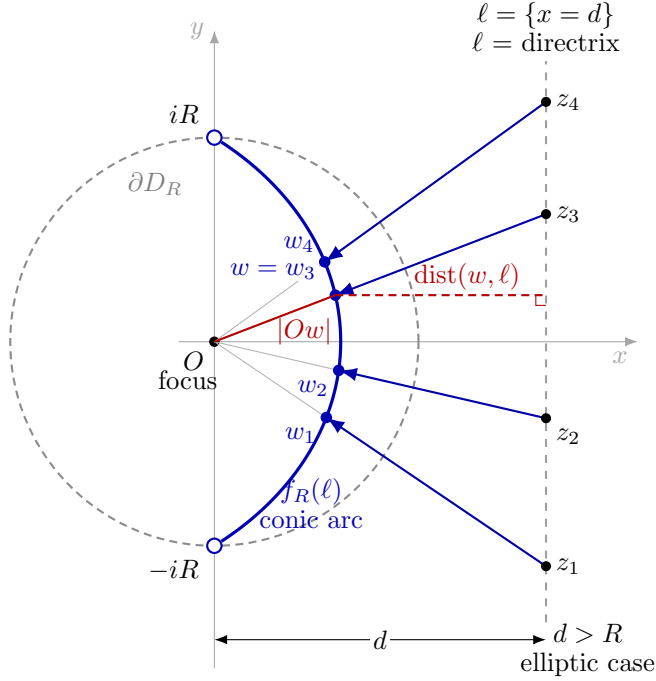
\begin{proof}
Because $f_R$ commutes with every rotation about $O$ (Proposition~\ref{prop:p10}), we may rotate coordinates so that $\ell=\{x=d\}$ with $d>0$. This costs no generality.

Let $w=f_R(z)$ for some $z\in\ell$, write $w=(x_w,y_w)$, and set $r:=|w|$. Recall from Section~\ref{sec:setup} that $f_R$ is a homeomorphism $\mathbb{C}\to D_R$ with inverse $f_R^{-1}(w)=w/(1-|w|/R)$ defined on all of $D_R$; since $r<R$, this inverse applies, and $z=f_R^{-1}(w)=(x_w,y_w)/(1-r/R)$. The condition $z\in\ell$ reads
\[
  \frac{x_w}{1-r/R} \;=\; d, \qquad\text{i.e.,}\qquad x_w \;=\; d\!\left(1-\frac{r}{R}\right).
\]
Because $r/R\in(0,1)$, we have $x_w\in(0,d)$, so $w$ lies strictly on the origin's side of $\ell$. The distance from $w$ to $\ell$ is therefore
\[
  \operatorname{dist}(w,\ell) \;=\; d-x_w \;=\; \frac{d}{R}\,r,
\]
which rearranges to the focus--directrix relation
\begin{equation}\label{eq:focus-directrix}
  |Ow| \;=\; \frac{R}{d}\,\operatorname{dist}(w,\ell).
\end{equation}
This identifies $f_R(\ell)$ as a subset of the conic $\mathcal{C}(O,\ell,R/d)$ with focus $O$, directrix $\ell$ itself, and eccentricity $e=R/d$. The conic classification follows: an ellipse for $d>R$ (where $e<1$), a parabola for $d=R$ ($e=1$), and a hyperbola for $d<R$ ($e>1$). For the full conic $\Gamma=\mathcal{C}(O,\ell,R/d)$, the latus-rectum chord associated with the focus $O$ is the line through $O$ perpendicular to the axis, namely $x=0$. Setting $x=0$ in the focus--directrix relation gives $\operatorname{dist}(w,\ell)=d$, hence $|Ow|=R$; so $\Gamma$ has semi-latus rectum $R$. The image arc $f_R(\ell)$ has $x_w\in(0,d)$, so it approaches but does not attain the two latus-rectum endpoints.

Finally, the polar form of \eqref{eq:focus-directrix} is the standard focal polar equation of $\Gamma$ with respect to the focus $O$; in the hyperbolic case, it is the equation of the focus-side branch. Writing $w=re^{i\theta}$, $x_w=r\cos\theta$ gives $\operatorname{dist}(w,\ell)=d-r\cos\theta$, and \eqref{eq:focus-directrix} yields
\begin{equation}\label{eq:polar-conic}
  r \;=\; \frac{R}{1+(R/d)\cos\theta}, \qquad \theta\in\bigl(-\tfrac{\pi}{2},\tfrac{\pi}{2}\bigr),
\end{equation}
the range of $\theta$ coming from the constraint $x_w>0$ derived above; this traces the part of $\Gamma$ with $x_w>0$, the open central arc between the two latus-rectum endpoints. Conversely, every point $w=re^{i\theta}$ of that arc lies in $f_R(\ell)$: since $1-r/R=\frac{(R/d)\cos\theta}{1+(R/d)\cos\theta}$, the inverse image $f_R^{-1}(w)=w/(1-r/R)$ has first coordinate $\frac{r\cos\theta}{1-r/R}=d$, so $f_R^{-1}(w)\in\ell$. Hence $f_R$ carries $\ell$ exactly onto this open arc, whose endpoints are identified in Proposition~\ref{prop:image-arc}. In the hyperbolic case $d<R$, the focus-side branch continues beyond $\theta=\pm\tfrac{\pi}{2}$ out to its asymptotic directions $\theta=\pm\arccos(-d/R)$, where $x_w<0$; that continuation, together with the far branch on the other side of $\ell$ (where $\operatorname{dist}(w,\ell)=r\cos\theta-d$ and the polar equation is $r=R/((R/d)\cos\theta-1)$), is not reached by $f_R(\ell)$ \cite{akopyan2007}.
\end{proof}

The theorem organizes the entire Euclidean menagerie of conics along a single real parameter. Let one line drift from infinity inward toward the origin and watch its image (Figure~\ref{fig:trichotomy}).

\begin{corollary}[The one-parameter trichotomy]\label{cor:sweep}
As $d=\operatorname{dist}(O,\ell)$ decreases from $+\infty$ to $0$, the associated conic $\Gamma$ runs through every type of conic, and $f_R(\ell)$ is the corresponding open arc:
\[
\begin{array}{rcl}
  d\to\infty &:& e=R/d\to0;\ \text{for a fixed line direction, the image arc tends to the}\\
              && \text{corresponding open semicircular arc approaching }\partial D_R;\\[2pt]
  d>R &:& 0<e<1,\ \text{a family of elliptic arcs};\\[2pt]
  d=R &:& e=1,\ \text{a parabolic arc};\\[2pt]
  d<R &:& e>1,\ \text{a family of hyperbolic arcs};\\[2pt]
  d\to0^{+} &:& e\to\infty,\ \text{the conic degenerates; a line through }O
                \text{ maps to an open diameter of }D_R.
\end{array}
\]
Throughout the sweep two data never change: the focus stays fixed at $O$, and the semi-latus rectum stays fixed at $R$.
\end{corollary}

The two endpoints are limiting degenerations in this fixed-$R$ scaling rather than ordinary members of the theorem: as $d\to\infty$ the eccentricity $e=R/d\to0$, whose limit is a circle rather than a focus--directrix conic with a finite directrix distance; and as $d\to0^+$ the directrix passes through the focus $O$, which the hypothesis $d>0$ excludes.

\begin{remark}[The $d\to\infty$ endpoint as the classical $e=0$ conic]
The endpoint $d\to\infty$ also has a natural classical interpretation. In the
limiting focus--directrix convention, a circle is the conic of eccentricity
$e=0$: its centre is the focus and its directrix is the line at infinity, the
eccentricity being read as the limiting ratio of a finite radius to an infinite
focus--directrix distance \cite[\S94, p.~95]{askwith1903}. The present sweep is
consistent with that convention. As $d\to\infty$, the directrix $\ell$ recedes to
the line at infinity, while
\[
  e=\frac{R}{d}\longrightarrow 0.
\]
At the same time, the carrier conic of the Self-Directrix arc tends to the circle
$\partial D_R$, since its semi-latus rectum remains $R$ while its eccentricity
tends to $0$.

There is one useful distinction to draw. For a fixed family of parallel affine
lines, the actual forward image arc tends only to the corresponding open
semicircle of $\partial D_R$; this is the limiting version of the angular range
$\theta\in(-\tfrac{\pi}{2},\tfrac{\pi}{2})$. The full circle appears either as the
limiting carrier conic or by passing to the radial compactification, where $f_R$
extends continuously by sending each ideal direction to the corresponding point
of $\partial D_R$, since
\[
  |f_R(z)|=\frac{R|z|}{R+|z|}\longrightarrow R \qquad (|z|\to\infty).
\]
Thus the $d\to\infty$ end is not a break in the conic family but its classical
$e=0$ member: the circle of radius $R$ centred at $O$. In this limiting sense the
two invariants of the sweep persist, namely the focus $O$ and the semi-latus
rectum $R$; and the self-directrix pattern survives at the projective boundary,
with the limiting input line now interpreted as the line at infinity.
\end{remark}

\begin{remark}[The universal parabolic constant at the transition]
The critical line $d=R$ is distinguished analytically, not only
projectively. By the Self-Directrix Theorem its image is the latus-rectum
arc of a parabola of semi-latus rectum $R$: the arc
$\theta\in(-\tfrac{\pi}{2},\tfrac{\pi}{2})$ of the polar form
\eqref{eq:polar-conic} with $e=1$, where
$r=R/(1+\cos\theta)=\tfrac{R}{2}\sec^{2}\tfrac{\theta}{2}$. Differentiating gives
$r'=R\sin\theta/(1+\cos\theta)^{2}$, so $r^{2}+r'^{2}=2R^{2}/(1+\cos\theta)^{3}$;
with $1+\cos\theta=2\cos^{2}\tfrac{\theta}{2}$ the polar arc-length integrand
becomes $\sqrt{r^{2}+r'^{2}}=\tfrac{R}{2}\sec^{3}\tfrac{\theta}{2}$, and the
length of the arc evaluates explicitly,
\[
  \int_{-\pi/2}^{\pi/2}\!\sqrt{r^{2}+r'^{2}}\,d\theta
  =\int_{-\pi/2}^{\pi/2}\!\frac{R}{2}\,\sec^{3}\!\tfrac{\theta}{2}\,d\theta
  =\bigl(\sqrt{2}+\ln(1+\sqrt{2})\bigr)R
  =P\,R\approx 2.29559\,R,
\]
where $P=\sqrt{2}+\ln(1+\sqrt{2})$ is the \emph{universal parabolic
constant}, the parabola's scale-free analogue of $\pi$ \cite{sondow2013parbelos,reeseSondow}.
Since $P$ is defined as the ratio of a parabola's latus-rectum arc to its
semi-latus rectum, the Self-Directrix Theorem reproduces that defining
configuration exactly, with semi-latus rectum $R$.

For finite $d$, this parabolic transition is the distinguished closed-form
case: away from it the image arc is elliptic or hyperbolic and its length is
naturally expressed as an elliptic integral depending on the eccentricity
$e=R/d$. In this latus-rectum normalization the two scale-free constants that
appear are the circular constant $\pi$ and the universal parabolic constant
$P$, at exactly the two distinguished lines of the trichotomy of
Corollary~\ref{cor:sweep}: the degenerate endpoint $d\to\infty$, whose image
is the semicircle of $\partial D_{R}$ of length $\pi R$ (cf.\ the $e=0$ reading
above), and the transition $d=R$, of length $PR$. Numerical evaluation gives arc lengths increasing from the limiting value $2R$
(the open diameter) as $d\to0^{+}$, passing through $PR$ at $d=R$, and tending
to $\pi R$ as $d\to\infty$; for instance, $d=1.5R$, $2R$, and $3R$ give
$\approx 2.430R$, $2.532R$, and $2.671R$, respectively.
\end{remark}

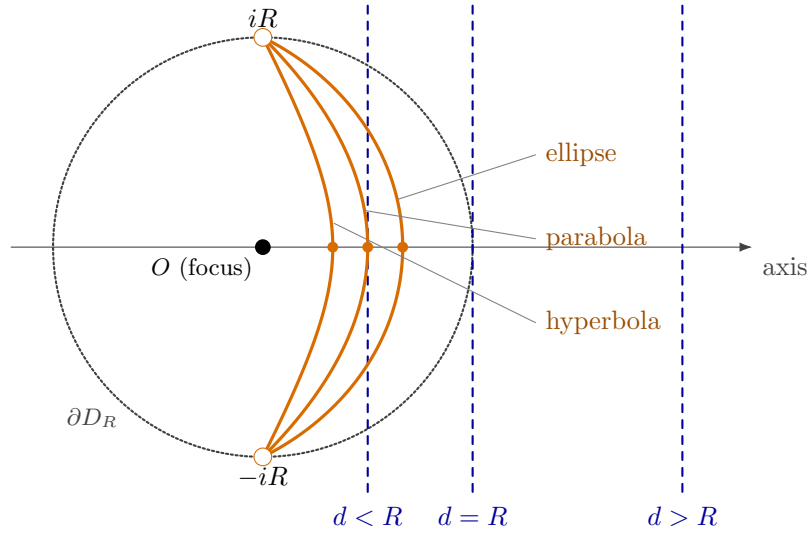
\begin{figure}[htbp]
\centering
\begin{tikzpicture}[scale=1.85, font=\small, line cap=round, line join=round, >=Latex,
   conic/.style={very thick}]
  \def\R{1.5}
  \draw[densely dotted, thick, black!75] (0,0) circle (\R);
  \node[black!75, font=\footnotesize] at (-1.22,-1.22) {$\partial D_R$};
  \draw[->, darkgray] (-1.8,0) -- (3.5,0) node[below right]{axis};
  \draw[thick, dashed, blue!60!black] (0.75,-1.75) -- (0.75,1.75);
  \draw[thick, dashed, blue!60!black] (1.50,-1.75) -- (1.50,1.75);
  \draw[thick, dashed, blue!60!black] (3.00,-1.75) -- (3.00,1.75);
  \node[below, align=center, blue!60!black] at (0.75,-1.78) {$d<R$};
  \node[below, align=center, blue!60!black] at (1.50,-1.78) {$d=R$};
  \node[below, align=center, blue!60!black] at (3.00,-1.78) {$d>R$};
  \draw[conic, orange!85!black]
    plot[domain=-89:89, samples=180]
      ({\R/(1+2*cos(\x))*cos(\x)}, {\R/(1+2*cos(\x))*sin(\x)});
  \draw[conic, orange!85!black]
    plot[domain=-89:89, samples=180]
      ({\R/(1+cos(\x))*cos(\x)}, {\R/(1+cos(\x))*sin(\x)});
  \draw[conic, orange!85!black]
    plot[domain=-89:89, samples=180]
      ({\R/(1+0.5*cos(\x))*cos(\x)}, {\R/(1+0.5*cos(\x))*sin(\x)});
  \fill (0,0) circle (1.6pt);
  \node[below left, font=\footnotesize] at (0,0) {$O$ (focus)};
  \fill[orange!85!black] (0.5,0)  circle (1.1pt);
  \fill[orange!85!black] (0.75,0) circle (1.1pt);
  \fill[orange!85!black] (1.0,0)  circle (1.1pt);
  \draw[orange!85!black, fill=white] (0,\R)  circle (1.8pt);
  \node[above] at (0,\R)  {$iR$};
  \draw[orange!85!black, fill=white] (0,-\R) circle (1.8pt);
  \node[below] at (0,-\R) {$-iR$};
  \draw[thin, gray] (0.96,0.34) -- (1.95,0.66);
  \node[right, orange!60!black] at (1.95,0.66) {ellipse};
  \draw[thin, gray] (0.74,0.26) -- (1.95,0.06);
  \node[right, orange!60!black] at (1.95,0.06) {parabola};
  \draw[thin, gray] (0.51,0.17) -- (1.95,-0.54);
  \node[right, orange!60!black] at (1.95,-0.54) {hyperbola};
\end{tikzpicture}
\caption{The one-parameter trichotomy (Corollary~\ref{cor:sweep}), drawn for a common $R$ and a fixed line direction. Three parallel lines (dashed; by the Self-Directrix Theorem each is also the directrix of its own image) at distances $d<R$, $d=R$, $d>R$ from $O$ produce a hyperbola, a parabola, and an ellipse. All three image arcs share the focus $O$, share the semi-latus rectum $R$, and have the same limiting endpoints $\pm iR$ (open dots, on the boundary circle $\partial D_R$, dotted) in their closures in the coordinates shown; for a different line direction these endpoints rotate with the direction of the line. They differ only in eccentricity $e=R/d$.}
\label{fig:trichotomy}
\end{figure}

Geometrically, for a fixed line direction, the fixed semi-latus rectum has the following consequence: every full image conic passes through the same two latus-rectum endpoints. In the coordinates shown these are $\pm iR$; in general they are $\pm Ru$, where $u$ is a unit vector parallel to the line. The open image arcs are anchored at that pair in their closures, fanning open as the directrix recedes.

\subsection{Boundary analysis of the image arc}\label{sec:image-arc}

Theorem~\ref{thm:directrix} placed the image on the conic $\Gamma$ and classified it. We now pin down which arc of $\Gamma$ is swept out.

\begin{proposition}[The image arc]\label{prop:image-arc}
Let $u$ be a unit vector parallel to $\ell$. Then $f_R(\ell)$ is the open arc of $\Gamma$ whose limiting endpoints in $\partial D_R$ are $\pm Ru$; these endpoints are not attained. It is the open arc between the two latus-rectum endpoints, lying on the side of the latus-rectum line that contains the nearer vertex; equivalently, the side facing the directrix $\ell$. In the hyperbolic case $d<R$, where $\Gamma$ is the full two-branch focus--directrix hyperbola, $f_R(\ell)$ lies on the focus-side branch, the branch whose vertex lies between $O$ and the directrix $\ell$.
\end{proposition}

\begin{proof}
Work in the normalized coordinates of the proof of Theorem~\ref{thm:directrix}, where $\ell=\{x=d\}$. As $z$ runs along $\ell$, the angle $\theta=\arg z$ sweeps the open interval $(-\tfrac{\pi}{2},\tfrac{\pi}{2})$ once, monotonically. At the two ends $|z|\to\infty$, so by \eqref{eq:recap-mag} $|w|\to R$ while $\arg w\to\pm\tfrac{\pi}{2}$; hence $f_R(\ell)$ is the open conic arc with limiting endpoints $\pm iR$. In coordinate-free form these endpoints are
\[
  \pm R u \;=\; \partial D_R\ \cap\ \{\text{the diameter of } D_R \text{ parallel to }\ell\},
\]
and they are not attained, since $f_R(\mathbb{C})$ is the open disk $D_R$. They are the endpoints of the latus rectum through $O$ (the focal chord perpendicular to the axis) because along the directions $\theta=\pm\tfrac{\pi}{2}$ the polar radius $|w|$ equals the semi-latus rectum $R$. Finally, $f_R\colon\ell\to f_R(\ell)$ is a homeomorphism, so $\ell$ is swept onto this arc bijectively and without backtracking.
\end{proof}

\begin{remark}[Completing the arc to the full conic]
The image is only an arc (Proposition~\ref{prop:image-arc}) because
\[
  f_R(\ell)=\Gamma_+\cap D_R ,
\]
where $\Gamma_+$ is the branch of the conic $\Gamma$ of Theorem~\ref{thm:directrix} reached by $f_R$: the focus-side branch, which is all of
$\Gamma$ when $e\le1$ and, for the hyperbola $e>1$, the branch whose vertex lies between $O$ and
$\ell$. This branch meets $\partial D_R$ exactly at the two latus-rectum endpoints $\pm Ru$ (the
$|z|\to\infty$ limits along $\ell$): the near arc lies inside $D_R$, these endpoints lie on
$\partial D_R$, and the rest of $\Gamma_+$ lies outside $\overline{D_R}$.

This distinction affects the equality with $\Gamma\cap D_R$ only for strongly hyperbolic lines ($e>2$).
In the normalized coordinates $\ell=\{x=d\}$, the intersections of the full algebraic conic $\Gamma$ with $\partial D_R$ occur at $x=0$, giving
$\pm Ru$, and at $x=2d$, giving $y^2=R^2-4d^2$, which is real precisely when $e=R/d\ge2$; these
extra points lie on the far branch. Thus for $e>2$ the far branch dips inside $D_R$ (its vertex is
at distance $R/(e-1)<R$), so $\Gamma\cap D_R$ would be strictly larger than $f_R(\ell)$, while
$\Gamma_+\cap D_R$ is always exactly the arc.

Since the range of $f_R$ is all of $D_R$ and nothing more, the outer part of $\Gamma_+$ is never a
forward image, no matter how far one goes along the line; the obstruction is the bounded range, not
the domain. The full conic $\Gamma$ is recovered by the first
two constructions below; the third explains why a line yields only an arc.
\begin{enumerate}[label=\textnormal{(\roman*)}]
\item\textbf{Algebraically.} A nondegenerate conic is determined by any arc of positive length:
two distinct conics meet in at most four points. The full conic $\Gamma$ is therefore already fixed
as the unique conic containing the arc $f_R(\ell)$; for a hyperbola this is the full two-branch
conic, whose second branch is supplied explicitly by the projective continuation below.

\item\textbf{Projectively.} Write $\ell$ in signed focal polar form, $1/\rho=\cos\theta/d$, valid
for every $\theta\neq\pm\tfrac{\pi}{2}$ with $\rho$ the signed radius. The cone law of Proposition~\ref{prop:lens} is exactly
$1/r=1/\rho+1/R$, so applying it over the whole circle of directions yields
\[
  \frac{1}{r}=\frac{1}{R}+\frac{\cos\theta}{d},
\]
the complete conic. Equivalently, the radial action is the one-dimensional M\"obius map $M_R$ of
Section~\ref{sec:cross-ratio},
\[
  M_R=\begin{pmatrix}R&0\\ 1&R\end{pmatrix};
\]
restricted, for each direction, to the genuine forward ray $\rho\in[0,\infty]$ it produces the arc
(the directions in which $\ell$ lies in front of $O$),
while over the full projective line $\rho\in\mathbb{RP}^1$ it traces all of $\Gamma$. The endpoints
$\pm Ru$ correspond to $\rho=\infty\mapsto R$, the part outside $D_R$ to the negative-radius
branch, and the pole $\rho=-R$ (where $\cos\theta=-d/R$) to the asymptotic direction of the
hyperbolic case.

\item\textbf{Structurally (Confocal--Codirectrix).} The line is the degenerate $e\to\infty$
member of the focus--directrix pencil to which $\Gamma$ belongs, all of whose members share the
focus $O$ and the directrix $\ell$. The Confocal--Codirectrix Theorem (Theorem~\ref{thm:confocal}) carries the entire admissible
focal-polar locus of one member to that of another member of the same pencil, with the same focus
and directrix; in the pencil coordinate this is again the lens shift
$1/\lambda\mapsto1/\lambda+1/R$ on the semi-latus rectum, sending the line $(\lambda=\infty)$ to
$\Gamma$ $(\lambda=R)$. Unlike the nondegenerate members of the pencil, the line is seen from $O$
only over the half-circle of forward directions $\theta\in(-\tfrac{\pi}{2},\tfrac{\pi}{2})$, which is
why its forward image is an arc rather than the full locus.
\end{enumerate}
Only the elliptic case $d>R$ has a bounded completion; for $d=R$ the far point recedes to infinity
(parabola), and for $d<R$ the continuation crosses to the second branch through the asymptotic
directions (hyperbola). Figure~\ref{fig:arc-completion} shows the elliptic case.
\end{remark}

\begin{figure}[htbp]
\centering
\begin{tikzpicture}[>=Latex, scale=1.0, line cap=round, line join=round,
   every node/.style={font=\small},
   arc/.style={very thick, blue!70!black},
   cont/.style={thick, orange!85!black, densely dashed},
   ray/.style={thin, gray!55},
   maparrow/.style={-{Latex[length=2mm]}, thin, red!65!black},
   directrix/.style={thick, gray!70, dashed}]
  \def\R{2}\def\ecc{0.5}\def\dd{4}
  \draw[->,thin,gray!65] (-4.7,0) -- (4.9,0) node[below right] {axis};
  \draw[->,thin,gray!65] (0,-2.8) -- (0,3.0) node[left] {$y$};
  \draw[thick,densely dashed,gray!80] (0,0) circle[radius=\R];
  \node[gray!75] at (-1.5,-1.5) {$\partial D_R$};
  \draw[directrix] (\dd,-2.6) -- (\dd,2.8);
  \node[above, gray!55!black] at (\dd,2.7) {$\ell=\{x=d\}$};
  \draw[cont, domain=90:270, samples=240, smooth, variable=\t]
     plot ({(\R/(1+\ecc*cos(\t)))*cos(\t)}, {(\R/(1+\ecc*cos(\t)))*sin(\t)});
  \node[orange!85!black, anchor=south] at (-2.2,2.2) {$\Gamma$ (full conic)};
  \draw[arc, domain=-90:90, samples=240, smooth, variable=\t]
     plot ({(\R/(1+\ecc*cos(\t)))*cos(\t)}, {(\R/(1+\ecc*cos(\t)))*sin(\t)});
  \node[blue!70!black, anchor=west] at (1.18,-0.85) {$f_R(\ell)=\Gamma_+\cap D_R$};
  \foreach \zy/\lab in {2.0/z, -1.2/{}}{
    \pgfmathsetmacro{\zmag}{sqrt(\dd*\dd+\zy*\zy)}
    \pgfmathsetmacro{\wx}{\R*\dd/(\R+\zmag)}
    \pgfmathsetmacro{\wy}{\R*\zy/(\R+\zmag)}
    \draw[ray] (0,0) -- (\dd,\zy);
    \draw[maparrow] (\dd,\zy) -- (\wx,\wy);
    \fill[black] (\dd,\zy) circle (1.2pt);
    \fill[blue!70!black] (\wx,\wy) circle (1.5pt);
  }
  \node[anchor=west] at (\dd,2.0) {$z\in\ell$};
  \pgfmathsetmacro{\wxa}{\R*\dd/(\R+sqrt(\dd*\dd+4.0))}
  \pgfmathsetmacro{\wya}{\R*2.0/(\R+sqrt(\dd*\dd+4.0))}
  \node[blue!70!black, anchor=south west] at (\wxa,\wya) {$w=f_R(z)$};
  \fill (0,0) circle (1.5pt);
  \node[below left] at (0,0) {$O$ (focus)};
  \draw[blue!70!black, fill=white, thick] (0,\R) circle (2pt);
  \draw[blue!70!black, fill=white, thick] (0,-\R) circle (2pt);
  \node[anchor=south west] at (0.06,\R) {$Ru=iR$};
  \node[anchor=north west] at (0.06,-\R) {$-Ru=-iR$};
\end{tikzpicture}
\caption{Completing the Self-Directrix arc, elliptic case $e=R/d=\tfrac12$ ($R=2$, $d=4$). The map $f_R$
pulls each point of the line $\ell=\{x=d\}$ inward along its ray from $O$ (two samples shown),
giving the bold arc $f_R(\ell)=\Gamma_+\cap D_R$. In this elliptic case $\Gamma_+=\Gamma$, and the
dashed continuation outside $D_R$ is the unique full ellipse determined by the image arc;
equivalently, it is traced by the projective radial action as the signed radial coordinate ranges
over $\mathbb{RP}^1$.}
\label{fig:arc-completion}
\end{figure}
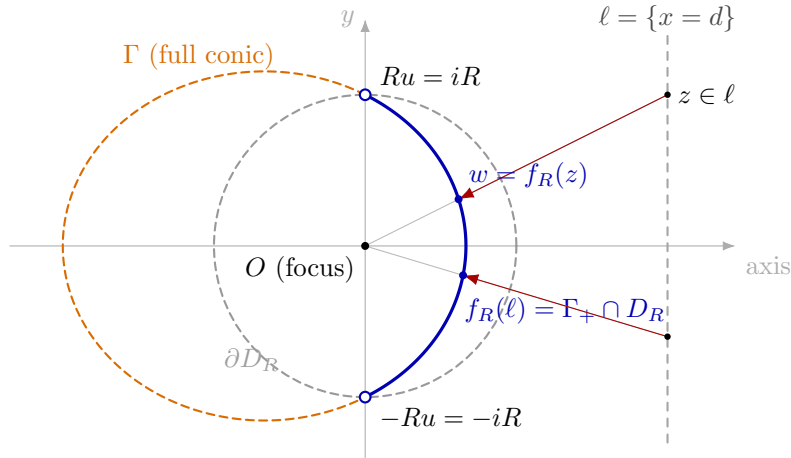

\subsection{A characterization by the Self-Directrix property}\label{sec:sd-char}

The Self-Directrix property is not incidental to $f_R$; it characterizes it. Among all radial maps
fixing $O$, the cone projections are exactly those whose image of every line lies on a focus--directrix
conic with the line as its own directrix.

\begin{theorem}[Self-Directrix characterization]\label{thm:sd-char}
Let $\Phi:\mathbb{C}\to\mathbb{C}$ be radial and fix the origin: there is a function
$\sigma:[0,\infty)\to[0,\infty)$ with $\sigma(0)=0$ and $\Phi(\rho e^{i\theta})=\sigma(\rho)\,e^{i\theta}$.
Assume that for every line $\ell$ not through $O$ there is a number $e(\ell)\in(0,\infty)$,
depending only on $\ell$, such that every point $w\in\Phi(\ell)$ lies in the open half-plane bounded
by $\ell$ that contains $O$ and satisfies
\[
  |Ow|=e(\ell)\,\operatorname{dist}(w,\ell),
\]
the focus--directrix equation of the conic $\mathcal{C}\bigl(O,\ell,e(\ell)\bigr)$ with focus $O$ and
directrix $\ell$. (In the hyperbolic case $e(\ell)>1$ the half-plane condition places the image on
the focus-side branch.) Then there is a unique
$R_\Phi\in(0,\infty)$ with $\Phi=f_{R_\Phi}$, and necessarily
$e(\ell)=R_\Phi/\operatorname{dist}(O,\ell)$ for every such $\ell$. Conversely, for every $R>0$ the
map $f_R$ has this property, by Theorem~\ref{thm:directrix}.
\end{theorem}

\begin{proof}
Since $\Phi$ is radial it is rotation equivariant and the hypothesis is rotation invariant, so it
suffices to treat the lines $\ell_d=\{x=d\}$, $d>0$; write $e(d):=e(\ell_d)$. A point
$z=\rho e^{i\theta}\in\ell_d$ has $\cos\theta=d/\rho$, and as $z$ runs over $\ell_d$ the radius
$\rho$ runs over $[d,\infty)$; its image $w=\sigma(\rho)e^{i\theta}$ has modulus $s:=\sigma(\rho)$
and abscissa $x_w=sd/\rho$. First, $\sigma(\rho)>0$ for every $\rho>0$: given $\rho>0$, pick any
$d\in(0,\rho]$ and the point of $\ell_d$ of radius $\rho$; were $\sigma(\rho)=0$ its image would be
$O$, contradicting $|OO|=e(d)\operatorname{dist}(O,\ell_d)=e(d)\,d>0$. The focus-side hypothesis gives
$\operatorname{dist}(w,\ell_d)=d-x_w=d\,(1-s/\rho)$, and $|Ow|=e(d)\operatorname{dist}(w,\ell_d)$
becomes $s=e(d)\,d\,(1-s/\rho)$; dividing by $s>0$,
\begin{equation}\label{eq:sdchar}
  \frac{1}{\sigma(\rho)}=\frac{1}{\rho}+\frac{1}{e(d)\,d}\qquad(\rho\ge d).
\end{equation}
For $0<d_1<d_2$ and every $\rho\ge d_2$, the left-hand side of \eqref{eq:sdchar} is the same for
both lines, so $e(d_1)d_1=e(d_2)d_2$; hence the product $e(d)\,d=:R_\Phi$ is a positive constant.
Then \eqref{eq:sdchar} reads $1/\sigma(\rho)=1/\rho+1/R_\Phi$ for every $\rho>0$ (any $\rho$ exceeds
some admissible $d$), so $\sigma(\rho)=R_\Phi\rho/(R_\Phi+\rho)$ and $\Phi=f_{R_\Phi}$. Since the normalized computation is
invariant under rotation, the same constant $R_\Phi$ is obtained for every line $\ell\not\ni O$; thus
$R_\Phi=e(\ell)\operatorname{dist}(O,\ell)$ is determined by the data, hence unique, and
$e(\ell)=R_\Phi/\operatorname{dist}(O,\ell)$.
\end{proof}

\begin{remark}[No regularity is assumed]
No continuity, monotonicity, or boundedness of $\sigma$ is hypothesized; all three follow. This
complements the normalized characterization of Theorem~\ref{thm:char}, where a continuous, strictly
increasing M\"obius radial action with bounded image is assumed: here the M\"obius form
$\sigma(\rho)=\rho/(1+\rho/R_\Phi)$ is \emph{forced} by the geometry, and the bound
$\sup|\Phi|=R_\Phi$ and the first-order normalization $\sigma'(0^+)=1$ are consequences rather than
inputs.
\end{remark}

\begin{remark}[A single eccentricity per line is essential]
The hypothesis requires \emph{one} eccentricity $e(\ell)$ to serve the whole image of each line.
Without this it is vacuous: for any strictly inward radial map (one with $0<\sigma(\rho)<\rho$) every
image point already satisfies a self-directrix relation, with the point-dependent eccentricity
$e=\bigl[\operatorname{dist}(O,\ell)\,(1/\sigma(\rho)-1/\rho)\bigr]^{-1}$. It is the constancy of
$e(\ell)$ along each line, together with the comparison across lines, that pins $\sigma$ down.
\end{remark}

\begin{remark}[The focus-side clause is essential]
Under the present no-regularity hypotheses the focus-side clause cannot be dropped; it forces the
plus sign in \eqref{eq:sdchar}. Omitting it, the far-side distance
$\operatorname{dist}(w,\ell)=x_w-d$ gives instead $1/\sigma(\rho)=1/\rho-1/(e(d)d)$, the inverse
branch. The pure far-side law cannot hold on all of $[0,\infty)$, having a pole; but
\emph{mixed}-branch radial maps can still send every line into the correct self-directrix conic
without equalling any $f_R$. For example, fix $R>0$ and put $\sigma(\rho)=R\rho/(R-\rho)$ for
$0<\rho<R/2$ and $\sigma(\rho)=R\rho/(R+\rho)$ for $\rho\ge R/2$, with $\sigma(0)=0$: for every line
$\ell_d$ each image point lies on $\mathcal{C}(O,\ell_d,R/d)$ with the single eccentricity $R/d$ (on
the far branch when $\rho<R/2$, the focus-side branch when $\rho\ge R/2$), yet $\Phi\neq f_R$. The
focus-side hypothesis is the cleanest way to exclude such maps. For radial maps, the half-plane part
of that hypothesis is equivalently expressed by the global inwardness condition $\sigma(\rho)<\rho$
for $\rho>0$, since the far branch is outward, $\sigma(\rho)>\rho$. Continuity rules out the displayed example, which jumps from $R$ to
$R/3$ at $\rho=R/2$, though we do not claim it alone excludes every branch mixture.
\end{remark}

\begin{remark}[Higher dimensions]
The proof is dimension-free. A radial map of $\mathbb{R}^n$ that sends every hyperplane $H$ not
through $O$ into the focus side of the focal quadric of revolution with focus $O$ and directrix
hyperplane $H$, with one eccentricity per $H$, must equal some $f_{R_\Phi}$; for $H_d=\{x_1=d\}$ a
point $\rho\mathbf{u}\in H_d$ has $\rho u_1=d$, and the same computation applies. This is the
converse of the higher-dimensional Self-Directrix statement (Section~\ref{sec:higher-d}).
\end{remark}

\subsection{The Confocal--Codirectrix Theorem}\label{sec:confocal}

The Self-Directrix Theorem treats lines. A line, however, is only the degenerate member of a much larger classical family (the focal conics with focus $O$), and the same mechanism that pins the image of a line to its own directrix turns out to fix the focus and directrix of \emph{every} focal polar locus in a fixed focus--directrix \emph{pencil} at once (a term defined precisely below). We now state and prove this generalization. It contains Theorem~\ref{thm:directrix} as the limiting line case (the remark following the proof), and like that theorem it is a statement about \emph{arcs}: the image of an unbounded conic is an arc inside $D_R$, not a full conic. Throughout, \emph{confocal} is meant in the weak sense of sharing the one distinguished focus $O$ together with its directrix; because the eccentricity changes, the second focus of an ellipse or hyperbola generally moves, so the image is \emph{not} confocal in the classical two-foci sense.

\paragraph{Focal polar loci and arcs.}
At a glance, the objects below are the focal radius, the focus--directrix distance, and the directrix,
\[
  \rho(\theta)=\frac{\lambda}{1+e\cos(\theta-\alpha)},\qquad
  \delta=\frac{\lambda}{e},\qquad
  \ell_{\alpha,\delta}=\{z:\operatorname{Re}(e^{-i\alpha}z)=\delta\},
\]
the focal locus being the whole ellipse for $e<1$, the whole parabola for $e=1$ (the backward direction $\alpha+\pi$ excluded in the polar parametrization), and the focus-side branch of the hyperbola for $e>1$. The rest of this paragraph makes these precise.

Fix an axis direction $\alpha$, an eccentricity $e>0$, and a semi-latus
rectum $\lambda>0$. We regard directions as angles modulo $2\pi$ when a
full circle of directions is involved, and write
\[
  \mathbb T := \mathbb R/(2\pi\mathbb Z)
\]
for the circle of directions. Let
\[
  \Theta(e,\alpha)
  :=
  \{\theta:\ 1+e\cos(\theta-\alpha)>0\}.
\]
Equivalently,
\[
  \Theta(e,\alpha)
  =
  \begin{cases}
    \mathbb T, & 0<e<1,\\[2pt]
    \mathbb T\setminus\{[\alpha+\pi]\}, & e=1,\\[2pt]
    (\alpha-\theta^\ast,\alpha+\theta^\ast),
      \quad \theta^\ast=\arccos(-1/e), & e>1.
  \end{cases}
\]
Here $[\alpha+\pi]$ denotes the direction $\alpha+\pi$ modulo $2\pi$.

The set $\Theta(e,\alpha)$ consists of the directions in which the focal radius
\begin{equation}\label{eq:cc-branch}
  \rho(\theta)=\frac{\lambda}{1+e\cos(\theta-\alpha)}
\end{equation}
is positive and finite. Thus it is the whole circle of directions for an
ellipse; the circle with the single backward direction $\alpha+\pi$ removed
for a parabola; and an open angular sector for the focus-side branch of a
hyperbola \cite{akopyan2007,brannan2012}. In the hyperbolic case $e>1$ we have $-1/e\in(-1,0)$, so
$\theta^\ast=\arccos(-1/e)\in(\tfrac{\pi}{2},\pi)$, and the endpoint
directions $\alpha\pm\theta^\ast$ are the asymptotic directions of the
corresponding focus-side branch, where
$1+e\cos(\theta-\alpha)\to 0$ and $\rho\to\infty$.

It is illuminating to rewrite \eqref{eq:cc-branch} in reciprocal form. Since $e/\lambda=1/\delta$ with $\delta=\lambda/e$,
\begin{equation}\label{eq:cc-recip}
  \frac{1}{\rho(\theta)}=\frac{1}{\lambda}+\frac{1}{\delta}\cos(\theta-\alpha).
\end{equation}
In this coordinate the two parameters of the locus are separated: $1/\lambda$ is the constant term and $1/\delta$ is the coefficient of $\cos(\theta-\alpha)$. The cone projection acts by the reciprocal lens identity $1/\rho\mapsto1/\rho+1/R$ (Proposition~\ref{prop:lens}), which shifts only the constant term,
\[
  \frac{1}{\lambda}\longmapsto\frac{1}{\lambda}+\frac{1}{R},
  \qquad
  \frac{1}{\delta}\longmapsto\frac{1}{\delta},
\]
leaving the $\cos$-coefficient untouched. This is exactly why the directrix distance $\delta$ is preserved while the semi-latus rectum changes (the mechanism behind Theorem~\ref{thm:confocal} below).

For any connected subarc $I\subseteq\Theta(e,\alpha)$ (understood as a
connected arc of $\mathbb T$, possibly all of $\mathbb T$, when $0<e<1$, as a connected arc of the
punctured circle $\mathbb T\setminus\{[\alpha+\pi]\}$ when $e=1$, and as a
connected subinterval of $(\alpha-\theta^\ast,\alpha+\theta^\ast)$ when
$e>1$), define the focal arc
\[
  \mathcal{B}_I(\lambda,e,\alpha)
  :=
  \bigl\{\rho(\theta)e^{i\theta}:\theta\in I\bigr\}.
\]
We write
\[
  \mathcal{B}(\lambda,e,\alpha):=\mathcal{B}_{\Theta(e,\alpha)}(\lambda,e,\alpha)
\]
for the full focal polar locus. It lies on the classical conic with focus
$O$, axis direction $\alpha$, eccentricity $e$, and semi-latus rectum
$\lambda$: for $0<e\le 1$ it is that entire conic (the whole ellipse when
$0<e<1$, the whole parabola when $e=1$), while for $e>1$ it is the
focus-side branch of the hyperbola. The vertex nearest the directrix (for $e>1$, the focus-side vertex) lies in direction $\alpha$, where $\rho$
attains its minimum $\lambda/(1+e)$; the directrix belonging to the focus
$O$ is the line perpendicular to the axis, on that same side of $O$, at
distance
\[
  \delta=\frac{\lambda}{e}
\]
from $O$. Concretely, this directrix is the line
\[
  \ell_{\alpha,\delta}=\bigl\{z\in\mathbb{C}:\operatorname{Re}(e^{-i\alpha}z)=\delta\bigr\},
\]
and the focus-side branch is the part lying in the half-plane
$\operatorname{Re}(e^{-i\alpha}z)<\delta$ that contains $O$.

One focus and one directrix determine a whole one-parameter family of
conics, and we use the word \emph{pencil} in this elementary sense: for a
fixed focus $O$, a fixed axis direction $\alpha$, and a fixed directrix line
$\ell_{\alpha,\delta}$, the associated \emph{focal pencil} is the family
\[
  \bigl\{\,\mathcal{C}\bigl(O,\ \ell_{\alpha,\delta},\ e\bigr)\ :\ e>0\,\bigr\},
\]
obtained by varying the eccentricity $e$ while the focus and the directrix
are held fixed. Equivalently, since the semi-latus rectum is
$\lambda=e\delta$, the pencil may be parametrized by $(\lambda,e)$ with the
ratio $\lambda/e=\delta$ held fixed. In this terminology, the
Confocal--Codirectrix Theorem below says that $f_R$ carries each member of
such a pencil to (an arc of) another member of the \emph{same} pencil,
since $\lambda'/e'=\lambda/e=\delta$.

A line not passing through $O$, at distance $\delta>0$, may be viewed as the
limiting member of this fixed-directrix pencil as $e\to\infty$ with
$\delta$ fixed, so that $\lambda=e\delta\to\infty$; in this limit the conic
collapses onto its own directrix. Correspondingly the admissible sector
$\Theta(e,\alpha)=(\alpha-\theta^\ast,\alpha+\theta^\ast)$, with
$\theta^\ast=\arccos(-1/e)\to\tfrac{\pi}{2}$, narrows to the open half-plane
$(\alpha-\tfrac{\pi}{2},\alpha+\tfrac{\pi}{2})$ of directions in which a ray
from $O$ meets the line. By the \emph{carrier} of a focal arc
$\mathcal{B}_I(\lambda,e,\alpha)$ we mean the complete classical conic containing
that arc. Thus an arc may be a proper subset of its carrier. Conversely, our usage of ``arc'' includes the case of the full carrier: when the angular parameter runs over the whole circle (as for the image of an ellipse in Theorem~\ref{thm:confocal}), the arc \emph{is} the entire closed carrier ellipse.

\begin{theorem}[Confocal--Codirectrix Theorem]\label{thm:confocal}
Let $\lambda,R>0$, $e>0$, $\alpha\in\mathbb{R}$, and let
$\delta=\lambda/e$ be the focus--directrix distance. For every subarc
$I\subseteq\Theta(e,\alpha)$,
\begin{equation}\label{eq:cc-image}
  f_R\bigl(\mathcal{B}_I(\lambda,e,\alpha)\bigr)
  =
  \mathcal{B}_I(\lambda',e',\alpha),
  \qquad
  \lambda'=\frac{R\,\lambda}{R+\lambda},
  \qquad
  e'=\frac{R\,e}{R+\lambda},
\end{equation}
the image being taken over the \emph{same} subarc $I$; since $e'<e$ one has
$I\subseteq\Theta(e,\alpha)\subseteq\Theta(e',\alpha)$ (if
$\cos(\theta-\alpha)\ge0$ then $1+e'\cos(\theta-\alpha)\ge1>0$, while if
$\cos(\theta-\alpha)<0$ then
$1+e'\cos(\theta-\alpha)>1+e\cos(\theta-\alpha)>0$), so the right-hand
side is a well-defined focal arc.
Consequently $f_R(\mathcal{B}_I(\lambda,e,\alpha))$ lies on a conic with
the same focus $O$, the same axis direction $\alpha$, and the same
directrix:
\begin{equation}\label{eq:cc-transf}
  \delta'=\frac{\lambda'}{e'}=\delta,
  \qquad\text{equivalently}\qquad
  \frac{1}{\lambda'}=\frac{1}{\lambda}+\frac{1}{R},
  \quad
  \frac{1}{e'}=\frac{1}{e}+\frac{\delta}{R}.
\end{equation}
The focus, axis, and directrix are therefore invariant under $f_R$, while
the eccentricity strictly decreases, $e'<e$. Applied to the full locus
$\mathcal{B}(\lambda,e,\alpha)$, that is, to $I=\Theta(e,\alpha)$, the
image is described precisely by the type of the input:
\begin{enumerate}[label=\textnormal{(\roman*)}]
\item \textbf{Ellipse} $(e<1)$: here
$\Theta(e,\alpha)=\Theta(e',\alpha)$ is the whole circle of directions, so
$f_R(\mathcal{B})$ is the whole carrier ellipse
$\mathcal{B}(\lambda',e',\alpha)$, lying strictly inside $D_R$; its
farthest point from $O$ is at distance $\lambda'/(1-e')<R$.

\item \textbf{Parabola} $(e=1)$: the carrier of the image is the ellipse
of parameters $(\lambda',e')$, with $e'=R/(R+\lambda)<1$, and
$f_R(\mathcal{B})$ is that carrier ellipse with exactly one point removed:
its far vertex, in direction $\alpha+\pi$. That vertex is at distance
$\lambda'/(1-e')=R$ from $O$: the carrier ellipse is internally tangent to
$\partial D_R$, and the omitted point is precisely the point of tangency.
Thus the image itself does not meet $\partial D_R$.

\item \textbf{Hyperbola branch} $(e>1)$: $f_R(\mathcal{B})$ is the open
arc, over
$\theta\in(\alpha-\theta^\ast,\alpha+\theta^\ast)$ with
$\theta^\ast=\arccos(-1/e)\in(\tfrac{\pi}{2},\pi)$, of its carrier conic.
The type of that carrier (ellipse, parabola, or hyperbola) is
governed by $e'$, as in Corollary~\ref{cor:cc-tri}. Its two limiting
endpoints are
$R\,e^{i(\alpha\pm\theta^\ast)}\in\partial D_R$, lying in the asymptotic
directions of the original hyperbola branch, where
$1+e\cos(\theta-\alpha)\to0$ and $\rho\to\infty$. These endpoints are
approached but not attained; the arc is in general a proper subarc of the
carrier.
\end{enumerate}
\end{theorem}

One focus and one directrix support a whole pencil of conics, and $f_R$
slides each focal polar locus inward along its rays to another member of the very same
pencil; see Figure~\ref{fig:confocal} for the elliptic case. The parabolic and
hyperbolic cases, with the directrix correspondingly closer to the focus, are shown
in Figures~\ref{fig:confocal-parabola} and~\ref{fig:confocal-hyperbola}.

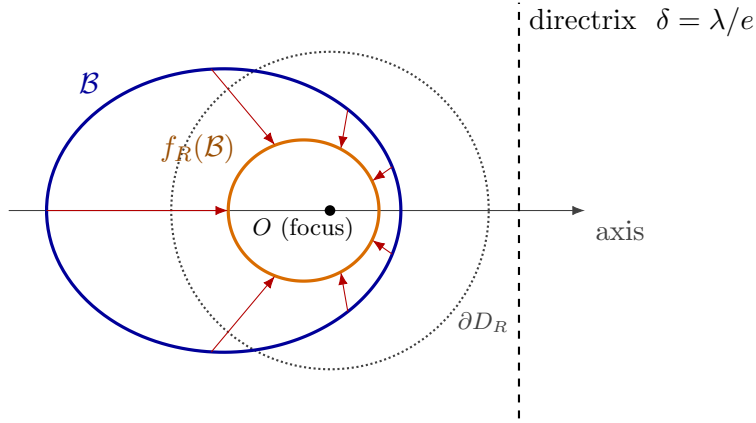
\begin{figure}[htbp]
\centering
\begin{tikzpicture}[scale=1.25,>=Latex,
   conic/.style={very thick},
   pull/.style={->,red!70!black,thin}]
\def\dlt{2}
\draw[->,darkgray] (-3.4,0)--(2.7,0) node[below right]{axis};
\fill (0,0) circle (1.6pt)
  node[right=2pt, font=\footnotesize, xshift=-35pt, yshift=-7pt]{$O$ (focus)};
\draw[thick,dashed] (\dlt,-2.2)--(\dlt,2.2);
\node[right] at (\dlt,2.0) {directrix $\ \delta=\lambda/e$};

\draw[densely dotted, thick, black!75] (0,0) circle (1.68);
\node[black!75, font=\footnotesize] at (1.62,-1.18) {$\partial D_R$};

\draw[conic,blue!60!black]
  plot[domain=0:360,samples=240,variable=\t]
  ({1.2/(1+0.6*cos(\t))*cos(\t)},{1.2/(1+0.6*cos(\t))*sin(\t)});
\node[blue!60!black] at (-2.55,1.35) {$\mathcal{B}$};

\draw[conic,orange!85!black]
  plot[domain=0:360,samples=240,variable=\t]
  ({0.7/(1+0.35*cos(\t))*cos(\t)},{0.7/(1+0.35*cos(\t))*sin(\t)});
\node[orange!60!black] at (-1.4,0.65) {$f_R(\mathcal{B})$};

\foreach \t in {35,80,130,180,230,280,325}{
  \pgfmathsetmacro{\ro}{1.2/(1+0.6*cos(\t))}
  \pgfmathsetmacro{\ri}{0.7/(1+0.35*cos(\t))}
  \draw[pull] ({\ro*cos(\t)},{\ro*sin(\t)}) --
              ({\ri*cos(\t)},{\ri*sin(\t)});
}
\end{tikzpicture}
\caption{The Confocal--Codirectrix Theorem in the elliptic case. The focus
$O$ and the directrix (dashed) are shared by every conic in the pencil. The
map $f_R$ pulls each point of the locus $\mathcal{B}$ inward along its ray
from $O$ (red arrows) onto $f_R(\mathcal{B})$, another focal locus with the same focus
and the same directrix but smaller eccentricity; the image lies strictly
inside the boundary circle $\partial D_R$ (dotted). Starting from a fixed
finite ellipse, varying $R$ sweeps only the sub-pencil of smaller
eccentricities $0<e'<e$; it is the limiting line member $(e\to\infty)$ whose
images sweep the whole pencil, recovering the trichotomy of
Corollary~\ref{cor:sweep}.}
\label{fig:confocal}
\end{figure}

\begin{figure}[htbp]
\centering
\begin{tikzpicture}[scale=1.25,>=Latex,
   conic/.style={very thick},
   pull/.style={->,red!70!black,thin}]
\def\dlt{1.2}
\draw[->,darkgray] (-3.4,0)--(2.7,0) node[below right]{axis};
\fill (0,0) circle (1.6pt)
  node[right=2pt, font=\footnotesize, xshift=-35pt, yshift=-7pt]{$O$ (focus)};
\draw[thick,dashed] (\dlt,-2.7)--(\dlt,2.7);
\node[right] at (\dlt,2.5) {directrix $\ \delta=\lambda/e$};

\draw[densely dotted, thick, black!75] (0,0) circle (1.68);
\node[black!75, font=\footnotesize] at (-1.62,-1.62) {$\partial D_R$};

\draw[conic,blue!60!black]
  plot[domain=-130:130,samples=300,variable=\t]
  ({1.2/(1+cos(\t))*cos(\t)},{1.2/(1+cos(\t))*sin(\t)});
\node[blue!60!black] at (-2.35,2.0) {$\mathcal{B}$};

\draw[conic,orange!85!black]
  plot[domain=0:360,samples=240,variable=\t]
  ({0.7/(1+0.58333*cos(\t))*cos(\t)},{0.7/(1+0.58333*cos(\t))*sin(\t)});
\node[orange!60!black] at (-0.5,0.45) {$f_R(\mathcal{B})$};

\draw[orange!85!black,fill=white] (-1.68,0) circle (1.8pt);
\node[orange!85!black, font=\footnotesize, anchor=north east] at (-1.74,-0.10) {removed vertex};

\foreach \t in {0,40,80,115,245,280,320}{
  \pgfmathsetmacro{\ro}{1.2/(1+cos(\t))}
  \pgfmathsetmacro{\ri}{0.7/(1+0.58333*cos(\t))}
  \draw[pull] ({\ro*cos(\t)},{\ro*sin(\t)}) --
              ({\ri*cos(\t)},{\ri*sin(\t)});
}
\end{tikzpicture}
\caption{The Confocal--Codirectrix Theorem in the parabolic case ($e=1$). The focus
$O$ and the directrix (dashed) are shared by the parabola $\mathcal{B}$ and its
image. The map $f_R$ pulls each point of $\mathcal{B}$ inward along its ray from $O$
(red arrows) onto the carrier ellipse of parameters $(\lambda',e')$, from which
exactly one point is missing: the far vertex (open dot), at distance
$\lambda'/(1-e')=R$ from $O$, where the carrier is internally tangent to
$\partial D_R$ (dotted), as in case~(ii) of Theorem~\ref{thm:confocal}.}
\label{fig:confocal-parabola}
\end{figure}
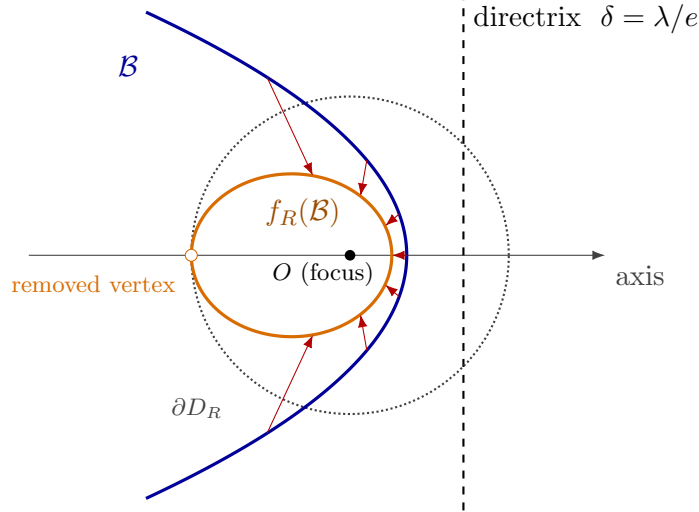

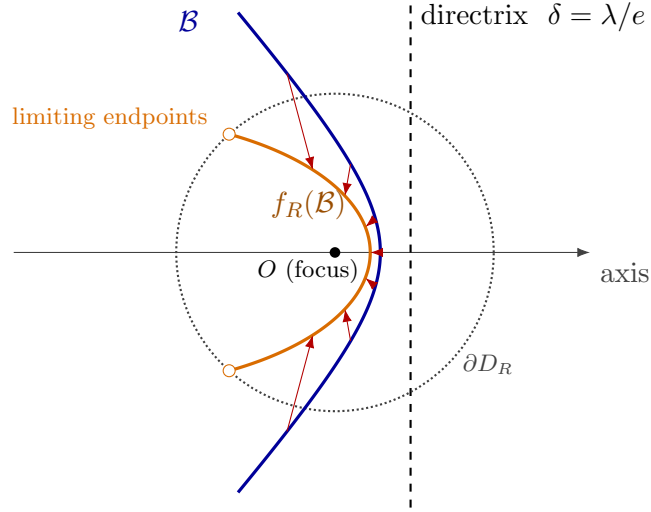
\begin{figure}[htbp]
\centering
\begin{tikzpicture}[scale=1.25,>=Latex,
   conic/.style={very thick},
   pull/.style={->,red!70!black,thin}]
\def\dlt{0.8}
\draw[->,darkgray] (-3.4,0)--(2.7,0) node[below right]{axis};
\fill (0,0) circle (1.6pt)
  node[right=2pt, font=\footnotesize, xshift=-35pt, yshift=-7pt]{$O$ (focus)};
\draw[thick,dashed] (\dlt,-2.7)--(\dlt,2.7);
\node[right] at (\dlt,2.5) {directrix $\ \delta=\lambda/e$};

\draw[densely dotted, thick, black!75] (0,0) circle (1.68);
\node[black!75, font=\footnotesize] at (1.62,-1.18) {$\partial D_R$};

\draw[conic,blue!60!black]
  plot[domain=-112:112,samples=300,variable=\t]
  ({1.2/(1+1.5*cos(\t))*cos(\t)},{1.2/(1+1.5*cos(\t))*sin(\t)});
\node[blue!60!black] at (-1.55,2.45) {$\mathcal{B}$};

\draw[conic,orange!85!black]
  plot[domain=-131.81:131.81,samples=300,variable=\t]
  ({0.7/(1+0.875*cos(\t))*cos(\t)},{0.7/(1+0.875*cos(\t))*sin(\t)});
\node[orange!60!black] at (-0.28,0.5) {$f_R(\mathcal{B})$};

\draw[orange!85!black,fill=white] (-1.12,1.2517) circle (1.8pt);
\draw[orange!85!black,fill=white] (-1.12,-1.2517) circle (1.8pt);
\node[orange!85!black,font=\footnotesize,anchor=east] at (-1.24,1.42) {limiting endpoints};

\foreach \t in {0,40,80,105,255,280,320}{
  \pgfmathsetmacro{\ro}{1.2/(1+1.5*cos(\t))}
  \pgfmathsetmacro{\ri}{0.7/(1+0.875*cos(\t))}
  \draw[pull] ({\ro*cos(\t)},{\ro*sin(\t)}) --
              ({\ri*cos(\t)},{\ri*sin(\t)});
}
\end{tikzpicture}
\caption{The Confocal--Codirectrix Theorem in the hyperbolic case ($e>1$). The focus
$O$ and the directrix (dashed) are shared by the focus-side branch $\mathcal{B}$ and
its image. The map $f_R$ pulls each point of $\mathcal{B}$ inward along its ray from
$O$ (red arrows) onto the open arc, over the asymptotic angular sector
$(\alpha-\theta^\ast,\alpha+\theta^\ast)$ of $\mathcal{B}$, of the carrier conic
$(\lambda',e')$, which for the parameters drawn is an ellipse ($e'<1$, in accordance with Corollary~\ref{cor:cc-tri}). The
limiting endpoints (open dots) lie on $\partial D_R$ (dotted), in the same directions as the asymptotes
of the original hyperbola branch, and are approached but not attained, as in case~(iii) of
Theorem~\ref{thm:confocal}.}
\label{fig:confocal-hyperbola}
\end{figure}

\begin{proof}
Because $f_R$ is radial, it fixes $\theta=\arg z$, and by
\eqref{eq:recap-mag} it acts on the reciprocal modulus by the constant
shift
\[
  \frac{1}{\rho}\longmapsto \frac{1}{\rho}+\frac{1}{R},
\]
the reciprocal lens identity of Proposition~\ref{prop:lens}. From
\eqref{eq:cc-branch}, a point of
$\mathcal{B}_I(\lambda,e,\alpha)$ at angle $\theta\in I$ has reciprocal
modulus
\[
  \frac{1}{\rho(\theta)}
  =
  \frac{1+e\cos(\theta-\alpha)}{\lambda}
  =
  \frac{1}{\lambda}
  +
  \frac{e}{\lambda}\cos(\theta-\alpha).
\]
Therefore its image, at the same angle, has reciprocal modulus
\[
  \frac{1}{\rho'(\theta)}
  =
  \frac{1}{\rho(\theta)}+\frac{1}{R}
  =
  \underbrace{\Bigl(\frac{1}{\lambda}+\frac{1}{R}\Bigr)}_{=\,1/\lambda'}
  +
  \frac{e}{\lambda}\cos(\theta-\alpha)
  =
  \frac{1}{\lambda'}
  \Bigl(1+\dfrac{e\lambda'}{\lambda}\cos(\theta-\alpha)\Bigr),
\]
with
\[
  \lambda'=\frac{R\lambda}{R+\lambda}.
\]
Hence
\[
  \rho'(\theta)
  =
  \frac{\lambda'}{1+e'\cos(\theta-\alpha)},
  \qquad
  e'=\frac{e\lambda'}{\lambda}
  =
  \frac{Re}{R+\lambda}.
\]
Since $f_R$ is a bijection from each source ray onto its image segment in
$D_R$ and fixes the ray's direction, it maps
$\mathcal{B}_I(\lambda,e,\alpha)$ bijectively onto
$\{\rho'(\theta)e^{i\theta}:\theta\in I\}$.

Because $0<e'<e$, the set of admissible directions can only grow: if
$1+e\cos(\theta-\alpha)>0$ then a fortiori $1+e'\cos(\theta-\alpha)>0$, so
\[
  I\subseteq\Theta(e,\alpha)\subseteq\Theta(e',\alpha).
\]
Hence $1+e'\cos(\theta-\alpha)>0$ throughout $I$, every $\rho'(\theta)$ is
positive and finite, and the image is exactly
$\mathcal{B}_I(\lambda',e',\alpha)$, the focal arc of parameters
$(\lambda',e')$ over the same subarc $I$. This proves \eqref{eq:cc-image}.

The focus--directrix distance of the image is
\[
  \delta'
  =
  \frac{\lambda'}{e'}
  =
  \frac{\lambda}{e}
  =
  \delta,
\]
giving \eqref{eq:cc-transf}; the reciprocal laws are immediate. Finally,
$e'<e$ because $R/(R+\lambda)<1$.

For the three cases, take $I=\Theta(e,\alpha)$. If $e<1$, then
$\Theta(e,\alpha)$ is the whole circle of directions, and since
$e'<e<1$ it equals $\Theta(e',\alpha)$. Hence the image is the full carrier
ellipse $\mathcal{B}(\lambda',e',\alpha)$. Its farthest point from $O$
occurs at $\theta=\alpha+\pi$, where, the denominator
$R+\lambda-Re=R(1-e)+\lambda$ being positive for $e<1$,
\[
  \rho'_{\max}
  =
  \frac{\lambda'}{1-e'}
  =
  \frac{R\lambda}{R+\lambda-Re}<
  R
  \iff
  \lambda<R+\lambda-Re
  \iff
  e<1.
\]
Thus the carrier ellipse lies strictly inside $D_R$, proving (i).

If $e=1$, then the original angular domain is the circle with the single
direction $\alpha+\pi$ removed. The image is therefore the focal arc of
$(\lambda',e')$ over that punctured circle, that is, the carrier ellipse
minus its point at $\theta=\alpha+\pi$. At that point,
\[
  \rho'_{\max}
  =
  \frac{\lambda'}{1-e'}
  =
  R,
\]
obtained by setting $e=1$ in the preceding formula. Thus the carrier
ellipse is internally tangent to $\partial D_R$ and the omitted point is
precisely the tangency point, proving (ii).

If $e>1$, then
\[
  \Theta(e,\alpha)
  =
  (\alpha-\theta^\ast,\alpha+\theta^\ast),
  \qquad
  \theta^\ast=\arccos(-1/e).
\]
As $\theta-\alpha\to\pm\theta^\ast$, one has
$\rho(\theta)\to\infty$. Hence, by Proposition~\ref{prop:lens},
$|f_R|\to R$, while the argument tends to
$\alpha\pm\theta^\ast$. Therefore the image is the open arc of the carrier
conic with limiting endpoints
\[
  R\,e^{i(\alpha\pm\theta^\ast)}\in\partial D_R,
\]
not attained because $f_R(\mathbb C)=D_R$. This proves (iii).
\end{proof}

\begin{remark}[The Self-Directrix Theorem as the limiting member]
Theorem~\ref{thm:directrix} is recovered as the limiting member
$1/e\to0$ of Theorem~\ref{thm:confocal}. A line $\ell$ at distance $d$
from $O$ is the limiting member of the fixed-directrix pencil whose
directrix is itself: $\delta=d$ is fixed while $e\to\infty$, that is,
$1/e=0$ and correspondingly
\[
  \frac{1}{\lambda}
  =
  \frac{1}{e\delta}
  =
  0.
\]
The laws \eqref{eq:cc-transf} are affine, hence continuous, in the
reciprocal coordinate $1/e$, with
\[
  \frac{1}{\lambda}
  =
  \frac{1}{\delta}\frac{1}{e}.
\]
Thus they extend to this limiting member. Evaluating them at
$1/\lambda=1/e=0$ gives at once
\[
  \frac{1}{\lambda'}
  =
  \frac{1}{\lambda}+\frac{1}{R}
  =
  \frac{1}{R},
  \qquad
  \frac{1}{e'}
  =
  \frac{1}{e}+\frac{\delta}{R}
  =
  \frac{d}{R},
\]
that is,
\[
  \lambda'=R,
  \qquad
  e'=\frac{R}{d},
  \qquad
  \delta'=d.
\]
The endpoints $R\,e^{i(\alpha\pm\theta^\ast)}$ of case~(iii) tend, as
$\theta^\ast=\arccos(-1/e)\to\pi/2$, to the two directions parallel to
$\ell$, namely $\pm Ru$ of Proposition~\ref{prop:image-arc}. This
reproduces the Self-Directrix arc (directrix $\ell$, semi-latus rectum
$R$, eccentricity $R/d$) and shows that the self-directrix phenomenon is
not special to lines: every finite focal locus keeps its directrix for
exactly the same reason.
\end{remark}

\begin{corollary}[Trichotomy of the carrier conic]\label{cor:cc-tri}
For a focal locus with focus $O$, directrix distance $\delta$, and
eccentricity $e$, the carrier conic of $f_R(\mathcal{B})$ is an ellipse,
parabola, or hyperbola according to the sign of
\[
  e\,(R-\delta)-R.
\]
In particular it is an ellipse if and only if
\[
  e\,(R-\delta)<R,
\]
which holds for \emph{every} eccentricity $e$ as soon as $\delta\ge R$.
Thus if the directrix lies at distance $\delta\ge R$ from the focus, $f_R$
carries the locus onto an elliptical arc regardless of the original type.
\end{corollary}

\begin{proof}
By \eqref{eq:cc-image},
\[
  e'
  =
  \frac{Re}{R+\lambda}
  =
  \frac{Re}{R+e\delta}.
\]
Therefore
\[
  e'<1
  \iff
  Re<R+e\delta
  \iff
  e(R-\delta)<R.
\]
If $\delta\ge R$, the left-hand side is $\le0<R$ for all $e>0$. The
parabolic and hyperbolic cases are given by equality and by the reverse
inequality, respectively.
\end{proof}

\begin{remark}[Reading the pencil; iteration]
The trichotomy of Corollary~\ref{cor:sweep} is the special case of this
pencil view in which the input is the limiting line member: for fixed $R$,
sweeping the line through distances $d>0$ realizes every eccentricity
$R/d$. A fixed finite focal locus is more rigid: its image eccentricity
\[
  e'=\frac{Re}{R+\lambda}
\]
ranges only over $0<e'<e$, sweeping the lower-eccentricity sub-pencil.
Iterating sharpens this: since $f_R^{\,n}=f_{R/n}$
(Corollary~\ref{cor:iteration}), the $n$-fold image of
$\mathcal{B}(\lambda,e,\alpha)$ is the focal arc of parameters
\[
  \lambda_n=\frac{R\lambda}{R+n\lambda},
  \qquad
  e_n=\frac{Re}{R+n\lambda},
\]
over the original angular domain, with
\[
  \delta_n=\frac{\lambda_n}{e_n}=\delta
\]
fixed. Both $\lambda_n$ and $e_n$ are $O(1/n)$, so the carrier conics
become asymptotically circular while collapsing to the focus $O$: a
fixed focus and fixed directrix, ever rounder and smaller. The limit is the
point $O$, not a nondegenerate circle.
\end{remark}

\begin{remark}[Why the branch, and the one-translation view]
For $e>1$, the polar formula \eqref{eq:cc-branch} with $\rho>0$ describes
only the branch on the focus side of the directrix. The far branch satisfies
\[
  \frac{1}{\rho}
  =
  -\frac{1}{\lambda}
  +
  \frac{e}{\lambda}\cos(\theta-\alpha),
\]
with the opposite constant term: a point $\rho e^{i\theta}$ on the far side of the directrix has unsigned directrix distance $\rho\cos(\theta-\alpha)-\delta$ instead of $\delta-\rho\cos(\theta-\alpha)$, and the focus--directrix condition $\rho=e\bigl(\rho\cos(\theta-\alpha)-\delta\bigr)$ rearranges, using $\lambda=e\delta$, to the displayed equation. Here $\rho>0$ requires
$e\cos(\theta-\alpha)>1$, so the far branch is parametrized over the
narrower angular sector centered on the axis direction $\alpha$;
geometrically it lies on the other side of the directrix from the
focus-side vertex. Applying the lens identity gives
\[
  \frac{1}{\rho'}
  =
  \Bigl(\frac{1}{R}-\frac{1}{\lambda}\Bigr)
  +
  \frac{e}{\lambda}\cos(\theta-\alpha),
\]
a different sign law, so the far branch must be treated separately. Its
image is, of course, still bounded and lies in $D_R$, since
\[
  \frac{1}{\rho'}=\frac{1}{\rho}+\frac{1}{R}>\frac{1}{R}
\]
there.

In the special case $\lambda=R$, the constant term vanishes, leaving
\[
  \rho'\cos(\theta-\alpha)=\frac{R}{e}=\delta.
\]
Thus the far-branch image maps homeomorphically onto the open chord of the directrix
lying inside $D_R$, with its endpoints on $\partial D_R$ omitted, not
onto the entire directrix line.

A homeomorphism cannot identify a two-component hyperbola with a
one-component ellipse, so no full-conic version of case~(iii) could hold.
Finally, the two reciprocal laws of \eqref{eq:cc-transf} are one statement
seen along transversal families: along a ray, $1/\rho$ translates by
$1/R$, while across the fixed-directrix pencil, $1/e$ translates by
$\delta/R$. Both are translations because, in the inverted coordinate of
Proposition~\ref{prop:inversion},
\[
  f_R=\iota\circ\tau_R\circ\iota
\]
acts as the rigid radial shift $\tau_R$. The present theorem is that shift
carrying the inverse of one focal locus (a lima\c{c}on; in the limiting
line case, a circle through the origin) to the inverse of another.
\end{remark}

\subsection{The focal-chord harmonic mean: Askwith's theorem and a chord-level view}\label{sec:focal-chord}

The parameter laws \eqref{eq:cc-transf} admit a second derivation, at the
level of individual focal chords, through a classical theorem recorded by
Askwith \cite[p.~114]{askwith1903}: ``\emph{The semi-latus rectum of a
conic is a harmonic mean between the segments of any focal chord.}'' This
subsection shows that Askwith's theorem is the chord-level shadow of the
lens mechanism of Theorem~\ref{thm:confocal}: the semi-latus rectum, being
a harmonic mean of focal radii, transforms under $f_R$ exactly as the
modulus of a single point does, which yields the law
$1/\lambda'=1/\lambda+1/R$ in one line; and the complementary projective
statement (the harmonic \emph{conjugate} rather than the harmonic
\emph{mean}) accounts for the invariance of the directrix,
$\delta'=\delta$. The two laws of \eqref{eq:cc-transf} thus appear as the
even and odd parts of one reciprocal-radius decomposition along focal
chords. Throughout, $\mathrm{H}(a,b):=2ab/(a+b)$ denotes the harmonic
mean.

\paragraph{Askwith's theorem in the reciprocal coordinate.}
In the reciprocal coordinate $u=1/\rho$ the theorem is visible at a
glance. A focal line through $O$ in direction $\theta$ meets the carrier
conic of $\mathcal{B}(\lambda,e,\alpha)$ at the two opposite radial
parameters $\theta$ and $\theta+\pi$; the radii are finite and positive
when both endpoints lie on the relevant branch, and are read as signed or
projective radii in the exceptional parabolic and hyperbolic cases (see
the proviso below). By \eqref{eq:cc-recip},
\begin{equation}\label{eq:chord-pm}
  \frac{1}{r_1}=\frac{1}{\lambda}+\frac{\cos(\theta-\alpha)}{\delta},
  \qquad
  \frac{1}{r_2}=\frac{1}{\lambda}-\frac{\cos(\theta-\alpha)}{\delta}.
\end{equation}
Adding annihilates the cosine term:
\begin{equation}\label{eq:askwith}
  \frac{1}{r_1}+\frac{1}{r_2}=\frac{2}{\lambda},
  \qquad\text{i.e.}\qquad
  \lambda=\mathrm{H}(r_1,r_2),
\end{equation}
which is Askwith's theorem. (For $e>1$, both reciprocals in
\eqref{eq:chord-pm} are positive, and both endpoints lie on the
focus-side branch, exactly when $e\,|\cos(\theta-\alpha)|<1$; for a
chord meeting the two branches the same computation holds with signed
radii, the far-branch segment counted negatively, as in the signed
focal polar form of the remark completing the arc in
Section~\ref{sec:image-arc}. This proviso does not affect what follows,
which uses only \eqref{eq:chord-pm}.) In this coordinate the two
parameters of the locus are separated by parity in the reversal
$\theta\mapsto\theta+\pi$:
\begin{equation}\label{eq:parity}
  \underbrace{\tfrac12\Bigl(\tfrac{1}{r_1}+\tfrac{1}{r_2}\Bigr)}_{\text{even part}}
  =\frac{1}{\lambda}=\frac{1}{\mathrm{H}(r_1,r_2)},
  \qquad\qquad
  \underbrace{\tfrac12\Bigl(\tfrac{1}{r_1}-\tfrac{1}{r_2}\Bigr)}_{\text{odd part}}
  =\frac{\cos(\theta-\alpha)}{\delta}.
\end{equation}
Askwith's theorem says the even part is the constant $1/\lambda$, so the
semi-latus rectum can be read off from \emph{any} focal chord, not only
from the latus rectum itself.

\begin{lemma}[Harmonic-mean equivariance]\label{lem:hm-equiv}
Let $\sigma(\rho)=R\rho/(R+\rho)$ be the radial action of $f_R$. Then for
all $r_1,r_2>0$,
\[
  \mathrm{H}\bigl(\sigma(r_1),\sigma(r_2)\bigr)
  =\sigma\bigl(\mathrm{H}(r_1,r_2)\bigr).
\]
\end{lemma}

\begin{proof}
The lens identity (Proposition~\ref{prop:lens}) says
$1/\sigma(r)=1/r+1/R$, a translation in the reciprocal coordinate; and
the harmonic mean is the arithmetic mean in the reciprocal coordinate.
Translation commutes with averaging:
\[
  \frac{2}{\mathrm{H}(\sigma(r_1),\sigma(r_2))}
  =\frac{1}{r_1}+\frac{1}{r_2}+\frac{2}{R}
  =\frac{2}{\mathrm{H}(r_1,r_2)}+\frac{2}{R}
  =\frac{2}{\sigma(\mathrm{H}(r_1,r_2))}.
  \qedhere
\]
\end{proof}

The identity extends by continuity to $r_i=\infty$ (where
$\sigma(\infty)=R$), and in the signed or projective readings used below
it should be read in the reciprocal coordinate, where it is simply the
affine translation $u\mapsto u+1/R$ commuting with averaging.

\begin{remark}
Equivalently, by the trapezoid construction of
Proposition~\ref{prop:planar},
$\sigma(\rho)=\tfrac12\,\mathrm{H}(\rho,R)$: the radial action of $f_R$
is itself half a harmonic mean, so Lemma~\ref{lem:hm-equiv} is an
instance of the associativity of reciprocal addition.
\end{remark}

\begin{proposition}[The $\lambda$-law via Askwith's theorem]\label{prop:lambda-chord}
Let $\mathcal{B}=\mathcal{B}(\lambda,e,\alpha)$ be a focal polar locus
with focus $O$, and let $\lambda'$ be the semi-latus rectum of the
carrier conic of $f_R(\mathcal{B})$. Then
\[
  \frac{1}{\lambda'}=\frac{1}{\lambda}+\frac{1}{R},
  \qquad\text{i.e.}\qquad
  \lambda'=\frac{R\lambda}{R+\lambda}=\sigma(\lambda),
\]
the first law of \eqref{eq:cc-transf}.
\end{proposition}

\begin{proof}
Because $f_R$ is radial and fixes $O$, it acts on each focal chord in the
reciprocal radial coordinate. More precisely, it sends the two
signed/projective radial parameters of a focal chord of the carrier conic of
$\mathcal{B}$, along the opposite directions $\theta$ and $\theta+\pi$, to the
corresponding parameters of a focal chord of the carrier conic of
$f_R(\mathcal{B})$, with values $\sigma(r_1)$ and $\sigma(r_2)$.
When both endpoints lie on the actual locus $\mathcal{B}$, this is the
ordinary forward image of the chord endpoints; otherwise it is the
projective continuation described after Lemma~\ref{lem:hm-equiv}.
By Askwith's theorem \eqref{eq:askwith}, applied on both sides, and
Lemma~\ref{lem:hm-equiv},
\[
  \lambda'
  =\mathrm{H}\bigl(\sigma(r_1),\sigma(r_2)\bigr)
  =\sigma\bigl(\mathrm{H}(r_1,r_2)\bigr)
  =\sigma(\lambda).
  \qedhere
\]
\end{proof}

In words: the semi-latus rectum transforms under $f_R$ exactly as if it
were the modulus of a single point, because it is a harmonic mean of
moduli and $f_R$ respects harmonic means. No polar equation is needed.
By the remark after Lemma~\ref{lem:hm-equiv}, one may also say: the image
semi-latus rectum is half the harmonic mean of the input semi-latus
rectum and the cone radius, $\lambda'=\tfrac12\,\mathrm{H}(\lambda,R)$.

The complementary half of Theorem~\ref{thm:confocal},
$\delta'=\delta$, is carried by the odd part of \eqref{eq:parity}, which
the reciprocal translation $u\mapsto u+1/R$ leaves untouched: subtracting
the two relations \eqref{eq:chord-pm} eliminates the constant term, and
the lens shift changes only the constant term. This too has a classical
projective reading, matching the harmonic-range result of
Proposition~\ref{prop:harmonic} below.

\begin{proposition}[Directrix point as harmonic conjugate]\label{prop:directrix-conj}
Let $P_1P_2$ be a focal chord of a conic with focus $O$ and corresponding
directrix $\ell$, and let $Q$ be the intersection of the chord line with
$\ell$, with $Q$ understood projectively (a point at infinity) when the
chord is parallel to $\ell$, i.e.\ when $\cos(\theta-\alpha)=0$. Then $(P_1,P_2;\,O,Q)=-1$: the focus and the directrix point are
harmonic conjugates with respect to the chord endpoints. In signed
coordinates along the chord
($P_1\leftrightarrow r_1$, $P_2\leftrightarrow -r_2$,
$O\leftrightarrow 0$), the conjugate point sits at
\[
  q=\frac{2r_1r_2}{r_2-r_1}=\frac{\delta}{\cos(\theta-\alpha)},
\]
the signed distance from $O$ to $\ell$ measured along the chord; the odd
part of \eqref{eq:parity} thus locates the directrix. See
Figure~\ref{fig:directrix-conj}.
\end{proposition}

\begin{proof}
The harmonic conjugate of $0$ with respect to $r_1$ and $-r_2$ is
$q=2r_1r_2/(r_2-r_1)$, the point making the cross-ratio $-1$;
equivalently, $2/q=1/r_1+1/(-r_2)$. By \eqref{eq:chord-pm},
\[
  \frac{1}{r_1}-\frac{1}{r_2}=\frac{2\cos(\theta-\alpha)}{\delta}
  \quad\Longrightarrow\quad
  q=\frac{2}{\,1/r_1-1/r_2\,}=\frac{\delta}{\cos(\theta-\alpha)},
\]
which is exactly the signed distance along the chord direction $\theta$
from $O$ to the line
$\ell_{\alpha,\delta}=\{z:\operatorname{Re}(e^{-i\alpha}z)=\delta\}$.
This is the polarity statement that the directrix is the polar of the
focus, specialized to the pencil of chords through $O$. Both ingredients
are classical: the general theorem that the polar of a point meets any
chord through that point in the harmonic conjugate of that point with
respect to the chord's endpoints is standard pole--polar theory
\cite{veblenyoung1910,coxeterPG1987}, and Askwith derives the
focus--directrix correspondence precisely from this harmonic property of
pole and polar \cite[ch.~IX, \S94]{askwith1917}.
\end{proof}

\begin{corollary}[Both laws of \eqref{eq:cc-transf}, by parity]\label{cor:chord-parity}
Along each focal chord, decompose the reciprocal radii
\eqref{eq:chord-pm} into even and odd parts as in \eqref{eq:parity}. The
lens shift $u\mapsto u+1/R$ adds $1/R$ to the even part and leaves the
odd part fixed. Hence, chord by chord:
\begin{itemize}[itemsep=2pt]
\item[(even)] the harmonic \emph{mean} of the segments (the semi-latus
rectum, by Askwith's theorem) obeys $1/\lambda'=1/\lambda+1/R$;
\item[(odd)] the harmonic \emph{conjugate} of the focus (the directrix
point, by Proposition~\ref{prop:directrix-conj}) is unchanged: for each
chord direction,
$1/\sigma(r_1)-1/\sigma(r_2)=1/r_1-1/r_2$, so the conjugate
$q=2/(1/r_1-1/r_2)$ is the same before and after the lens shift; hence
$\delta'=\delta$ and the image conic has the \emph{same} directrix.
\end{itemize}
Together these recover the parameter laws \eqref{eq:cc-transf}, including
$1/e'=\delta/\lambda'=1/e+\delta/R$.
\end{corollary}

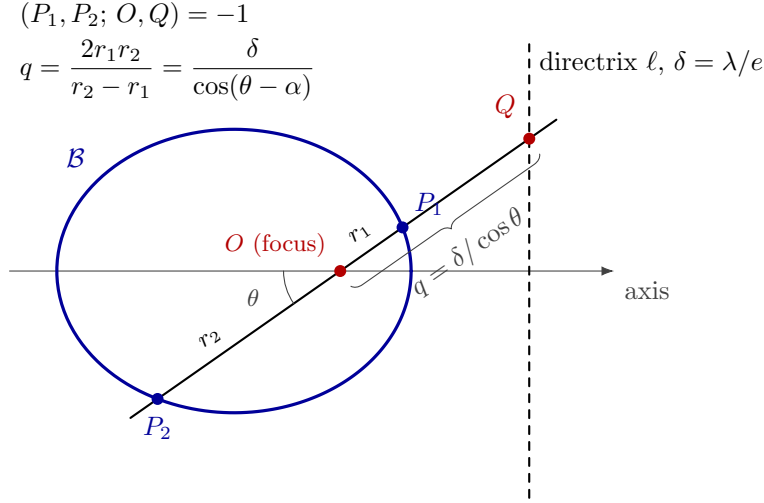
\begin{figure}[htbp]
\centering
\begin{tikzpicture}[scale=1.25,>=Latex,line cap=round,line join=round,font=\small,
   conic/.style={very thick}]
\def\dlt{2}
\draw[->,darkgray] (-3.5,0)--(2.9,0) node[below right]{axis};
\draw[thick,dashed] (\dlt,-2.4)--(\dlt,2.4);
\node[right] at (\dlt,2.2) {directrix $\ell$,\ $\delta=\lambda/e$};
\draw[conic,blue!60!black]
  plot[domain=0:360,samples=240,variable=\t]
  ({1.2/(1+0.6*cos(\t))*cos(\t)},{1.2/(1+0.6*cos(\t))*sin(\t)});
\node[blue!60!black] at (-2.8,1.2) {$\mathcal{B}$};
\coordinate (O)  at (0,0);
\coordinate (P1) at (0.6591,0.4615);
\coordinate (P2) at (-1.9330,-1.3535);
\coordinate (Q)  at (2,1.4004);
\coordinate (A)  at (-2.2197,-1.5542);   
\coordinate (B)  at (2.2867,1.6012);     
\draw[thick] (A) -- (B);
\draw[darkgray] (-0.6,0) arc[start angle=180,end angle=215,radius=0.6];
\node[darkgray,font=\footnotesize] at (-0.92,-0.28) {$\theta$};
\path (O)  -- (P1) node[midway,sloped,above=1pt,font=\footnotesize]{$r_1$};
\path (P2) -- (O)  node[pos=0.35,sloped,above=1pt,font=\footnotesize]{$r_2$};
\draw[decorate,decoration={brace,mirror,amplitude=4pt,raise=7pt},gray!50!black]
  (O) -- (Q) node[midway,sloped,below=13pt]{$q=\delta/\cos\theta$};
\fill[blue!60!black] (P1) circle (1.8pt);
\fill[blue!60!black] (P2) circle (1.8pt);
\fill[red!70!black]  (O)  circle (1.8pt);
\fill[red!70!black]  (Q)  circle (1.8pt);
\node[blue!60!black,above right=1pt] at (P1) {$P_1$};
\node[blue!60!black,below=3pt] at (P2) {$P_2$};
\node[red!70!black,font=\footnotesize,above left=2pt] at (O) {$O$ (focus)};
\node[red!70!black,above=4pt,xshift=-9pt] at (Q) {$Q$};
\node[anchor=south west,align=left] at (-3.5,1.7)
  {$(P_1,P_2;\,O,Q)=-1$\\[3pt]
   $q=\dfrac{2r_1r_2}{r_2-r_1}=\dfrac{\delta}{\cos(\theta-\alpha)}$};
\end{tikzpicture}
\caption{The directrix point as harmonic conjugate
(Proposition~\ref{prop:directrix-conj}), drawn with $\alpha=0$ in the
elliptic case. A focal chord of $\mathcal{B}$ in direction $\theta$ meets
the conic at $P_1$ and $P_2$, with focal radii $r_1$ and $r_2$, and its
line meets the directrix $\ell$ at $Q$ (a point at infinity when the chord
is parallel to $\ell$); the angle $\theta$ is marked by the equal vertical
angle on the $P_2$-side for visual clarity. The four collinear points, in
the order $P_2$, $O$, $P_1$, $Q$, form a harmonic range: the focus $O$
divides the chord internally and the directrix point $Q$ externally, in
the same ratio, so $(P_1,P_2;\,O,Q)=-1$. The signed distance
$q=2/(1/r_1-1/r_2)=\delta/\cos(\theta-\alpha)$ is determined by the odd
part of the reciprocal-radius decomposition \eqref{eq:parity}; since the
lens shift $u\mapsto u+1/R$ preserves that odd part, $Q$, and with it the
directrix, is unchanged by $f_R$ (Corollary~\ref{cor:chord-parity}).}
\label{fig:directrix-conj}
\end{figure}

\begin{remark}[Relation to the harmonic range on rays]
Proposition~\ref{prop:harmonic} below, the harmonic range
$\bigl(O,p;\,f_R(p),f_R^{-1}(p)\bigr)=-1$ on each ray, is the same
algebra at a single point: it unpacks to
$|p|=\mathrm{H}\bigl(|f_R(p)|,\,|f_R^{-1}(p)|\bigr)$, since the forward
and backward maps perturb $1/|p|$ symmetrically by $\pm 1/R$.
Structurally this is identical to a focal chord's symmetric perturbation
of $1/\lambda$ by $\pm\cos(\theta-\alpha)/\delta$ in
\eqref{eq:chord-pm}: harmonic ranges are exactly symmetric pairs in a
reciprocal coordinate, and both phenomena are instances of that one
fact.
\end{remark}

\begin{remark}[The Self-Directrix Theorem, chord by chord]
A line $\ell$ not through $O$ is the $\lambda=\infty$ member of its
fixed-directrix pencil. In the signed focal polar form
$1/\rho=\cos\theta/d$ of the remark in Section~\ref{sec:image-arc}
(coordinates with $\ell=\{x=d\}$), each full line through $O$ meets
$\ell$ once, represented twice with opposite signed radii:
$1/\rho(\theta)+1/\rho(\theta+\pi)=0$. The ``harmonic mean of the chord
segments'' is infinite, consistent with $\lambda=\infty$. Applying the
lens shift to both reciprocals gives
\[
  \frac{1}{r_1'}+\frac{1}{r_2'}=\frac{2}{R},
\]
so \emph{every} focal chord of the image conic has harmonic mean exactly
$R$. By Askwith's theorem the image has semi-latus rectum $R$: the fixed
latus rectum of Theorem~\ref{thm:directrix}, verified chord by chord,
appears as the lens-shifted image of the trivial cancellation of a
line's signed reciprocal intersections.
Proposition~\ref{prop:lambda-chord} degenerates correspondingly:
$\lambda'=\sigma(\infty)=R$. At the opposite formal limit $e=0$ (outside
the finite-directrix focus--directrix setting, since $\delta=\lambda/e$
is then infinite), every focal
chord of a circle centered at $O$ is a diameter with equal segments, the
harmonic mean is the radius itself, and
Proposition~\ref{prop:lambda-chord} reads
$\lambda'=\sigma(\rho)=R\rho/(R+\rho)$: the concentric-circle law of
Corollary~\ref{cor:circles} below is the $e=0$ instance of the
Confocal--Codirectrix parameter law seen through the harmonic mean.
\end{remark}

\section{The image of a circle}\label{sec:circle-image}

The Self-Directrix Theorem turns lines into conics; it is natural to ask about circles. The
answer is a trichotomy. A circle centered at $O$ stays a circle (Corollary~\ref{cor:circles}). A
circle \emph{through} $O$ becomes not a conic but a rational (genus~$0$) \emph{circular quartic},
the inverse of one branch of a conchoid of Nicomedes~\cite{lawrence1972} (Proposition~\ref{prop:circle-O},
Figure~\ref{fig:image-circle}). Every \emph{other} circle, neither centered at $O$ nor through
it, maps to a smooth closed curve lying on a circular quartic whose normalization has genus~$1$
(Proposition~\ref{prop:general-circle}, Figure~\ref{fig:general-circle}), of which the first two
cases are the genus-$0$ degenerations.

The genus statements below are meant in the algebraic-geometric sense: the genus is the genus of the
nonsingular projective model, equivalently of the normalization of the plane algebraic carrier. For
background on geometric genus, normalization of singular plane curves, the degree--genus formula, and
branched covers of $\mathbb{P}^1$, see Kirwan~\cite{Kirwan1992}. The computations needed here are
elementary once the special equations arising from $f_R$ have been written down: the
circle-through-$O$ case is rational, while the generic case is birational to a double cover
$\sigma^2=q_4(\rho)$ with $q_4$ squarefree of degree four. Throughout this section the \emph{carrier} of
an image means the full plane algebraic curve containing it; the actual image is in general only a real
arc or loop on its carrier, which may in addition carry singular or complex points absent from the image.

\subsection{Circles centered at the origin}

\begin{corollary}[Concentric circles gain curvature $1/R$]\label{cor:circles}
$f_R$ carries the circle $|z|=\rho$ to the circle $|w|=R\rho/(R+\rho)$, and the curvatures
(reciprocal radii, i.e.\ $1/\mathrm{radius}$, not signed or Euclidean curvature) of these concentric circles satisfy $1/|w|=1/|z|+1/R$: every circle
centered at $O$ has its curvature increased by exactly $1/R$, the cone curvature. The nested
family of all origin-centered circles is mapped into the same concentric family, with radii compressed into $(0,R)$.
\end{corollary}

\subsection{Circles through the origin}

\begin{proposition}[Image of a circle through $O$]\label{prop:circle-O}
Let $\mathcal C$ be the circle through $O$ of diameter $D$ with axis along $\alpha$, parametrized by
$\rho=D\cos(\theta-\alpha)$ for $|\theta-\alpha|\le\pi/2$, the two endpoint directions $\theta-\alpha=\pm\pi/2$ both representing the single point $O$. Then
\[
  |f_R(z)|\;=\;R-\frac{R}{1+(D/R)\cos(\theta-\alpha)}\qquad(z\in\mathcal C),
\]
so the \emph{radial deficit} $R-|f_R(z)|$, regarded as a polar radius in the direction $\theta$,
is the focal radius of a conic with focus $O$, axis $\alpha$, semi-latus rectum $R$, and
eccentricity $D/R$ (an \emph{auxiliary} focal-polar curve, not the image $f_R(\mathcal C)$ itself), taken on the angular interval $\cos(\theta-\alpha)\ge0$ swept by $\mathcal C$. In Cartesian coordinates, taking $\alpha=0$, the image lies on the rational quartic
\begin{equation}\label{eq:circimg}
  D^2x^2(x^2+y^2)=R^2\bigl(Dx-(x^2+y^2)\bigr)^2 ,
\end{equation}
a \emph{circular} quartic (it passes through each of the two circular points at infinity with multiplicity one; a quartic is \emph{bicircular} when both circular points occur with multiplicity two) but
not a \emph{bicircular} one, so it is not one of the standard bicircular quartics such as a
lima\c{c}on, Cassini oval, or Cartesian oval; moreover the selected image is not a conic. This quartic is the \emph{algebraic carrier} of the image, not the image
set itself: clearing the radical in $Dx\sqrt{x^2+y^2}=R\bigl(Dx-(x^2+y^2)\bigr)$ by squaring also admits
the opposite sign. The image proper is the real loop selected by the unsquared
relation, equivalently by
\[
  r=|w|=\frac{DR\cos(\theta-\alpha)}{R+D\cos(\theta-\alpha)},\qquad \cos(\theta-\alpha)\ge0,
\]
a bounded closed $C^2$ loop through $O$, the inverse of the \emph{outer} conchoid branch of Figure~\ref{fig:conchoid-inversion}, while
the opposite sign gives, in signed polar form, the inverse of the \emph{inner} branch,
$r=DR\cos(\theta-\alpha)/(R-D\cos(\theta-\alpha))$. In the swept interval $\cos(\theta-\alpha)\ge0$ the
denominator vanishes and changes sign for $D>R$ (so $r$ must be read there as a signed radius), while
for $D=R$ it vanishes only at $\theta=\alpha$, without changing sign; thus for $D\ge R$ this branch has
a pole, and for $D<R$ it remains bounded. Equivalently, in the
inversion picture $f_R=\iota\circ\tau_R\circ\iota$ of Proposition~\ref{prop:inversion}, $\iota(\mathcal C)$ is
a line, $\tau_R$ sends that line to one branch of a conchoid of Nicomedes (the line offset by the constant
$1/R$ along rays from $O$), and $f_R(\mathcal C)$ is the inverse of that branch.
\end{proposition}

\begin{proof}
On $\mathcal C$, $\rho=D\cos(\theta-\alpha)$, so by \eqref{eq:mag}
$|f_R(z)|=R\rho/(R+\rho)=R-\frac{R^2}{R+D\cos(\theta-\alpha)}
=R-\frac{R}{1+(D/R)\cos(\theta-\alpha)}$, and the subtracted term is the focal radius
$R/(1+(D/R)\cos(\theta-\alpha))$ of a conic with semi-latus rectum $R$, axis $\alpha$, and
eccentricity $D/R$. For the Cartesian form take $\alpha=0$ and write
$r=|f_R(z)|=DR\cos\theta/(R+D\cos\theta)$ (the image radius, distinct from the input radius
$\rho=|z|$); then $Rr^2=Dx(R-r)$ with $x=r\cos\theta$,
so $Dxr=R\bigl(Dx-(x^2+y^2)\bigr)$, and squaring (using $r^2=x^2+y^2$) gives
\eqref{eq:circimg}, whose quartic part $(x^2+y^2)\bigl(R^2(x^2+y^2)-D^2x^2\bigr)$ has the single
factor $x^2+y^2$ (circular, not bicircular). That $f_R(\mathcal C)$ is not a conic for any finite $D,R>0$ follows from its local form at $O$.
Taking $\alpha=0$ and $u=\cos\theta\ge0$ near $O$, the image is
$x=DRu^2/(R+Du)$ and $|y|=DRu\sqrt{1-u^2}/(R+Du)$; eliminating $u$ shows that the image set itself, as a graph over the vertical
tangent at $O$, has the local equation $x=y^2/D+|y|^3/(DR)+O(y^4)$. The $|y|^3$ term makes the loop
non-analytic at $O$, whereas a nondegenerate conic is analytic at each of its points, and a
degenerate real conic is a union of lines, a double line, a point, or the empty set, none of which is
this simple closed loop; hence $f_R(\mathcal C)$ is not a conic. The same
conclusion is visible in the inversion factorization: $\iota$ carries the circle-through-$O$ to a
line; the middle map $\tau_R(w)=w+\tfrac1R\,w/|w|$ offsets that line by the fixed radial length
$1/R$, the defining construction of the conchoid of Nicomedes (here its outer branch); and applying $\iota$ once more
yields the inverse of that branch, a quartic curve rather than a conic.
\end{proof}

\begin{remark}[Regularity at $O$]
The graph $x=y^2/D+|y|^3/(DR)+O(y^4)$ obtained in the proof shows that the image loop is $C^2$ but
not $C^3$ at $O$: the even term $|y|^3$ has a continuous second derivative but a jump in the third.
The distinction is between carrier and image: the algebraic carrier \eqref{eq:circimg} is
real-analytic along each of its two sign branches, but the image loop is assembled by gluing the two
sides at $O$, where those branches meet, which is exactly why the selected loop is only $C^2$ and
fails to be real-analytic there.
This is the regularity recorded in Figure~\ref{fig:image-circle}.
\end{remark}

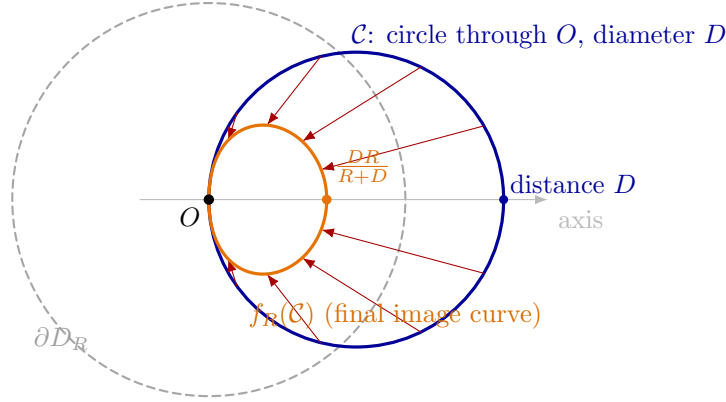
\begin{figure}[htbp]
\centering
\begin{tikzpicture}[scale=1.3,>=Latex,line cap=round,line join=round,font=\small,
   circ/.style={very thick,blue!60!black},
   img/.style={very thick,orange!90!black},
   pull/.style={->,red!65!black,thin}]
  \def\Rr{2.0}\def\Dd{3.0}
  \draw[densely dashed,thick,gray!70] (0,0) circle (\Rr);
  \node[gray!75] at (-1.5,-1.42) {$\partial D_R$};
  \draw[->,gray!55] (-0.7,0)--(3.45,0) node[below right]{axis};
  \draw[circ] (\Dd/2,0) circle (\Dd/2);
  \node[blue!60!black,anchor=west] at (1.35,1.66) {$\mathcal C$: circle through $O$, diameter $D$};
  \foreach \th in {-72,-52,-32,-15,15,32,52,72}{
    \pgfmathsetmacro{\rc}{\Dd*cos(\th)}
    \pgfmathsetmacro{\ri}{\Dd*\Rr*cos(\th)/(\Rr+\Dd*cos(\th))}
    \draw[pull] ({\rc*cos(\th)},{\rc*sin(\th)}) -- ({\ri*cos(\th)},{\ri*sin(\th)});
  }
  \draw[img,domain=-89.3:89.3,samples=220,smooth,variable=\t]
     plot ({\Dd*\Rr*cos(\t)*cos(\t)/(\Rr+\Dd*cos(\t))},
           {\Dd*\Rr*cos(\t)*sin(\t)/(\Rr+\Dd*cos(\t))});
  \node[orange!85!black,anchor=west] at (0.30,-1.18) {$f_R(\mathcal C)$ (final image curve)};
  \fill (0,0) circle (1.5pt) node[below left=-1pt]{$O$};
  \fill[blue!60!black] (\Dd,0) circle (1.3pt) node[above right=-2pt]{distance $D$};
  \pgfmathsetmacro{\fr}{\Dd*\Rr/(\Rr+\Dd)}
  \fill[orange!90!black] (\fr,0) circle (1.4pt);
  \node[orange!85!black,anchor=south west,xshift=-1pt] at (\fr,0.05) {$\tfrac{DR}{R+D}$};
\end{tikzpicture}
\caption{The final curve: the image $f_R(\mathcal C)$ (orange) of a circle $\mathcal C$ through
the origin (blue), drawn for $D=\tfrac32R$. Each point of $\mathcal C$ is pulled inward along its
ray from $O$ (red arrows) by the radial law $|f_R|=R-R/\!\left(1+(D/R)\cos\theta\right)$ of
Proposition~\ref{prop:circle-O}. The image is a closed $C^2$ loop through $O$, tangent there to
$\mathcal C$ (both meet the $y$-axis at $O$), with farthest point $DR/(R+D)$ on the axis; it lies
strictly inside $D_R$ and, unlike the image of a \emph{line}, is not a conic but a quartic,
the inverse of the outer branch of the conchoid of Figure~\ref{fig:conchoid}. (At $O$ the loop is $C^2$ but not $C^3$ for
finite $R$, as the local expansion shows; the full quartic \emph{carrier} has a multiplicity-two
point at $O$, through which the selected real loop passes.) The far point of $\mathcal C$, at
distance $D$ and here outside $\partial D_R$, is pulled all the way in to $DR/(R+D)$.}
\label{fig:image-circle}
\end{figure}

\begin{remark}
Proposition~\ref{prop:circle-O} pinpoints why lines are luckier than circles. After the
inner inversion both become lines, but $f_R$ then offsets radially by $1/R$; a line offset
radially is a conchoid, and only when the inner inversion is omitted, i.e. when one starts
from a line rather than a circle, does the offset act on $1/\rho$ as the clean constant
shift that produces a conic.
\end{remark}

\begin{remark}[Does it have a name? And a special case]
We are not aware of a standard proper name for the quartic $f_R(\mathcal C)$; the classically
named curve naturally appearing in the factorization is the intermediate \emph{conchoid of
Nicomedes}, of whose outer branch $f_R(\mathcal C)$ is the inverse. As an algebraic curve the carrier quartic
is \emph{rational} (geometric genus $0$): being the birational-inversion image of the rational
conchoid, it is parametrized rationally by $t=\tan(\theta/2)$ through
$r=DR\cos\theta/(R\pm D\cos\theta)$ (the signed image radius, the two signs giving the two
branches), and its only finite singularity is a multiplicity-two singularity at
$O$, with tangent cone $x^2=0$. In the case $D=R$, where the unused inner conchoid branch acquires a cusp at the pole $O$, the
image takes the memorable form
\[
  |f_R(z)|=\frac{R\cos(\theta-\alpha)}{1+\cos(\theta-\alpha)}
  =R-\frac{R}{1+\cos(\theta-\alpha)},
\]
that is, $R$ minus the focal radius of the parabola with focus $O$, axis $\alpha$, and
semi-latus rectum $R$.
\end{remark}

\begin{figure}[htbp]
\centering
\begin{tikzpicture}[scale=1.2,>=Latex,line cap=round,line join=round,font=\small,
   outer/.style={very thick,blue!60!black},
   inner/.style={very thick,orange!85!black}]
  \def\aa{1.2}\def\bb{1.6}
  \draw[->,gray!50] (-1.15,0)--(3.15,0);
  \draw[->,gray!50] (0,-3.25)--(0,3.35);
  \draw[thick,densely dashed,gray!75] (\aa,-3.2)--(\aa,3.3);
  \node[gray!10!black,anchor=south] at (\aa,3.34) {ruler line $\ell_0$ (asymptote)};
  \draw[outer,domain=-53:53,samples=160,smooth,variable=\x]
     plot ({\aa+\bb*cos(\x)},{\aa*tan(\x)+\bb*sin(\x)});
  \draw[inner,domain=-66:66,samples=260,smooth,variable=\x]
     plot ({\aa-\bb*cos(\x)},{\aa*tan(\x)-\bb*sin(\x)});
  \def\tt{50}
  \pgfmathsetmacro{\Qx}{\aa}\pgfmathsetmacro{\Qy}{\aa*tan(\tt)}
  \pgfmathsetmacro{\Ppx}{\aa+\bb*cos(\tt)}\pgfmathsetmacro{\Ppy}{\aa*tan(\tt)+\bb*sin(\tt)}
  \pgfmathsetmacro{\Pmx}{\aa-\bb*cos(\tt)}\pgfmathsetmacro{\Pmy}{\aa*tan(\tt)-\bb*sin(\tt)}
  \draw[thin,gray!65] (0,0)--(\Ppx,\Ppy);
  \draw[decorate,decoration={brace,amplitude=3pt,raise=1.5pt},blue!55!black]
     (\Qx,\Qy)--(\Ppx,\Ppy) node[midway,above left=0pt]{$b_0$};
  \draw[decorate,decoration={brace,amplitude=3pt,raise=1.5pt},orange!75!black]
     (\Pmx,\Pmy)--(\Qx,\Qy) node[midway,below right=-1pt]{$b_0$};
  \draw[decorate,decoration={brace,amplitude=4pt,mirror,raise=2pt},gray!60]
     (0,0)--(\aa,0) node[midway,below=5pt]{$a_0$};
  \fill (0,0) circle (1.5pt) node[below left=-1pt]{$O$ (pole)};
  \fill (\aa,0) circle (1.0pt) node[above right=-2pt]{$F$};
  \fill (\Qx,\Qy) circle (1.2pt) node[left=1pt]{$Q$};
  \fill[blue!60!black] (\Ppx,\Ppy) circle (1.4pt) node[right=1pt]{$P_{+}$};
  \fill[orange!85!black] (\Pmx,\Pmy) circle (1.4pt) node[right=1pt]{$P_{-}$};
  \node[blue!60!black,anchor=west] at (1.45,-1.75) {$r=a_0\sec\theta+b_0$};
  \node[orange!85!black,anchor=east] at (-0.5,-1.05) {$r=a_0\sec\theta-b_0$};
  \node[orange!70!black,anchor=east,align=center] at (-0.30,0.95){loop\\$(b_0>a_0)$};
\end{tikzpicture}
\caption{The conchoid of Nicomedes, $r=a_0\sec\theta\pm b_0$, generated from a pole $O$, a ruler
line $\ell_0$ at distance $a_0$, and a fixed offset $b_0$: each ray from $O$ meets $\ell_0$ at $Q$,
and the two curve points $P_{\pm}$ lie at distance $b_0$ from $Q$ along that ray. This is exactly
the curve produced inside the factorization $f_R=\iota\circ\tau_R\circ\iota$ of
Proposition~\ref{prop:circle-O}: the radial offset map $\tau_R$ carries the line
$\iota(\mathcal C)$, at distance $a_0=1/D$, to this conchoid with offset $b_0=1/R$, and the image
$f_R(\mathcal C)$ of a circle through $O$ is its inverse $\iota$. The inner branch (orange) has
a loop, a cusp, or no singular point at the pole according as $b_0>a_0$, $b_0=a_0$, or $b_0<a_0$, that is,
as $D>R$, $D=R$, or $D<R$; the looped case $b_0>a_0$ is drawn. Of the two branches, only the outer ($+b_0$) one is inverted to the final oval, as shown in Figure~\ref{fig:conchoid-inversion}. The ruler line $\ell_0$ is the common
asymptote of both branches.}
\label{fig:conchoid}
\end{figure}
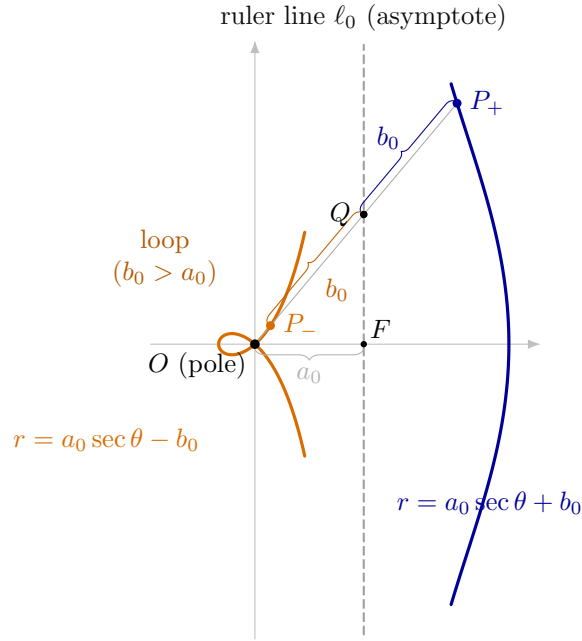

\begin{figure}[htbp]
\centering
\begin{tikzpicture}[scale=1.5,>=Latex,line cap=round,line join=round,font=\small,
   outer/.style={very thick,blue!60!black},
   inner/.style={thick,gray!80},
   oval/.style={very thick,orange!90!black},
   inv/.style={->,green!45!black,semithick}]
  \def\aa{0.33333}\def\bb{0.5}
  \draw[densely dashed,thick,gray!70] (0,0) circle (2);
  \node[black!75!gray,anchor=north] at (0,-2.12) {$\partial D_R$ (radius $R$)};
  \draw[densely dotted,semithick,teal!75] (0,0) circle (1);
  \draw[thin,teal!70] (-0.71,0.71) -- (-1.48,1.12);
  \node[teal!55!black,font=\footnotesize,anchor=east,align=center] at (-1.48,1.15)
     {unit circle $|z|=1$\\(circle of $\iota$)};
  \draw[inner,domain=-48:48,samples=120,variable=\t]
     plot ({\aa-\bb*cos(\t)},{\aa*tan(\t)-\bb*sin(\t)});
  \draw[thin,gray!65] (-0.16,-0.02) -- (-1.02,-0.92);
  \node[black!75!gray,font=\footnotesize,anchor=north,align=center] at (-0.8,-0.96)
     {inner branch\\(not used)};
  \draw[outer,domain=-70:70,samples=180,variable=\t]
     plot ({\aa+\bb*cos(\t)},{\aa*tan(\t)+\bb*sin(\t)});
  \node[blue!60!black,anchor=west] at (0.16,1.72) {outer branch (transformed)};
  \draw[oval,domain=-89:89,samples=200,variable=\t]
     plot ({cos(\t)*cos(\t)/(\aa+\bb*cos(\t))},{cos(\t)*sin(\t)/(\aa+\bb*cos(\t))});
  \draw[thin,orange!80!black] (0.98,-0.57) -- (1.28,-0.88);
  \node[orange!50!black,anchor=west] at (1.28,-0.92) {final oval $f_R(\mathcal C)$};
  \foreach \t in {0,64}{
    \pgfmathsetmacro{\ro}{\aa/cos(\t)+\bb}
    \pgfmathsetmacro{\rv}{1/\ro}
    \draw[inv] ({\ro*cos(\t)},{\ro*sin(\t)}) -- ({\rv*cos(\t)},{\rv*sin(\t)});
    \fill[blue!60!black] ({\ro*cos(\t)},{\ro*sin(\t)}) circle (0.9pt);
    \fill[orange!90!black] ({\rv*cos(\t)},{\rv*sin(\t)}) circle (0.9pt);
  }
  \fill[teal!70] (0.6667,0.7454) circle (0.9pt);
  \fill[teal!70] (0.6667,-0.7454) circle (0.9pt);
  \fill (0,0) circle (1.4pt) node[below right=-2pt]{$O$};
  \node[blue!60!black,font=\scriptsize,anchor=south east] at (0.84,0.03){$P$};
  \node[orange!50!black,font=\scriptsize,anchor=south west] at (1.21,0.03){$P'\!=\!\iota(P)$};
  \node[green!40!black,font=\scriptsize,anchor=south] at (1.03,0.10){$\iota$};
\end{tikzpicture}
\caption{Which branch is inverted, and how it becomes the oval, drawn with the same
reference circle $\partial D_R$ as Figure~\ref{fig:image-circle}. The factorization
$f_R=\iota\circ\tau_R\circ\iota$ ends with inversion $\iota$ in the unit circle $|z|=1$ (dotted
teal); the conchoid (the output of $\tau_R$) and its inverse oval are overlaid in the one plane, related
by this final inversion. The overlay is schematic: $\iota$ inverts radii about $|z|=1$, so the
conchoid and its image oval share only this normalized scale, not a common length unit. Of the two branches of the conchoid $r=a_0\sec\theta\pm b_0$ ($a_0=1/D$, $b_0=1/R$,
$D=\tfrac32R$, so $b_0>a_0$), the outward push $\tau_R$ produces only the \emph{outer} branch
$r=a_0\sec\theta+b_0$ (blue) from the line $\iota(\mathcal C)=\{x=a_0\}$; the looped inner branch
(gray, at $O$) is never used. Inversion sends each point $P$ of the outer branch to
$P'=\iota(P)$ on the same ray with $|OP|\,|OP'|=1$ (green), carrying the outer branch onto the
closed oval $f_R(\mathcal C)$ (orange), $r=|w|=\cos\theta/(a_0+b_0\cos\theta)
=DR\cos\theta/(R+D\cos\theta)$. The two circles play different roles and should not be confused:
the \emph{dashed} circle is $\partial D_R$ of radius $R$, exactly as in
Figure~\ref{fig:image-circle}, and the oval lies strictly inside it in both figures; the
\emph{dotted} circle is the unit circle of inversion (radius $1$, here inside $\partial D_R$ since
$R=2$ in the figure's scale), which $\iota$ fixes pointwise, so the outer branch meets its inverse
image oval where it crosses this circle (teal dots). The conchoid vertex at $r=a_0+b_0$ inverts to the
oval's farthest point $1/(a_0+b_0)=DR/(R+D)$, which is $<R$, so the oval stays inside $\partial D_R$.}
\label{fig:conchoid-inversion}
\end{figure}
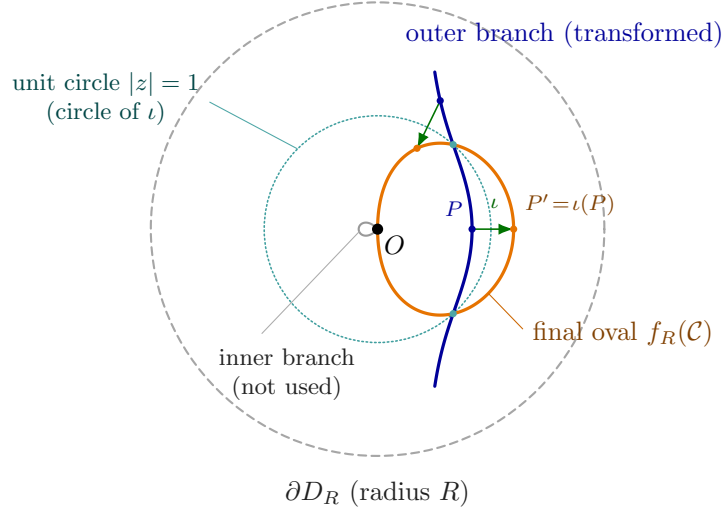

\begin{remark}
Whether the final oval reaches the unit circle of inversion in Figure~\ref{fig:conchoid-inversion} is
parameter-dependent. The outer conchoid branch has vertex at $r=a_0+b_0=1/D+1/R$, so it crosses the unit
circle $|z|=1$, and the oval meets it there, exactly when $1/D+1/R\le1$; the drawn example has
$1/D+1/R=\tfrac{5}{6}$.
\end{remark}

\subsection{Circles in general position}

\begin{proposition}[Image of a general circle]\label{prop:general-circle}
Let $\mathcal C$ be a circle whose center is at distance $c>0$ from $O$, of radius $a>0$, with
$c\neq a$ (so $O\notin\mathcal C$), and write $d=c^2-a^2$. Then $f_R(\mathcal C)$ is a smooth
closed \emph{real} curve in $D_R$, a topological circle. Its \emph{algebraic carrier} is the
\emph{circular quartic} (in coordinates with the center on the positive axis)
\begin{equation}\label{eq:gencirc}
  \tfrac{4}{R^2}\,(x^2+y^2)\,(cx-d)^2
  =\Bigl[\bigl(1+\tfrac{d}{R^2}\bigr)(x^2+y^2)-2cx+d\Bigr]^2 ,
\end{equation}
which is neither a circle nor a conic and whose normalization has geometric genus $1$. The asserted
smoothness is a property of the parametrized image $f_R(\mathcal C)$, not of the full real locus of
the squared carrier \eqref{eq:gencirc}, which may in addition carry isolated singular points
(see the genus-drop remark following Lemma~\ref{lem:gencirc-genus}). The image itself is the bounded real curve selected from the carrier \eqref{eq:gencirc} by the
\emph{unsquared} relation
\begin{equation}\label{eq:gencirc-unsq}
  \bigl(1+\tfrac{d}{R^2}\bigr)(x^2+y^2)-2cx+d+\tfrac{2}{R}\,(cx-d)\sqrt{x^2+y^2}=0.
\end{equation}
Equivalently, $|w|=R\rho/(R+\rho)$, where $\rho>0$ is the radius of the corresponding point of
$\mathcal C$, a root of $\rho^2-2c\rho\cos\theta+d=0$. When $c>a$, a ray inside the angular sector that $\mathcal C$ subtends from $O$ has two such roots, coinciding on the two boundary tangent rays. If $O$ lies outside $\mathcal C$ ($c>a$), the image is a closed
loop not enclosing $O$, contained in the annulus
$\tfrac{R(c-a)}{R+c-a}\le|w|\le\tfrac{R(c+a)}{R+c+a}$ and spanning only the directions $\mathcal C$
subtends from $O$; if $O$ lies inside $\mathcal C$ ($c<a$), the image is a closed loop encircling
$O$, contained in the analogous annulus $\tfrac{R(a-c)}{R+a-c}\le|w|\le\tfrac{R(a+c)}{R+a+c}$.
\end{proposition}

\begin{proof}
By rotation-equivariance put the center on the positive axis, $\mathcal C=\{|z-c|=a\}$, i.e.
$\rho^2-2c\rho\cos\theta+d=0$; in $u=1/\rho$ this is $d\,u^2-2c\cos\theta\,u+1=0$. The lens
identity sends $u\mapsto u'=u+1/R$; substituting $u=u'-1/R$, writing $u'=1/r$ with
$r=|w|=\sqrt{x^2+y^2}$ and $\cos\theta=x/r$, and multiplying through by $r^2$ gives the unsquared
relation \eqref{eq:gencirc-unsq}; squaring once to clear $r$ gives the carrier \eqref{eq:gencirc}.
Its degree-four part contains exactly one factor $x^2+y^2$, so the carrier is circular but not
bicircular; that it is not a circle or conic follows from the genus computation below, which gives
normalized genus $1$ whenever $c>0$, $a>0$, and $c\neq a$. The squaring is the only place a
second branch can enter, so \eqref{eq:gencirc-unsq} selects $f_R(\mathcal C)$. That
$f_R(\mathcal C)$ is a smooth simple closed curve is immediate: $\mathcal C$ avoids $O$, and on
$\mathbb{C}\setminus\{0\}$ the map $f_R$ is a smooth diffeomorphism with nonsingular derivative (its
radial and tangential stretch factors $R^2/(R+\rho)^2$ and $R/(R+\rho)$ are both positive), so it
carries the smooth circle $\mathcal C$ to a smooth embedded loop. Here ``smooth'' refers
to this selected image curve, cut out by \eqref{eq:gencirc-unsq}, not to the full real locus of the
carrier \eqref{eq:gencirc}: the carrier may carry additional real or complex singular points where the
two sign branches meet, and these are resolved by the normalization computed next.

The carrier therefore has normalized genus $1$ by Lemma~\ref{lem:gencirc-genus} below, so it is
irreducible and is neither a circle nor a conic. Geometric genus is understood throughout over the
complex projective normalization, for which we follow Kirwan~\cite{Kirwan1992}; the real image
$f_R(\mathcal C)$ is by contrast a smoothly embedded topological circle. The two regimes and the radius range follow because
$|w|=R\rho/(R+\rho)$ increases with $\rho$, and a ray from $O$ meets $\mathcal C$ once when $O$ is
interior ($c<a$) and in a near/far pair over a directional sector when $O$ is exterior ($c>a$).
\end{proof}

\begin{lemma}[Genus of the generic-circle carrier]\label{lem:gencirc-genus}
For $c>0$, $a>0$, $c\neq a$, with $d=c^2-a^2$, the carrier quartic \eqref{eq:gencirc} is irreducible,
and its complex projective normalization is isomorphic to the smooth projective model of
\[
  \sigma^2=\bigl((a+c)^2-\rho^2\bigr)\bigl(\rho^2-(a-c)^2\bigr).
\]
The quartic in $\rho$ on the right has the four distinct roots $\rho=\pm(a+c),\pm(a-c)$, so this model
is an elliptic curve; hence the carrier has geometric genus $1$, and in particular is neither a circle
nor a conic.
\end{lemma}

\begin{proof}
Work on the lifted curve in the variables $(x,y,r)$ cut out by $r^2=x^2+y^2$ together with the
unsquared relation \eqref{eq:gencirc-unsq}, of which the carrier \eqref{eq:gencirc} is the square.
\emph{(i) Birationality to the carrier.} Away from $cx=d$, the relation \eqref{eq:gencirc-unsq} solves
for the radius as a rational function of $(x,y)$,
\[
  r=-\frac{R\bigl[(1+d/R^2)(x^2+y^2)-2cx+d\bigr]}{2(cx-d)}\qquad(cx\neq d),
\]
so the projection $(x,y,r)\mapsto(x,y)$ onto the carrier has a rational inverse and is birational; the
squaring that produced \eqref{eq:gencirc} adds no generic information, and the lifted curve follows the
single selected sign branch. \emph{(ii) Recovering the circle parameter.} On the lifted curve the
original radius is recovered rationally by $1/\rho=1/r-1/R$, whence $z=\rho\,(x,y)/r$; conversely,
parametrizing $\mathcal C$ by $z(\varphi)=c+ae^{i\varphi}$ gives $\rho^2=c^2+a^2+2ac\cos\varphi$, and
with $s=\cos\varphi=(\rho^2-c^2-a^2)/(2ac)$ and $\eta=\sin\varphi$, eliminating $s$ from
$\eta^2=1-s^2$ gives
\[
  \eta^2=\frac{\bigl((a+c)^2-\rho^2\bigr)\bigl(\rho^2-(a-c)^2\bigr)}{4a^2c^2};
\]
after the harmless vertical rescaling $\sigma=2ac\,\eta$ this is the displayed model
$\sigma^2=\bigl((a+c)^2-\rho^2\bigr)\bigl(\rho^2-(a-c)^2\bigr)$. \emph{(iii) Isomorphism
and irreducibility.} These rational maps are mutually inverse on dense open subsets, so they induce a
birational map between the normalization of the carrier and the smooth projective model of
$\sigma^2=(\text{quartic})$; a birational map of smooth projective curves is an isomorphism. Because the
double cover $\sigma^2=(\text{quartic})$ is irreducible (the squarefree quartic
on the right is not a square in $\mathbb{C}(\rho)$) and the carrier
contains a dense open subset birational to it, the carrier is irreducible. \emph{(iv) Genus.} The
quartic $\bigl((a+c)^2-\rho^2\bigr)\bigl(\rho^2-(a-c)^2\bigr)$ is squarefree precisely because $c>0$,
$a>0$, $c\neq a$ force its four roots to be distinct, and a double cover of $\mathbb{P}^1$ branched at
four distinct points has genus $1$. The singular points created on the plane quartic by squaring are
resolved by the normalization and do not affect the geometric genus.
\end{proof}

\begin{remark}[Where the genus drops]
The squared carrier \eqref{eq:gencirc} carries both sign branches, and they cross at its finite
singular points, which for $R^2+d\neq0$ are
\[
  x=\frac{d}{c},\qquad y^2=\frac{d\,(R^2a^2-d^2)}{c^2\,(R^2+d)}
\]
(equivalently, at squared radius $x^2+y^2=R^2d/(R^2+d)$). When distinct, these are two ordinary double points
(nodes), each of $\delta$-invariant $1$; according to the parameters they form a real pair or a
complex-conjugate pair. Either way their total $\delta$-contribution is $2$, accounting for the drop
from the arithmetic genus $3$ of a plane quartic to the geometric genus $1$ of the normalization. On the
locus $d^2=R^2a^2$, where $y=0$, the two points coalesce into a single \emph{non-ordinary} double point
with a double tangent line; the normalization of Lemma~\ref{lem:gencirc-genus} still gives geometric genus $1$ here (the four roots $\pm(a+c),\pm(a-c)$ remain distinct), so the total $\delta$-contribution at this coalesced point is again $2$; and the singularities recede to infinity as $R^2+d\to0$. They are singularities of the squared carrier, not of the
parametrized image $f_R(\mathcal C)$, which is smooth throughout; a real node need not even lie on
the image, since the image is only the bounded arc of the selected branch \eqref{eq:gencirc-unsq}.
The nodes are thus an artifact of clearing the radical in \eqref{eq:gencirc-unsq}.
\end{remark}

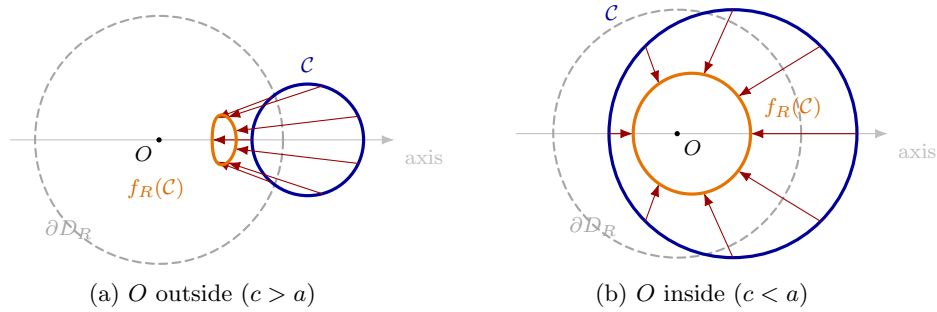
\begin{figure}[htbp]
\centering
\begin{tikzpicture}[scale=0.82,>=Latex,line cap=round,line join=round,font=\small,
   circ/.style={very thick,blue!60!black},
   img/.style={very thick,orange!90!black},
   pull/.style={->,red!60!black,thin}]
  \def\RR{2.0}\def\cc{2.4}\def\aa{0.9}
  \draw[->,gray!50] (-2.4,0)--(3.8,0) node[below right,font=\scriptsize]{axis};
  \draw[densely dashed,thick,gray!70] (0,0) circle (\RR);
  \node[gray!70,font=\scriptsize] at (-1.45,-1.45) {$\partial D_R$};
  \draw[circ] (\cc,0) circle (\aa);
  \node[blue!60!black,font=\scriptsize] at (\cc,1.22) {$\mathcal C$};
  \foreach \f in {25,75,130,180,230,285,335}{
    \pgfmathsetmacro{\rr}{sqrt(\cc*\cc+\aa*\aa+2*\aa*\cc*cos(\f))}
    \pgfmathsetmacro{\zx}{\cc+\aa*cos(\f)}\pgfmathsetmacro{\zy}{\aa*sin(\f)}
    \draw[pull] (\zx,\zy)--({\RR*\zx/(\RR+\rr)},{\RR*\zy/(\RR+\rr)});
  }
  \draw[img,domain=0:360,samples=170,smooth,variable=\f]
    plot ({\RR*(\cc+\aa*cos(\f))/(\RR+sqrt(\cc*\cc+\aa*\aa+2*\aa*\cc*cos(\f)))},
          {\RR*\aa*sin(\f)/(\RR+sqrt(\cc*\cc+\aa*\aa+2*\aa*\cc*cos(\f)))});
  \node[orange!85!black,font=\scriptsize,anchor=east] at (0.58,-0.78) {$f_R(\mathcal C)$};
  \fill (0,0) circle (1.2pt) node[below left=-2pt,font=\scriptsize]{$O$};
  \node[font=\footnotesize] at (0.7,-2.5) {(a) $O$ outside ($c>a$)};
\end{tikzpicture}\hspace{7mm}
\begin{tikzpicture}[scale=0.82,>=Latex,line cap=round,line join=round,font=\small,
   circ/.style={very thick,blue!60!black},
   img/.style={very thick,orange!90!black},
   pull/.style={->,red!60!black,thin}]
  \def\RR{2.0}\def\cc{0.9}\def\aa{2.0}
  \draw[->,gray!50] (-2.6,0)--(3.4,0) node[below right,font=\scriptsize]{axis};
  \draw[densely dashed,thick,gray!70] (0,0) circle (\RR);
  \node[gray!70,font=\scriptsize] at (-1.35,-1.5) {$\partial D_R$};
  \draw[circ] (\cc,0) circle (\aa);
  \node[blue!60!black,font=\scriptsize] at (-1.05,1.95) {$\mathcal C$};
  \foreach \f in {0,45,90,135,180,225,270,315}{
    \pgfmathsetmacro{\rr}{sqrt(\cc*\cc+\aa*\aa+2*\aa*\cc*cos(\f))}
    \pgfmathsetmacro{\zx}{\cc+\aa*cos(\f)}\pgfmathsetmacro{\zy}{\aa*sin(\f)}
    \draw[pull] (\zx,\zy)--({\RR*\zx/(\RR+\rr)},{\RR*\zy/(\RR+\rr)});
  }
  \draw[img,domain=0:360,samples=170,smooth,variable=\f]
    plot ({\RR*(\cc+\aa*cos(\f))/(\RR+sqrt(\cc*\cc+\aa*\aa+2*\aa*\cc*cos(\f)))},
          {\RR*\aa*sin(\f)/(\RR+sqrt(\cc*\cc+\aa*\aa+2*\aa*\cc*cos(\f)))});
  \node[orange!85!black,font=\scriptsize,anchor=west] at (1.24,0.42) {$f_R(\mathcal C)$};
  \fill (0,0) circle (1.2pt) node[below right=-2pt,font=\scriptsize]{$O$};
  \node[font=\footnotesize] at (0.4,-2.6) {(b) $O$ inside ($c<a$)};
\end{tikzpicture}
\caption{The image of a generic circle $\mathcal C$ (blue) under $f_R$: a smooth closed curve
(orange) lying on the circular quartic \eqref{eq:gencirc}, whose normalization has geometric genus $1$. Each point of $\mathcal C$ is pulled inward along its ray from $O$ (red). \textbf{(a)} With
$O$ outside $\mathcal C$ the image is a loop not enclosing $O$, confined to the cone of directions
$\mathcal C$ subtends from $O$. \textbf{(b)} With $O$ inside $\mathcal C$ the image is a loop
encircling $O$. Unlike the centered circle (a circle) and the circle through $O$ (the rational
quartic of Figure~\ref{fig:image-circle}), the generic carrier has geometric genus $1$.}
\label{fig:general-circle}
\end{figure}

\begin{remark}
The two earlier cases are the genus-$0$ degenerations of \eqref{eq:gencirc}. As $c\to a$ (circle
through $O$) the roots $\pm(a-c)$ collide at $\rho=0$, the quartic acquires a singular point and
its genus drops to $0$, the rational quartic of Proposition~\ref{prop:circle-O}. At $c=0$ the
circle is $|z|=a$ and maps to the circle $|w|=Ra/(R+a)$ (Corollary~\ref{cor:circles}); here the squared
carrier \eqref{eq:gencirc} factors into sign branches, and the actual image is the single branch
$|w|=Ra/(R+a)$. The
Self-Directrix line case is the further limit $a,c\to\infty$ with $c-a$ fixed.
\end{remark}

\begin{table}[htbp]
\centering
\caption{Image under the cone projection $f_R$ of the basic input curves, collecting the results of Sections~\ref{sec:geometry} and~\ref{sec:circle-image}.}
\label{tab:curve-images}
{\small
\begin{tabularx}{\textwidth}{@{}l L L@{}}
\hline
Input curve & Image under $f_R$ & Notes \\
\hline
Line at distance $d>0$ & conic arc, focus $O$, $e=R/d$, semi-latus rectum $R$ & directrix is the line itself; ellipse/parabola/hyperbola for $d>R$, $d=R$, $d<R$ \\[4pt]
Focal conic $\mathcal{B}(\lambda,e,\alpha)$ & focal arc $\mathcal{B}(\lambda',e',\alpha)$ & same focus and directrix; $\lambda'=\frac{R\lambda}{R+\lambda}$, $e'=\frac{Re}{R+\lambda}$ \\[4pt]
Circle centered at $O$, radius $a$ & concentric circle of radius $Ra/(R+a)$ & curvature raised by $1/R$ \\[4pt]
Circle through $O$ & closed $C^2$ loop on a rational circular quartic & inverse of one branch of a conchoid of Nicomedes; not $C^3$ at $O$; not a conic \\[4pt]
Generic circle & smooth closed real curve on a circular quartic & neither centered at $O$ nor through $O$; carrier is non-rational (geometric genus $1$) \\
\hline
\end{tabularx}
}
\end{table}

\FloatBarrier

\section{The composition law}\label{sec:semigroup}

With the central geometric result in hand, we turn to the algebraic, projective, and dynamical structure of the family $\{f_R\}_{R>0}$. A consequence of Proposition~\ref{prop:lens} is that this family is closed under composition, and composition adds curvatures.

\begin{theorem}[Composition law]\label{thm:semigroup}
For every $R_1,R_2>0$,
\[
  f_{R_1}\circ f_{R_2} \;=\; f_{R_3}, \qquad \text{where} \qquad \frac{1}{R_3} \;=\; \frac{1}{R_1}+\frac{1}{R_2}.
\]
\end{theorem}

\begin{proof}
The case $z=0$ is immediate. For $z\ne0$, each $f_R$ preserves the ray through $z$, because
\[
  f_R(\rho e^{i\theta})=\frac{\rho}{1+\rho/R}e^{i\theta}.
\]
Let $w=f_{R_2}(z)$ and $u=f_{R_1}(w)$. By the lens identity,
\[
  \frac{1}{|u|}=\frac{1}{|w|}+\frac{1}{R_1}
  =\frac{1}{|z|}+\frac{1}{R_2}+\frac{1}{R_1}
  =\frac{1}{|z|}+\frac{1}{R_3}.
\]
Thus $u$ and $f_{R_3}(z)$ have the same modulus and lie on the same ray from the origin, so they are equal.
\end{proof}

\begin{corollary}\label{cor:abelian}
The family $\{f_R\}_{R>0}$ is an abelian semigroup under composition, isomorphic to $(\mathbb{R}_{>0},+)$ via the curvature parameter $f_R\mapsto\kappa=1/R$. Equivalently, the induced composition law on radii is $R_1*R_2=R_1R_2/(R_1+R_2)$, so it is the curvatures $1/R$, not the radii, that add.
\end{corollary}

\begin{proof}
Under the parameter $\tau=1/R$, Theorem~\ref{thm:semigroup} says that composition becomes addition: $\tau_3=\tau_1+\tau_2$. Hence associativity and commutativity follow from those of addition on $\mathbb{R}_{>0}$.
\end{proof}

\begin{remark}
The semigroup is \emph{not} a group: as $R$ ranges over $(0,\infty)$, so does $1/R$, so there is no identity element of the form $f_R$ (any such identity would need $1/R=0$, impossible for $R>0$). The identity arises as the formal limit $R\to\infty$ (cone of vanishing curvature), and the inverses $f_R^{-1}$ correspond to formal radius $-R$. Both extensions are made precise in the next section.
\end{remark}

\section{The one-parameter partial group of radial M\"obius maps}\label{sec:group}

The maps $f_R$ extend, with appropriate domains, to a one-parameter partial group. The word ``partial'' is important: these maps do not form a group of self-homeomorphisms of one fixed space. For positive curvature they map the whole plane onto a disk, while for negative curvature they map a disk onto the whole plane.

\begin{definition}\label{def:phi}
For $\kappa\in\mathbb{R}$, define the \emph{radial M\"obius map of curvature $\kappa$} by
\[
  \phi_\kappa(z) \;:=\; \frac{z}{1+\kappa\,|z|},
\]
on the set
\[
  \Omega_\kappa \;:=\; \begin{cases}\mathbb{C}, & \kappa\ge 0,\\[2pt] \{z\in\mathbb{C}\,:\,|z|<1/|\kappa|\}, & \kappa<0.\end{cases}
\]
\end{definition}

For $\kappa>0$, the denominator $1+\kappa|z|$ is always positive, so $\phi_\kappa$ is defined on all of $\mathbb{C}$, and equals $f_{1/\kappa}$. For $\kappa=0$, $\phi_0(z)=z/(1+0)=z$, so $\phi_0=\mathrm{id}$. For $\kappa<0$, the denominator vanishes when $|z|=1/|\kappa|$, so we restrict to the open disk of radius $1/|\kappa|$, where the denominator is positive. In all three cases the domain has the single uniform description
\[
  \Omega_\kappa=\{z\in\mathbb{C}:\,1+\kappa|z|>0\},
\]
the set on which the denominator is positive; this is the form used throughout the proof below.

The qualifier ``radial'' is essential: each $\phi_\kappa$ is a M\"obius transformation only in the radial coordinate $|z|$ along each ray from the origin (Section~\ref{sec:cross-ratio}), and is \emph{not} a complex M\"obius transformation of the plane; indeed it is not even holomorphic, since it depends on $|z|$. The term ``radial M\"obius map'' should be read throughout in this raywise sense.

\begin{theorem}[Partial group law]\label{thm:group}
Let
\[
  D_{\kappa_1,\kappa_2}
  :=\{z\in\Omega_{\kappa_2}:\phi_{\kappa_2}(z)\in\Omega_{\kappa_1}\}.
\]
On this composition domain,
\[
  \phi_{\kappa_1}\circ\phi_{\kappa_2} \;=\; \phi_{\kappa_1+\kappa_2}.
\]
Equivalently, the identity holds exactly where the left-hand side is defined; this composition domain is precisely
\[
  D_{\kappa_1,\kappa_2}=\Omega_{\kappa_2}\cap\Omega_{\kappa_1+\kappa_2}
  =\{z:\,1+\kappa_2|z|>0\ \text{and}\ 1+(\kappa_1+\kappa_2)|z|>0\},
\]
so it is always contained in $\Omega_{\kappa_1+\kappa_2}$ but need not equal it. In particular, each $\phi_\kappa$ is a homeomorphism of $\Omega_\kappa$ onto $\Omega_{-\kappa}$, with $\phi_\kappa^{-1}=\phi_{-\kappa}$:
\[
  \phi_\kappa(\Omega_\kappa)=\Omega_{-\kappa},\qquad \phi_\kappa^{-1}=\phi_{-\kappa}.
\]
For $\kappa>0$ this reads $\phi_\kappa:\mathbb{C}\xrightarrow{\ \sim\ }\{|z|<1/\kappa\}$, and for $\kappa<0$ it runs the other way, $\phi_\kappa:\{|z|<1/|\kappa|\}\xrightarrow{\ \sim\ }\mathbb{C}$. Thus $\{\phi_\kappa\}_{\kappa\in\mathbb{R}}$ is a one-parameter partial group, or equivalently a maximal radial flow with varying domains, modeled by addition of the curvature parameter $\kappa$.
\end{theorem}

\begin{proof}
\emph{Step 1 (the lens identity holds for all real $\kappa$).} For $z\ne 0$ in $\Omega_\kappa$, the denominator $1+\kappa|z|$ is positive, so
\[
  |\phi_\kappa(z)| \;=\; \frac{|z|}{1+\kappa|z|} \;>\; 0.
\]
Taking reciprocals,
\[
  \frac{1}{|\phi_\kappa(z)|} \;=\; \frac{1+\kappa|z|}{|z|} \;=\; \frac{1}{|z|} + \kappa.
\]
This is the lens identity with curvature $\kappa$, valid for any sign of $\kappa$.

\emph{Step 2 (rotation equivariance).} As in the case of $f_R$, we have $\arg(\phi_\kappa(z))=\arg(z)$ for $z\ne 0$ in $\Omega_\kappa$, since $\phi_\kappa(z)=z/(1+\kappa|z|)$ multiplies $z$ by the positive real number $1/(1+\kappa|z|)$.

\emph{Step 3 (composition).} For $z\ne 0$ in $D_{\kappa_1,\kappa_2}$, set $w=\phi_{\kappa_2}(z)$ and $u=\phi_{\kappa_1}(w)$. The domain condition says exactly that both of these quantities are defined. Moreover, since $|w|=|z|/(1+\kappa_2|z|)$,
\[
  1+\kappa_1|w|=1+\frac{\kappa_1|z|}{1+\kappa_2|z|}=\frac{1+(\kappa_1+\kappa_2)|z|}{1+\kappa_2|z|};
\]
multiplying through by $1+\kappa_2|z|>0$ and using $1+\kappa_1|w|>0$ gives
\[
  1+(\kappa_1+\kappa_2)|z|=(1+\kappa_2|z|)(1+\kappa_1|w|)>0,
\]
so $z\in\Omega_{\kappa_1+\kappa_2}$. Conversely, the same factorization shows that whenever $1+\kappa_2|z|>0$ and $1+(\kappa_1+\kappa_2)|z|>0$ one has $1+\kappa_1|w|>0$, that is $w\in\Omega_{\kappa_1}$; hence $D_{\kappa_1,\kappa_2}=\Omega_{\kappa_2}\cap\Omega_{\kappa_1+\kappa_2}$ exactly. (The point $z=0$ lies in both sides, since $1+\kappa\cdot 0=1>0$ places it in every $\Omega_\kappa$, so the set equality holds there too.) By Step~1,
\[
  \frac{1}{|w|} \;=\; \frac{1}{|z|}+\kappa_2,
  \qquad
  \frac{1}{|u|} \;=\; \frac{1}{|w|}+\kappa_1.
\]
Substituting,
\[
  \frac{1}{|u|} \;=\; \frac{1}{|z|}+\kappa_1+\kappa_2 \;=\; \frac{1}{|\phi_{\kappa_1+\kappa_2}(z)|}.
\]
Combined with $\arg(u)=\arg(z)=\arg(\phi_{\kappa_1+\kappa_2}(z))$ (from Step~2), this gives $u=\phi_{\kappa_1+\kappa_2}(z)$. The case $z=0$ is again trivial.

\emph{Step 4 (inverses, range, and the partial group property).} First identify the range. The radial map $\sigma_\kappa(r):=r/(1+\kappa r)$ has derivative $\sigma_\kappa'(r)=1/(1+\kappa r)^2>0$ on the admissible range $A_\kappa:=\{r\ge 0:1+\kappa r>0\}$, so it is a strictly increasing continuous function there, with $\sigma_\kappa(0)=0$ and limiting value $1/\kappa$ as $r\to\infty$ when $\kappa>0$ (respectively $+\infty$ as $r\to(1/|\kappa|)^-$ when $\kappa<0$); in either case $\sigma_\kappa$ maps $A_\kappa$ bijectively onto $A_{-\kappa}$. Concretely, for $z\in\Omega_\kappa$ the image $r:=|\phi_\kappa(z)|$ satisfies $1/r=1/|z|+\kappa$ by Step~1, so $1/r-\kappa=1/|z|>0$, that is $1-\kappa r>0$; hence $\phi_\kappa(z)\in\Omega_{-\kappa}$, and by the bijectivity of $\sigma_\kappa$ the modulus $r$ runs over exactly the admissible radii of $\Omega_{-\kappa}$. Therefore $\phi_\kappa(\Omega_\kappa)=\Omega_{-\kappa}$. Next, setting $\kappa_1=\kappa$ and $\kappa_2=-\kappa$ in Step~3 gives $\phi_\kappa\circ\phi_{-\kappa}=\phi_0=\mathrm{id}$ and $\phi_{-\kappa}\circ\phi_\kappa=\mathrm{id}$ on the domains where these compositions are defined. Since $\phi_\kappa$ is continuous and ray-preserving and $\sigma_\kappa$ is a bijection $A_\kappa\to A_{-\kappa}$, the map $\phi_\kappa$ is a homeomorphism $\Omega_\kappa\to\Omega_{-\kappa}$, and $\phi_{-\kappa}$ is its inverse, mapping $\Omega_{-\kappa}$ back onto $\Omega_\kappa$.

The formula in Step~3 shows that the curvature parameter adds exactly as the real parameter of a one-parameter flow; the qualification is that the maximal domains vary with the sign and size of $\kappa$.
\end{proof}

Concretely, the inverse of the cone projection is the radial M\"obius map of curvature $-1/R$:
\[
  f_R^{-1}(w) \;=\; \phi_{-1/R}(w) \;=\; \frac{w}{1-\tfrac{1}{R}|w|}, \qquad |w|<R.
\]

\section{Iteration and the infinitesimal generator}\label{sec:iteration}

Specializing Theorem~\ref{thm:semigroup} to $R_1=R_2=R$ gives a closed-form iteration formula.

\begin{corollary}[Iteration]\label{cor:iteration}
For every integer $n\ge 1$,
\[
  f_R^{\,n} \;=\; f_{R/n}, \qquad |f_R^{\,n}(z)| \;=\; \frac{R\,|z|}{R+n|z|}.
\]
\end{corollary}

\begin{proof}
The curvature of $f_R$ is $1/R$. After $n$ compositions, curvatures add to $n/R$, so $f_R^{\,n}=f_{R/n}$ by Theorem~\ref{thm:semigroup}. Applying \eqref{eq:mag} with $R/n$ gives
\[
  |f_R^{\,n}(z)|=|f_{R/n}(z)|=\frac{(R/n)|z|}{(R/n)+|z|}
  =\frac{R|z|}{R+n|z|}.
\]
\end{proof}

So $n$ iterations of the cone of radius $R$ are geometrically equivalent to a single cone of radius $R/n$. As $n\to\infty$, $|f_R^{\,n}(z)|\to 0$ for every $z$; for every fixed $z\ne0$, the decay is asymptotically $1/n$, more precisely $|f_R^{\,n}(z)|\sim R/n$, meaning $\lim_{n\to\infty}|f_R^{\,n}(z)|\big/(R/n)=1$. Since $|f_R^{\,n}(z)|=R|z|/(R+n|z|)$ and $R+n|z|\sim n|z|$ as $n\to\infty$, the leading term is $R/n$ regardless of the starting radius $|z|$. The exceptional point $z=0$ is fixed for all $n$. Two nonzero points starting far apart are squeezed to the same neighborhood of the origin at the same asymptotic rate. The iteration is therefore a dynamical contraction toward the origin in the orbit sense; $f_R$ is \emph{not} a strict Euclidean contraction in the Banach fixed-point sense, since its derivative tends to the identity at the origin (see Section~\ref{sec:setup}) and its global Lipschitz constant is $1$.

The partial group law lets us define a continuous-time interpolation $f_R^{\,t}:=\phi_{t/R}$ for $t\in\mathbb{R}$. The maximal flow domain is the open set
\[
  \mathcal{D}\;=\;\Bigl\{(t,z)\in\mathbb{R}\times\mathbb{C}:\ 1+\tfrac{t}{R}\,|z|>0\Bigr\},
\]
on which
\[
  f_R^{\,t}(z)=\frac{z}{1+(t/R)\,|z|}.
\]
For $t\ge 0$ this is a continuous one-parameter semigroup of self-maps of $\mathbb{C}$ (and $f_R^{\,n}$ for integer $n\ge 1$ matches Corollary~\ref{cor:iteration}, since $\phi_{n/R}=f_{R/n}$). For $t<0$ the same formula gives the backward flow only on the disk $|z|<R/|t|$.

Recall that the \emph{infinitesimal generator} of such a one-parameter family is the vector field $V$ obtained by differentiating the family at $t=0$ (the standard dynamical-systems notion \cite{teschl2012}),
\[
  V(z):=\frac{d}{dt}\Big|_{t=0} f_R^{\,t}(z),
\]
so that $V(z)$ is the initial velocity of the point $z$ under the motion $t\mapsto f_R^{\,t}(z)$. The generator encodes the entire family infinitesimally: the curves $t\mapsto f_R^{\,t}(z_0)$ are recovered from $V$ as the solutions of the autonomous ODE $\dot z=V(z)$, so each map $f_R^{\,t}$ is the time-$t$ flow of the single vector field $V$. The next proposition identifies this vector field explicitly.

\begin{proposition}[Infinitesimal generator]\label{prop:gen}
For each fixed $z_0\in\mathbb{C}$, the curve $t\mapsto f_R^{\,t}(z_0)$ solves the ODE
\[
  \dot z \;=\; -\,\frac{|z|}{R}\,z, \qquad z(0)=z_0,
\]
on its maximal interval of existence, i.e.\ the largest open interval containing $t=0$ on which a solution exists: all $t\in\mathbb{R}$ if $z_0=0$, and $t>-R/|z_0|$ if $z_0\ne0$. In particular, the family $\{f_R^{\,t}\}_{t\ge 0}$ is the global forward time-$t$ flow of the radial vector field $V(z)=-|z|\,z/R$ on $\mathbb{C}$. This vector field is globally $C^1$ and is smooth on $\mathbb{C}\setminus\{0\}$, but it is not $C^2$ at the origin.
\end{proposition}

\begin{proof}
The case $z_0=0$ is trivial. Suppose $z_0=\rho_0 e^{i\theta_0}$ with $\rho_0>0$. On the maximal interval $I=(-R/\rho_0,\infty)$,
\[
  f_R^{\,t}(z_0)=\frac{\rho_0 R}{R+t\rho_0}\,e^{i\theta_0}=:\rho(t)e^{i\theta_0}.
\]
Therefore
\[
  \dot\rho(t)=-\frac{\rho_0^2R}{(R+t\rho_0)^2}=-\frac{\rho(t)^2}{R}.
\]
Since the argument is constant,
\[
  \frac{d}{dt}f_R^{\,t}(z_0)=\dot\rho(t)e^{i\theta_0}
  =-\frac{\rho(t)^2}{R}e^{i\theta_0}
  =-\frac{|z(t)|}{R}z(t),
\]
where $z(t)=f_R^{\,t}(z_0)$. Thus the curve solves the ODE on $I$.

\emph{Regularity of $V$.} Identify $\mathbb{C}$ with $\mathbb{R}^2$ and write $V(\mathbf{x})=-|\mathbf{x}|\,\mathbf{x}/R$. Away from the origin $V$ is smooth; by the product rule and $D|\mathbf{x}|=\mathbf{x}^{\!\top}\!/|\mathbf{x}|$,
\[
  DV(\mathbf{x})=-\frac{1}{R}\Bigl(|\mathbf{x}|\,I+\frac{\mathbf{x}\otimes\mathbf{x}}{|\mathbf{x}|}\Bigr),
  \qquad \mathbf{x}\ne 0.
\]
At the origin, $|V(\mathbf{x})|=|\mathbf{x}|^2/R=o(|\mathbf{x}|)$, so $V$ is differentiable there with $DV(0)=0$; and since $\bigl\||\mathbf{x}|\,I\bigr\|=|\mathbf{x}|$ and $\bigl\|\mathbf{x}\otimes\mathbf{x}/|\mathbf{x}|\bigr\|=|\mathbf{x}|$, both terms above tend to $0$ as $\mathbf{x}\to 0$, so $DV(\mathbf{x})\to 0=DV(0)$. Hence $V$ is $C^1$ on all of $\mathbb{R}^2$, in particular locally Lipschitz. It is not $C^2$ at the origin: along any unit vector $\mathbf{u}$, the radial component $g(t):=\mathbf{u}\cdot V(t\mathbf{u})=-t|t|/R$ has one-sided second derivatives $-2/R$ from the right and $+2/R$ from the left at $t=0$, so $DV$ fails to be differentiable there. This proves the regularity statements in the proposition.

It remains to see why $I=(-R/\rho_0,\infty)$ is the \emph{maximal} interval. Since $V$ is $C^1$ by the previous step, hence locally Lipschitz, the solution through $z_0$ is unique \cite{teschl2012}, so the curve displayed above is \emph{the} solution wherever it exists. Forward in time the denominator $R+t\rho_0$ only grows, so the solution exists for all $t\ge 0$ (indeed $|z(t)|$ decreases to $0$). Backward in time, as $t\downarrow -R/\rho_0$ the denominator tends to $0^+$ and
\[
  |z(t)|=\frac{\rho_0 R}{R+t\rho_0}\longrightarrow\infty:
\]
the solution blows up in finite backward time, leaving every compact subset of $\mathbb{C}$, and therefore cannot be continued to $t=-R/\rho_0$ or beyond. Hence no solution exists on any larger interval, and $I$ is maximal. (Geometrically, $-R/\rho_0$ is the backward time at which the reciprocal radius $1/|z(t)|=1/\rho_0+t/R$ reaches $0$, i.e.\ at which the point is pushed to infinity.)
\end{proof}

The cone projection family is therefore the time-$1$ map of an explicit radial inward-pointing vector field of magnitude $|z|^2/R$. Thus the radial speed is quadratic in the radius: writing $\rho(t)=|z(t)|$, the flow obeys $\dot\rho=-\rho^2/R$. Near the origin this speed vanishes to \emph{second} order, so the flow has no exponential contraction rate there (unlike a linear field $\dot z=-cz$, whose solutions decay like $e^{-ct}$), which is the dynamical counterpart of $Df_R(0)=I$ (the map is tangent to the identity at the origin). The same ODE becomes transparent in the reciprocal-radius coordinate $u=1/\rho$: by the chain rule $\dot u=-\dot\rho/\rho^2=1/R$, so $u(t)=u(0)+t/R$ moves at constant speed $1/R$. This is the flow form of the reciprocal-radius law of Proposition~\ref{prop:lens}.

\section{Projective structure on rays}\label{sec:cross-ratio}

The radial action $\rho\mapsto R\rho/(R+\rho)$ is a one-dimensional M\"obius transformation, with $2\times 2$ matrix
\[
  M_R \;=\; \begin{pmatrix} R & 0 \\ 1 & R \end{pmatrix}, \qquad \det M_R = R^2 \ne 0.
\]
Here $M_R$ acts on the radial coordinate $\rho\in[0,\infty]$ by the usual rule $\begin{pmatrix} a & b \\ c & d \end{pmatrix}\!\cdot\rho=(a\rho+b)/(c\rho+d)$, which for $M_R$ gives
\[
  M_R\cdot\rho \;=\; \frac{R\rho+0}{1\cdot\rho+R} \;=\; \frac{R\rho}{R+\rho} \;=\; |f_R(z)|\quad\text{when } \rho=|z|,
\]
so the matrix indeed represents the radial action of $f_R$; the two boundary values are $M_R\cdot 0=0$ and $M_R\cdot\infty=R$. General theory (every M\"obius transformation preserves the cross-ratio of any four points \cite{needham1997}) therefore implies that $f_R$ preserves the cross-ratio of any four points on a common ray from the origin. Throughout this section the cross-ratio of four collinear points is taken in the ordering convention
\[
  (A,B;C,D)=\frac{(A-C)(B-D)}{(A-D)(B-C)}.
\]
We give the explicit calculation.

\begin{proposition}[Cross-ratio preservation on rays]\label{prop:cross-ratio}
Let $z_j=\rho_j\,e^{i\theta}\in\mathbb{C}$, $j=1,2,3,4$, share a common argument $\theta$, with distinct moduli $\rho_j\ge 0$. (When one of the points is $O$, with $\rho_j=0$, we regard $O$ as lying on the chosen ray, even though it has no unique argument.) Set $s_j:=|f_R(z_j)|=R\rho_j/(R+\rho_j)$. Then
\[
  \frac{(s_1-s_3)(s_2-s_4)}{(s_1-s_4)(s_2-s_3)} \;=\; \frac{(\rho_1-\rho_3)(\rho_2-\rho_4)}{(\rho_1-\rho_4)(\rho_2-\rho_3)}.
\]
\end{proposition}

\begin{proof}
For $i\ne j$,
\[
  s_i-s_j=\frac{R\rho_i}{R+\rho_i}-\frac{R\rho_j}{R+\rho_j}
  =\frac{R^2(\rho_i-\rho_j)}{(R+\rho_i)(R+\rho_j)}.
\]
Substituting this expression for the four differences in the displayed cross-ratio, the factor $R^4$ cancels and each factor $R+\rho_j$ appears once in the numerator and once in the denominator. The remaining expression is
\[
  \frac{(\rho_1-\rho_3)(\rho_2-\rho_4)}{(\rho_1-\rho_4)(\rho_2-\rho_3)}.
\]
\end{proof}

This is the projective signature of the cone construction. While $f_R:\mathbb{C}\to\mathbb{C}$ is not a M\"obius transformation of $\mathbb{C}$ (it is not even holomorphic, since it depends on $|z|$), its restriction to any single ray through the origin \emph{is} a M\"obius transformation. The cone construction realizes this family geometrically: it applies one and the same one-dimensional M\"obius map $\rho\mapsto R\rho/(R+\rho)$ on every ray from the origin, and the rotational symmetry of the cone carries the construction on one ray to the construction on any other, which is why a single map suffices for all rays. We emphasize that the preservation statement is projective on each ray separately; it is \emph{not} a cross-ratio theorem for arbitrary collinear quadruples in $\mathbb{C}$: on a full line through the origin, consisting of two opposite rays, the radial action depends on $|z|$ and is not a M\"obius function of the signed coordinate.

\begin{remark}
Allowing $\rho=\infty$ (the point at infinity in the projective completion of the ray), the radial M\"obius map sends $\infty\mapsto R$. Cross-ratios involving $\infty$ are preserved in the usual sense. With the convention used in Proposition~\ref{prop:cross-ratio},
\[
  (\rho_1,\rho_2;\rho_3,\infty)
  \;=\; \frac{\rho_1-\rho_3}{\rho_2-\rho_3},
\]
while the image cross-ratio is
\[
  (s_1,s_2;s_3,R)
  \;=\;
  \frac{(s_1-s_3)(s_2-R)}{(s_1-R)(s_2-s_3)}.
\]
Substituting $s_j=R\rho_j/(R+\rho_j)$ into the latter expression gives the former one; thus the preservation statement remains correct.
\end{remark}

\subsection{Forward and backward maps as harmonic conjugates}

Proposition~\ref{prop:cross-ratio} showed that $f_R$ preserves the cross-ratio of four points on a common ray. We close the projective discussion with a special case relating $f_R$ to its inverse. For $0<|p|<R$ the inverse
\[
  f_R^{-1}(p)=\frac{p}{\,1-\frac{1}{R}|p|\,}
\]
is defined and, like $f_R$, is radial; recall from Theorem~\ref{thm:group} that it equals $\phi_{-1/R}$. The restriction $0<|p|<R$ is needed only so that this backward image $f_R^{-1}(p)$ exists. We use the classical projective-geometric language of cross-ratios and harmonic ranges; see, for example, \cite{veblenyoung1910,coxeterPG1987}.

\begin{proposition}\label{prop:harmonic}
With the cross-ratio convention of Proposition~\ref{prop:cross-ratio},
for every point $p$ with $0<|p|<R$, the four points $O$, $p$, $f_R(p)$,
and $f_R^{-1}(p)$, all lying on the ray $Op$, form a harmonic range:
\[
  \bigl(O,\,p;\ f_R(p),\,f_R^{-1}(p)\bigr)=-1.
\]
Equivalently, $f_R(p)$ and $f_R^{-1}(p)$ are harmonic conjugates with respect to $O$ and $p$.
\end{proposition}

\begin{proof}
All four points lie on the ray $Op$; parameterize it by the positive radial coordinate measured from $O$ and write $\zeta=|p|$. Then
\[
  O\leftrightarrow 0,\qquad p\leftrightarrow\zeta,\qquad
  u:=f_R(p)\leftrightarrow\frac{R\zeta}{R+\zeta},\qquad
  v:=f_R^{-1}(p)\leftrightarrow\frac{R\zeta}{R-\zeta},
\]
the last two by \eqref{eq:mag} and its inverse. The cross-ratio is computed in this ray coordinate:
\[
  \bigl(O,p;u,v\bigr)
  =\frac{(0-u)(\zeta-v)}{(0-v)(\zeta-u)}
  =\frac{(-u)(\zeta-v)}{(-v)(\zeta-u)}
  =\frac{u}{v}\cdot\frac{\zeta-v}{\zeta-u}.
\]
Compute each factor. First,
\[
  \frac{u}{v}=\frac{R\zeta/(R+\zeta)}{R\zeta/(R-\zeta)}=\frac{R-\zeta}{R+\zeta}.
\]
Next,
\[
  \zeta-u=\zeta-\frac{R\zeta}{R+\zeta}=\frac{\zeta(R+\zeta)-R\zeta}{R+\zeta}
         =\frac{\zeta^{2}}{R+\zeta},
\]
\[
  \zeta-v=\zeta-\frac{R\zeta}{R-\zeta}=\frac{\zeta(R-\zeta)-R\zeta}{R-\zeta}
         =\frac{-\zeta^{2}}{R-\zeta}.
\]
Therefore
\[
  \frac{\zeta-v}{\zeta-u}
  =\frac{-\zeta^{2}/(R-\zeta)}{\zeta^{2}/(R+\zeta)}
  =-\,\frac{R+\zeta}{R-\zeta},
\]
and multiplying the two factors,
\[
  \bigl(O,p;u,v\bigr)
  =\frac{R-\zeta}{R+\zeta}\cdot\left(-\frac{R+\zeta}{R-\zeta}\right)=-1.
  \qedhere
\]
\end{proof}

This harmonic relation is the projective-geometric form of the reciprocal-radius law of Proposition~\ref{prop:lens}, equivalently the curvature law for the concentric circles centered at $O$. Write $\kappa=1/R$. By Proposition~\ref{prop:lens} and the group law of Theorem~\ref{thm:group},
\[
  \frac{1}{|f_R(p)|}=\frac{1}{|p|}+\kappa,
  \qquad
  \frac{1}{|f_R^{-1}(p)|}=\frac{1}{|p|}-\kappa,
\]
so the forward and backward maps shift the reciprocal radius $1/|p|$ by the same amount $\kappa$, in opposite directions. Since $1/\rho$ is the curvature of the circle of radius $\rho$ centered at $O$, the two image points lie on concentric circles whose curvatures differ from that of the circle centered at $O$ and passing through $p$ by exactly $\pm\kappa$. Read on the ray through $p$ in projective terms, this symmetric pair of reciprocal-radius shifts is precisely the harmonic-conjugacy relation proved above.

The harmonic relation is in fact exclusive to the family $\phi_\kappa$ of
Theorem~\ref{thm:group}: within the class of radial M\"obius maps, the
identity of Proposition~\ref{prop:harmonic} \emph{characterizes} the flow.

\begin{proposition}[Harmonic conjugacy characterizes the flow]\label{prop:harmonic-char}
Let $\Phi(\rho e^{i\theta})=\sigma(\rho)e^{i\theta}$ be a ray-preserving
map, distinct from the identity, whose radial action $\sigma$ is a
M\"obius function of the radius, and suppose there is a nonempty open
interval $I\subset(0,\infty)$ of radii on which $\sigma$ and
$\sigma^{-1}$ are both defined and positive and
\[
  \bigl(O,\,p;\ \Phi(p),\,\Phi^{-1}(p)\bigr)=-1
  \qquad\text{whenever } |p|\in I.
\]
Then $\Phi=\phi_\kappa$ for some $\kappa\neq 0$, i.e.\
$\sigma(\rho)=\rho/(1+\kappa\rho)$. Conversely, every $\phi_\kappa$ with
$\kappa\neq 0$ satisfies the identity for all $p$ with
$0<|p|<1/|\kappa|$.
\end{proposition}

\begin{proof}
Fix a ray and use its positive radial coordinate as in the proof of
Proposition~\ref{prop:harmonic}, writing $\zeta=|p|$, $x=\sigma(\zeta)$,
$y=\sigma^{-1}(\zeta)$. In the ordering convention of
Proposition~\ref{prop:cross-ratio},
\[
  \bigl(0,\zeta;x,y\bigr)=\frac{x(\zeta-y)}{y(\zeta-x)}=-1
  \quad\Longleftrightarrow\quad
  \zeta(x+y)=2xy
  \quad\Longleftrightarrow\quad
  \frac{1}{x}+\frac{1}{y}=\frac{2}{\zeta}.
\]
Pass to the reciprocal coordinate $u=1/\zeta$ and let
$g(u):=1/\sigma(1/u)$ be the reciprocal conjugate of $\sigma$; then
$1/\sigma^{-1}(1/u)=g^{-1}(u)$, and the identity reads
\[
  g(u)+g^{-1}(u)=2u,
  \qquad\text{i.e.}\qquad
  g^{-1}=2\,\mathrm{id}-g
\]
on an open interval. Since $\sigma$ is M\"obius, so is $g$; write
$g(u)=(au+b)/(cu+d)$ with $ad-bc\neq 0$. If $c\neq 0$, then
$2u-g(u)=\bigl(2cu^{2}+(2d-a)u-b\bigr)/(cu+d)$, and polynomial division
shows the numerator is divisible by $cu+d$ only when $ad-bc=0$; so
$2\,\mathrm{id}-g$ is a rational function of degree two, which cannot
agree with the M\"obius function $g^{-1}$ on an open interval. Hence
$c=0$ and $g(u)=\alpha u+\beta$ is affine with $\alpha\neq0$. Then
$g^{-1}(u)=(u-\beta)/\alpha$, and
\[
  \alpha u+\beta+\frac{u-\beta}{\alpha}=2u
  \quad\text{on an interval}
  \quad\Longrightarrow\quad
  \alpha+\frac{1}{\alpha}=2
  \quad\Longrightarrow\quad
  (\alpha-1)^{2}=0,
\]
so $\alpha=1$ and $\beta$ is unconstrained: $g$ is the translation
$u\mapsto u+\beta$. Undoing the conjugation, $1/\sigma(\rho)=1/\rho+\beta$,
i.e.\ $\sigma(\rho)=\rho/(1+\beta\rho)$ and $\Phi=\phi_\beta$; and
$\beta\neq0$ since $\Phi$ is not the identity. The converse is the
computation preceding this proposition with $\kappa=\beta$: the reciprocal
shifts $\pm\kappa$ are symmetric about $1/|p|$, which is the displayed
harmonic condition; both $\phi_\kappa(p)$ and
$\phi_\kappa^{-1}(p)=\phi_{-\kappa}(p)$ are defined and positive exactly
when $1\pm\kappa|p|>0$, i.e.\ $0<|p|<1/|\kappa|$.
\end{proof}

\begin{remark}
With $\kappa=1/R$ this recovers and sharpens
Proposition~\ref{prop:harmonic}: among radial M\"obius maps, the single
projective identity $\bigl(O,p;\,\Phi(p),\Phi^{-1}(p)\bigr)=-1$ pins down
the one-parameter family $\phi_\kappa$, with only the parameter left
free. It is thus a projective companion to the Normalized
characterization of Theorem~\ref{thm:char}, where boundedness and
first-order normalization select the single member $f_R$. The M\"obius
hypothesis on the radial action cannot simply be dropped: as a bare
functional equation on bijections, $g^{-1}=2\,\mathrm{id}-g$ admits
discontinuous solutions other than translations (for example,
$g(u)=u+1$ for $u$ rational and $g(u)=u+2$ for $u$ irrational), so some
rigidity hypothesis is needed.
\end{remark}

\section{Further constructions and interpretations}\label{sec:alt-constructions}

The lens-equation identity (Proposition~\ref{prop:lens}) is the algebraic skeleton of $f_R$, and any geometric setup encoding the relation $1/|f_R(z)|=1/|z|+1/R$ produces $f_R$ on each ray from the origin. We collect several constructions, interpretations, and characterizations related to this relation, beyond the cone of Section~\ref{sec:setup}. The subsections below are independent of one another and may be read in any order.

\subsection{The planar two-line construction}\label{sub:planar}

The cone construction reduces to a strictly two-dimensional similar-triangles diagram: the side view of Figure~\ref{fig:side} is itself a complete construction of $f_R$, requiring no third dimension.

\begin{proposition}\label{prop:planar}
Fix $z\in\mathbb{C}\setminus\{0\}$ and work in an auxiliary $2$-dimensional construction plane $\Pi_z$. Inside $\Pi_z$, fix a baseline $\ell_0$ with a chosen origin $O\in\ell_0$ and a parallel line $\ell_h$ at distance $h>0$. Choose $Z\in\ell_0$ with $|OZ|=|z|$; choose $T\in\ell_h$, and $S\in\ell_h$ with $\overrightarrow{TS}$ parallel to $\overrightarrow{OZ}$, of the same sense and length $R$. Draw the diagonals $L_1=\overline{ZT}$ and $L_2=\overline{OS}$ of the trapezoid $OZST$, meeting at $P$. Let the line through $P$ parallel to $\ell_0$ meet the leg $\overline{OT}$ at $Q$ (Figure~\ref{fig:planar}). Then $|PQ|=R|z|/(R+|z|)$, independent of $h$ and of the position of $T$ on $\ell_h$. Placing this length on the ray from $0$ through $z$ in $\mathbb{C}$ recovers $f_R(z)$.
\end{proposition}

\begin{proof}
Apart from the prescribed lengths $|OZ|=|z|$ and $|TS|=R$, the proof uses only the parallelism $\overline{OZ}\parallel\overline{TS}$ (equivalently, $\ell_0\parallel\ell_h$); neither $h$ nor the horizontal position of $T$ enters. In triangle $TOZ$ the segment $\overline{PQ}$ is parallel to the base $\overline{OZ}$, with $Q$ on $\overline{TO}$ and $P$ on $\overline{TZ}$, so the similar triangles $\triangle TQP\sim\triangle TOZ$ give $|PQ|/|OZ|=|TP|/|TZ|$. Because $\overline{OZ}\parallel\overline{TS}$, the diagonals cut the vertically opposite similar triangles $\triangle OPZ\sim\triangle SPT$ in the ratio $|OZ|:|TS|=|z|:R$, whence $|PZ|:|TP|=|z|:R$ and $|TP|/|TZ|=|TP|/(|TP|+|PZ|)=R/(R+|z|)$. Therefore $|PQ|=|OZ|\cdot R/(R+|z|)=R|z|/(R+|z|)$, independent of $h$. By rotation equivariance of $f_R$, transporting this length to the ray $Oz\subset\mathbb{C}$ yields $f_R(z)$.
\end{proof}

In the normalized realization drawn in Figure~\ref{fig:planar} ($T=(0,h)$ directly above $O$, and $S=(R,h)$), the leg $\overline{OT}$ is vertical, so $|PQ|$ is also the horizontal coordinate of $P$, namely $R|z|/(R+|z|)$. (For a shifted upper base the length $|PQ|$ is unchanged, but $P$'s horizontal coordinate is not.) This realizes the classical compass-and-straightedge construction of the parallel combination
\[
  \frac{1}{1/R+1/|z|}=\frac{R|z|}{R+|z|}.
\]
This quantity is exactly one half of the harmonic mean of the two parallel sides $|z|$ and $R$. Indeed, the segment $\overline{QW}$ through $P$ parallel to the bases, with $Q$ on the leg $\overline{OT}$ and $W$ on the leg $\overline{ZS}$, has length $|QW|=2R|z|/(R+|z|)$, and $P$ bisects it, so $|PQ|=\tfrac12|QW|=R|z|/(R+|z|)=|f_R(z)|$. That the segment through the intersection of the diagonals of a trapezoid, parallel to its bases and with endpoints on the legs, has length the harmonic mean of the two bases is classical~\cite{posamentier1996,charosh1965}. Stood on end (with the two parallel sides drawn as verticals of heights $|z|$ and $R$ over a common base and the diagonals crossing between them), this is the configuration sometimes called \emph{crossed ladders}, in which the crossing height equals the parallel combination $R|z|/(R+|z|)$ independently of the base width. The base-width independence is the planar shadow of the height independence of Proposition~\ref{prop:p10}, and the parallel-combination law is the same reciprocal-addition rule $1/|f_R(z)|=1/|z|+1/R$ that governs resistors in parallel (Proposition~\ref{prop:lens}).

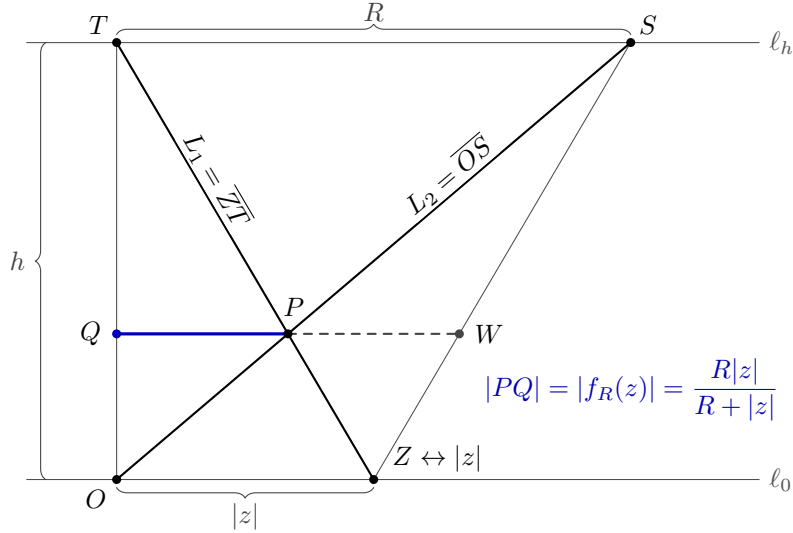
\begin{figure}[htbp]
\centering
\begin{tikzpicture}[scale=1.7, line cap=round, line join=round, font=\small,
                    >={Latex[length=2.2mm]}]
  \coordinate (O) at (0,0);
  \coordinate (Z) at (2,0);
  \coordinate (T) at (0,3.4);
  \coordinate (S) at (4,3.4);
  \coordinate (P)  at (1.3333,1.1333);
  \coordinate (Pl) at (0,1.1333);
  \coordinate (Pr) at (2.6667,1.1333);
  \draw[gray!60!black] (-0.7,0)   -- (5.0,0)   node[right]{$\ell_0$};
  \draw[gray!60!black] (-0.7,3.4) -- (5.0,3.4) node[right]{$\ell_h$};
  \draw[gray!60!black] (O) -- (T);
  \draw[gray!60!black] (Z) -- (S);
  \draw[thick] (Z) -- (T) node[pos=2/3, sloped, above, fill=white, inner sep=1pt] {$L_1=\overline{ZT}$};
  \draw[thick] (O) -- (S) node[pos=2/3, sloped, above, fill=white, inner sep=1pt] {$L_2=\overline{OS}$};
  \draw[gray!55!black, thick, dashed] (P) -- (Pr);
  \draw[blue!70!black, very thick] (Pl) -- (P);
  \draw[decorate, decoration={brace, amplitude=4pt, mirror, raise=3pt}, gray!50!black]
        (O) -- (Z) node[midway, below=5pt]{$|z|$};
  \draw[decorate, decoration={brace, amplitude=4pt, raise=3pt}, gray!50!black]
        (T) -- (S) node[midway, above=7pt, fill=white, inner sep=1pt]{$R$};
  \draw[decorate, decoration={brace, amplitude=4pt, mirror, raise=0pt}, gray!50!black]
        (-0.55,3.4) -- (-0.55,0) node[midway, left=4pt]{$h$};
  \foreach \pt in {O,Z,T,S,P}{\fill (\pt) circle (0.035);}
  \fill[blue!70!black] (Pl) circle (0.035);
  \fill[gray!55!black] (Pr) circle (0.035);
  \node[below left=0pt] at (O) {$O$};
  \node[above right=-1pt, xshift=4pt] at (Z) {$Z\leftrightarrow|z|$};
  \node[above left=-1pt] at (T) {$T$};
  \node[above right=-1pt] at (S) {$S$};
  \node[above left=0pt, xshift=10pt, yshift=3pt] at (P) {$P$};
  \node[left=2pt] at (Pl) {$Q$};
  \node[right=2pt] at (Pr) {$W$};
  \node[blue!70!black, anchor=west,xshift=-25pt] (FE) at (3.3,0.7)
        {$|PQ|=|f_R(z)|=\dfrac{R|z|}{R+|z|}$};
      \end{tikzpicture}
      \caption{The planar two-line construction of $f_R$ (Proposition~\ref{prop:planar}), drawn in the normalized case $T=(0,h)$, $S=(R,h)$ with $R=4$, $|z|=2$. The diagonals $L_1=\overline{ZT}$ and $L_2=\overline{OS}$ of the trapezoid $OZST$ meet at $P$. The line through $P$ parallel to the bases meets the left leg $\overline{OT}$ at $Q$ and the right leg $\overline{ZS}$ at $W$. The blue left segment has length $|PQ|=R|z|/(R+|z|)=|f_R(z)|$, independent of the height $h$, so it realizes the radial coordinate of $f_R(z)$. The full parallel segment has length $|QW|=2R|z|/(R+|z|)$, the harmonic mean of the base lengths $|z|$ and $R$; since $P$ is its midpoint, $|PQ|=|PW|$ (the dashed half).}
\label{fig:planar}
\end{figure}

\subsection{Inversion, radial translation, inversion}\label{sub:inversion}

Because the lens identity reads $1/|f_R(z)|=1/|z|+1/R$, the map becomes a radial \emph{translation} after passing through inversion. This yields a construction using inversion together with a radial translation:

\begin{proposition}\label{prop:inversion}
Let $\iota:\mathbb{C}\setminus\{0\}\to\mathbb{C}\setminus\{0\}$ denote inversion through the unit circle, $\iota(z)=z/|z|^2$. Let $\tau_R$ denote radial translation by $1/R$:
\[
  \tau_R(w) \;:=\; w + \frac{1}{R}\,\frac{w}{|w|}, \qquad w\ne 0.
\]
Then on $\mathbb{C}\setminus\{0\}$,
\[
  f_R \;=\; \iota \circ \tau_R \circ \iota.
\]
\end{proposition}

\begin{proof}
Let $z\ne 0$. Compute step by step.

\emph{Step 1 (first inversion).} $\iota(z)=z/|z|^2$. This is the point on the same ray from the origin as $z$, at distance $1/|z|$ from the origin: $|\iota(z)|=|z|/|z|^2=1/|z|$, and $\iota(z)/|\iota(z)|=z/|z|$ since the scalar factor $|z|^{-2}$ is positive.

\emph{Step 2 (radial translation).} $\tau_R(\iota(z))=\iota(z)+\tfrac{1}{R}\iota(z)/|\iota(z)|$. Since $\iota(z)/|\iota(z)|=z/|z|$ (same ray), this equals
\[
  \tau_R(\iota(z)) \;=\; \frac{z}{|z|^2}+\frac{1}{R}\cdot\frac{z}{|z|} \;=\; \frac{z}{|z|}\left(\frac{1}{|z|}+\frac{1}{R}\right).
\]
This is a positive scalar multiple of $z/|z|$, hence on the same ray, with magnitude $1/|z|+1/R$.

\emph{Step 3 (second inversion).} Apply $\iota$ to the result of Step~2. The argument is preserved (it is on the same ray) and the magnitude is inverted:
\[
  |\iota(\tau_R(\iota(z)))| \;=\; \frac{1}{1/|z|+1/R} \;=\; \frac{|z|\,R}{R+|z|}.
\]
Combining,
\[
  \iota(\tau_R(\iota(z))) \;=\; \frac{|z|\,R}{R+|z|}\cdot\frac{z}{|z|} \;=\; \frac{R\,z}{R+|z|} \;=\; \frac{z}{1+|z|/R} \;=\; f_R(z),
\]
which is \eqref{eq:fz}.
\end{proof}

The construction is illustrated in Figure~\ref{fig:inversion}. It says: \emph{the cone projection is the conjugate of radial translation by inversion}, i.e., $f_R=\iota\,\tau_R\,\iota^{-1}$, since $\iota=\iota^{-1}$. The middle map $\tau_R$ is not itself an inversion or a complex M\"obius transformation; it is a radial translation in the inverted radius coordinate.

\begin{figure}[htbp]
  \centering
\begin{tikzpicture}[
  >=Latex,
  scale=3.00,
  line cap=round,
  line join=round,
  font=\small,
  every node/.style={inner sep=1.5pt},
  point/.style={circle, fill=black, inner sep=1.15pt},
  lab/.style={fill=white, fill opacity=0.94, text opacity=1, rounded corners=1pt}
]
  \def\zmag{3.0}
  \def\rmag{1.0}
  \pgfmathsetmacro{\zp}{1/\zmag}                    
  \pgfmathsetmacro{\fz}{\rmag*\zmag/(\rmag+\zmag)} 
  \pgfmathsetmacro{\zpp}{\zp + 1/\rmag}             

  \coordinate (O) at (0,0);
  \coordinate (I) at (\zp,0);
  \coordinate (F) at (\fz,0);
  \coordinate (One) at (1,0);
  \coordinate (T) at (\zpp,0);
  \coordinate (Z) at (\zmag,0);

  \draw[thick, gray!35] (O) circle (1);
  \draw[thick, ->] (-0.15,0) -- (3.34,0);
  \draw[thin, gray!50] (1,-0.055) -- (1,0.055);

  \node[point] at (O) {};
  \node[point] at (I) {};
  \node[point] at (F) {};
  \node[point] at (T) {};
  \node[point] at (Z) {};

  \node[lab, below left=5pt] at (O) {$O$};
  \node[lab, below=5pt] at (I) {$\iota(z)$};
  \node[lab, below=6pt] at (F) {$f_R(z)$};
  \node[lab, right = 5pt, above=3pt] at (One) {$1$};
  \draw[thin, gray!65] (T) -- (1.55,-0.36);
  \node[lab, below right=1pt] at (1.55,-0.36) {$\tau_R(\iota(z))$};
  \node[lab, below=6pt] at (Z) {$z$};

  \draw[->, thin]
    (Z) .. controls (2.95,1.82) and (0.40,1.82) .. (I);
  \node[lab, above=2pt] at (1.68,1.36) {first inversion};

  \draw[->, very thick] (\zp,-0.38) -- (\zpp,-0.38);
  \node[lab, below=2pt] at ({(\zp+\zpp)/2},-0.38) {translate by $1/R$};

  \draw[->, thin]
    (T) .. controls (1.28,0.72) and (0.82,0.72) .. (F);
  \node[lab, above=1pt] at (1.05,0.55) {second inversion};
\end{tikzpicture}
\caption{The inversion--translation--inversion construction of $f_R$, depicted on the ray from $O$ through $z$. Numerical illustration: $R=1$, $|z|=3$, so $|\iota(z)|=1/3$, $|\tau_R(\iota(z))|=1/3+1=4/3$, and $|f_R(z)|=3/(1+3)=3/4=1/(4/3)$, consistent with $1/|f_R(z)|=1/|z|+1/R$.}
\label{fig:inversion}
\end{figure}
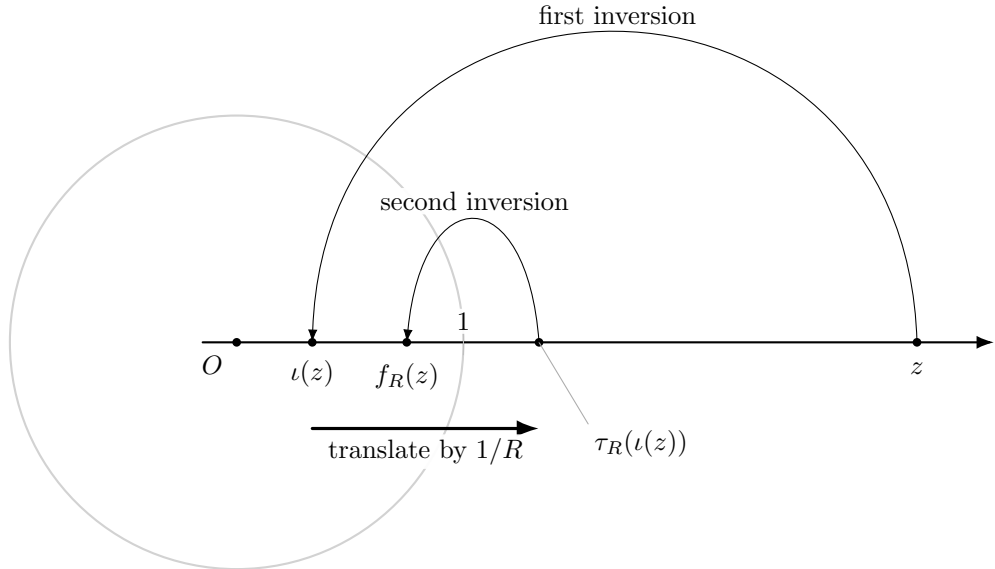

\subsection{A thin-lens reciprocal-distance analogy}\label{sub:lens}

The lens identity $1/|f_R(z)|=1/|z|+1/R$ has the same reciprocal-distance form as a thin-lens formula \cite{hecht2017} after adopting the sign convention
\[
  \frac{1}{v}=\frac{1}{u}+\frac{1}{f}.
\]
This is not the most common elementary convention, where one often writes $1/f=1/u+1/v$ for positive real object and image distances. Therefore the optical interpretation should be read as a formal ray-diagram analogy rather than as the usual real-image law for a converging lens. Identifying
\[
  u\leftrightarrow |z|, \qquad v\leftrightarrow |f_R(z)|, \qquad f\leftrightarrow R,
\]
the cone projection reproduces this reciprocal-distance law on each ray from the origin.

\begin{proposition}[Formal thin-lens analogy]\label{prop:lens-optics}
Let $z\ne 0$. In the sign convention $1/v=1/u+1/f$, the assignment
\[
  u=|z|, \qquad f=R, \qquad v=|f_R(z)|
\]
satisfies the thin-lens algebraic relation.
\end{proposition}

\begin{proof}
The stated relation is exactly the lens identity of Proposition~\ref{prop:lens}. The proposition asserts only this reciprocal-distance equivalence; it does not identify $f_R$ with the usual physical imaging map under the standard positive-distance convention.
\end{proof}

\subsection{The radial M\"obius perspective}\label{sub:mobius}

As used in the raywise cross-ratio discussion, the radial action $\rho\mapsto R\rho/(R+\rho)$ is the M\"obius transformation with matrix
\[
  M_R = \begin{pmatrix} R & 0 \\ 1 & R \end{pmatrix}\in\mathrm{GL}_2(\mathbb{R}), \qquad \det M_R = R^2.
\]
After passing to $\mathrm{PSL}_2(\mathbb{R})$ by normalizing the matrix to have determinant $1$ (which, since $\det M_R=R^2$, corresponds to dividing all matrix entries by $R$),
\[
  \widetilde M_R = \frac{1}{R}\begin{pmatrix} R & 0 \\ 1 & R \end{pmatrix} = \begin{pmatrix} 1 & 0 \\ 1/R & 1 \end{pmatrix}.
\]
This is a standard parabolic element of the lower unipotent subgroup of
\(\mathrm{PSL}_2(\mathbb{R})\) \cite{beardon1983}, the group acting on the upper half-plane \(\mathbb{H}=\{z\in\mathbb{C}:\operatorname{Im}z>0\}\) by M\"obius transformations, with boundary the extended real line \(\mathbb{R}\cup\{\infty\}\). Equipped with the metric \(ds^2=(dx^2+dy^2)/y^2\), \(\mathbb{H}\) is the hyperbolic plane and \(\mathrm{PSL}_2(\mathbb{R})\) is precisely its group of orientation-preserving isometries \cite{beardon1983,anderson2005}; the geodesics are the vertical half-lines and the semicircles meeting \(\mathbb{R}\) orthogonally, and a parabolic element is one fixing a single boundary point, where it acts as a limit of rotations. (This \(ds^2\) is the standard Poincar\'e metric on \(\mathbb{H}\): the Euclidean line element \(dx^2+dy^2\) is rescaled by \(1/y^2\), so a fixed Euclidean displacement counts as more length the nearer it lies to the boundary \(y=0\), placing \(\mathbb{R}\cup\{\infty\}\) at infinite distance and giving \(\mathbb{H}\) constant curvature \(-1\). It is invoked here only to interpret \(\widetilde M_R\): it is the metric whose isometry group is exactly \(\mathrm{PSL}_2(\mathbb{R})\), so it is what gives the bare matrix \(\widetilde M_R\) its geometric meaning: the parabolic classification, the fixed boundary point, and the horocycles of Figures~\ref{fig:horocycle-conjugacy}--\ref{fig:horocycle-invariance}. It is distinct from the flat pullback metric \(g_R\) of Section~\ref{sec:metric}, which is intrinsic to the cone projection.) Its action on the boundary
coordinate \(\rho\in[0,\infty]\) is
\[
  \rho \longmapsto \frac{\rho}{1+\rho/R}
  = \frac{R\rho}{R+\rho},
\]
which is exactly the radial action of \(f_R\). Thus the hyperbolic interpretation concerns the raywise radial coordinate, not \(f_R\) as a complex M\"obius transformation of the plane.
Concretely:
\begin{proposition}\label{prop:mobius}
Pick any ray from the origin and parameterize it by $\rho=|z|\ge 0$. The radial action $\rho\mapsto R\rho/(R+\rho)$ coincides with the restriction to the projective half-line $[0,\infty]$ of the boundary action of the parabolic M\"obius transformation $\widetilde M_R\in\mathrm{PSL}_2(\mathbb{R})$; this restriction maps $[0,\infty]$ into $[0,R]$.
\end{proposition}

\begin{proof}
The matrix $\widetilde M_R=\bigl(\begin{smallmatrix}1&0\\1/R&1\end{smallmatrix}\bigr)$ acts on $\rho$ by $\rho\mapsto(1\cdot\rho+0)/((1/R)\rho+1)=\rho/(1+\rho/R)=R\rho/(R+\rho)$, which equals $|f_R(z)|$ when $\rho=|z|$. This action is strictly increasing in $\rho$ with boundary values $0\mapsto 0$ and $\infty\mapsto R$, so it maps $[0,\infty]$ into (indeed onto) $[0,R]$.
\end{proof}

Equivalently, \(\widetilde M_R\) is conjugate to an ordinary horizontal translation.
Indeed, the map \(S(z)=-1/z\), which is an involution (so \(S^{-1}=S\)), sends the ideal fixed point \(0\) of \(\widetilde M_R\)
to \(\infty\), and one checks that
\[
  S\circ \widetilde M_R\circ S^{-1}(z)=z-\frac{1}{R}.
\]
Thus \(\widetilde M_R\) preserves the horocycles of \(\mathbb H\) based at \(0\) and acts by translations along them \cite{beardon1983,anderson2005}. In the upper half-plane model these horocycles are precisely the Euclidean circles in \(\mathbb H\) tangent to the real axis at \(0\); equivalently, they meet orthogonally the geodesics with ideal endpoint \(0\). This should not be read as
a hyperbolic-isometry statement about the full planar map \(f_R\), which is not
conformal as a map of \(\mathbb C\). Figure~\ref{fig:horocycle-invariance} shows the
horocycle foliation based at \(0\) together with the geodesics it is orthogonal to,
and Figure~\ref{fig:horocycle-conjugacy} displays the conjugacy \(S\circ\widetilde M_R\circ S^{-1}\)
that turns this motion into an ordinary horizontal translation.

To make the word ``preserves'' concrete, it is worth tracing a single orbit.
Working with \(\widetilde M_R\) directly is awkward because it slides each circle
along itself; the conjugacy of Figure~\ref{fig:horocycle-conjugacy} removes that
difficulty. Recall that \(S(z)=-1/z\) carries the fixed point \(0\) to \(\infty\)
and conjugates \(\widetilde M_R\) to the horizontal translation \(z\mapsto z-1/R\).
A horocycle based at \(0\) is a circle tangent to the real axis at the origin,
say the circle of radius \(r\) centred at \((0,r)\); the quantity
\[
  \operatorname{Im}\!\Bigl(-\tfrac1z\Bigr)=\frac{y}{x^2+y^2}
  \qquad (z=x+iy)
\]
is constant on it, equal to \(1/(2r)\), and so labels which horocycle a point lies on.
Because \(S\) sends this quantity to the ordinary height \(\operatorname{Im}\) and
sends \(\widetilde M_R\) to a horizontal translation, the orbit stays at constant
height in the conjugated picture and hence on a fixed horocycle in the original one.
Concretely, take \(R=1\), so \(\widetilde M_R\) acts by \(w=z/(z+1)\), and start at
\(z_0=i\):
\[
  i\;\longmapsto\;\tfrac12+\tfrac12 i\;\longmapsto\;\tfrac25+\tfrac15 i
   \;\longmapsto\;\tfrac{3}{10}+\tfrac{1}{10}i\;\longmapsto\;\tfrac{4}{17}+\tfrac{1}{17}i
   \;\longmapsto\;\cdots
\]
Every iterate satisfies \(y/(x^2+y^2)=1\), that is \(x^2+(y-\tfrac12)^2=\tfrac14\): the
points move but never leave the circle of radius \(\tfrac12\) tangent to the real axis
at \(0\), sliding along it toward the fixed point \(0\)
(Figure~\ref{fig:horocycle-orbit}).

\begin{figure}[htbp]
  \centering
  \begin{tikzpicture}[>=Latex, line join=round, line cap=round, font=\small, scale=4.2]
    \draw[thick] (-0.35,0) -- (1.15,0);
    \node[anchor=west] at (1.12,-0.06) {$\mathbb{R}$};
    \fill (0,0) circle (0.4pt);
    \node[below] at (0,-0.02) {$0$};
    \draw[blue!55!black, thick] (0,0.5) circle (0.5);
    \node[blue!50!black, anchor=west] at (0.55,0.72)
      {horocycle: $x^2+(y-\tfrac12)^2=\tfrac14$};
    \coordinate (z0) at (0,1);
    \coordinate (z1) at (0.5,0.5);
    \coordinate (z2) at (0.4,0.2);
    \coordinate (z3) at (0.3,0.1);
    \coordinate (z4) at (0.235,0.0588);
    \foreach \a/\b in {z0/z1, z1/z2, z2/z3, z3/z4} {
      \draw[orange!85!black, very thick, ->] (\a) to[bend right=18] (\b);
    }
    \foreach \p/\lab/\pos in {%
       z0/$z_0=i$/{above},%
       z1/$z_1$/{above right},%
       z2/$z_2$/{right},%
       z3/$z_3$/{right},%
       z4/$z_4$/{below right}} {
      \fill[orange!60!black] (\p) circle (0.5pt);
      \node[\pos=1.5pt, font=\footnotesize] at (\p) {\lab};
    }
  \end{tikzpicture}
  \caption{A single orbit under the parabolic map, illustrating that $\widetilde M_R$ slides points \emph{along} a horocycle. With $R=1$, the map is $w=z/(z+1)$; the orbit of $z_0=i$ is $i\mapsto\tfrac12+\tfrac12 i\mapsto\tfrac25+\tfrac15 i\mapsto\tfrac{3}{10}+\tfrac{1}{10}i\mapsto\cdots$. Every iterate lies on the same horocycle $x^2+(y-\tfrac12)^2=\tfrac14$ (the circle of radius $\tfrac12$ tangent to the real axis at $0$), and the points move along it toward the fixed point $0$, never leaving it. ``Preserves the horocycle'' refers to the circle as a whole: each point genuinely moves to a new location, but always to another point of the same circle, so the circle is carried onto itself. It does not mean the points are held fixed.}
  \label{fig:horocycle-orbit}
\end{figure}
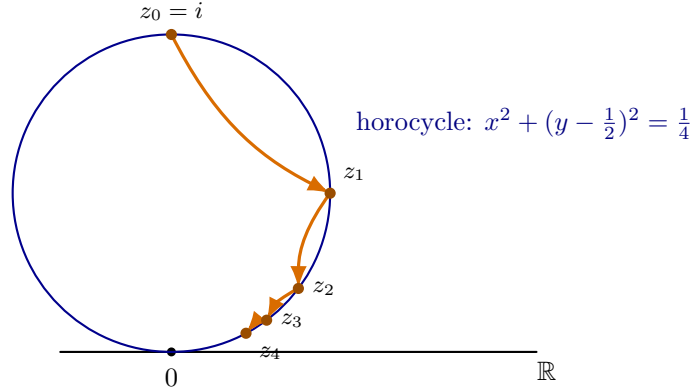

\begin{figure}[htbp]
  \centering
  \begin{tikzpicture}[>=Latex, line join=round, line cap=round, font=\small]
    \begin{scope}
      \draw[thick] (-2.6,0) -- (2.6,0);
      \node[anchor=west] at (2.4,-0.28) {$\mathbb{R}$};
      \fill (0,0) circle (1.6pt);
      \node[below] at (0,-0.06) {$0$};
      \foreach \r in {0.55,1.05,1.7} {
        \draw[blue!55!black] (0,\r) circle (\r);
      }
      \coordinate (P0) at ($(0,1.05)+(-30:1.05)$);
      \coordinate (P1) at ($(0,1.05)+(-80:1.05)$);
      \draw[orange!85!black, very thick, ->]
        (P0) .. controls ($(0,1.05)+(-50:1.32)$) .. (P1);
      \fill[orange!60!black] (P0) circle (1.7pt);
      \node[anchor=west] at (1.55,1.55)
        {\begin{tabular}{@{}l@{}}$\widetilde M_R$ slides along\\ the horocycle\end{tabular}};
      \node[font=\bfseries] at (0,4.30) {Fixed point at $0$};
      \node[align=center] at (0,3.86) {horocycles $=$ circles\\tangent to the real axis};
    \end{scope}
    \begin{scope}[xshift=6.3cm]
      \node[font=\bfseries] at (0.2,2.45) {$S(z)=-1/z$};
      \node at (0.2,2.05) {$0\mapsto\infty$};
      \draw[thick,->] (-0.85,1.6) -- (1.25,1.6);
    \end{scope}
    \begin{scope}[xshift=10.4cm]
      \draw[thick] (-2.6,0) -- (2.6,0);
      \node[anchor=west] at (2.4,-0.28) {$\mathbb{R}$};
      \foreach \y in {0.8,1.45,2.1} {
        \draw[teal!70!black] (-2.6,\y) -- (2.6,\y);
      }
      \coordinate (Q0) at (1.4,1.45);
      \coordinate (Q1) at (-0.5,1.45);
      \draw[orange!85!black, very thick, ->] (Q0) -- (Q1);
      \fill[orange!60!black] (Q0) circle (1.7pt);
      \node[anchor=south] at (0.45,1.55) {$z\mapsto z-\tfrac1R$};
      \node[font=\bfseries] at (0,4.30) {Fixed point at $\infty$};
      \node at (0,3.92) {horocycles $=$ horizontal lines};
    \end{scope}
  \end{tikzpicture}
  \caption{Conjugacy carrying the parabolic element $\widetilde M_R$ to a horizontal translation. Left: $\widetilde M_R$ fixes the boundary point $0$, and the horocycles based at $0$ are the Euclidean circles in $\mathbb{H}$ tangent to the real axis at the origin; the map slides each point along the horocycle through it. Right: the involution $S(z)=-1/z$ sends the fixed point $0$ to $\infty$, and in these coordinates $\widetilde M_R$ becomes the rigid horizontal translation $z\mapsto z-1/R$, which preserves the horizontal lines $\operatorname{Im}z=\text{const}$, the horocycles based at $\infty$. Conjugation thus turns the curved flow around $0$ into a rigid shift, so $\widetilde M_R$ carries each horocycle onto itself.}
  \label{fig:horocycle-conjugacy}
\end{figure}
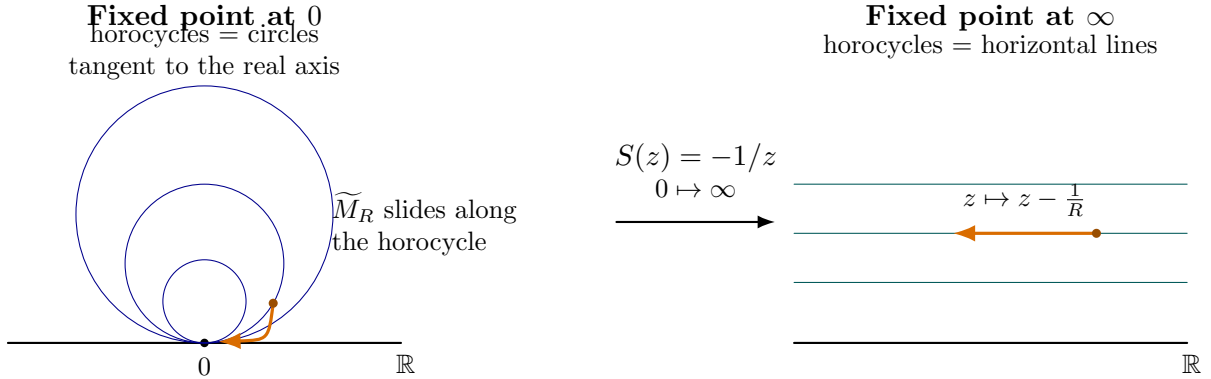

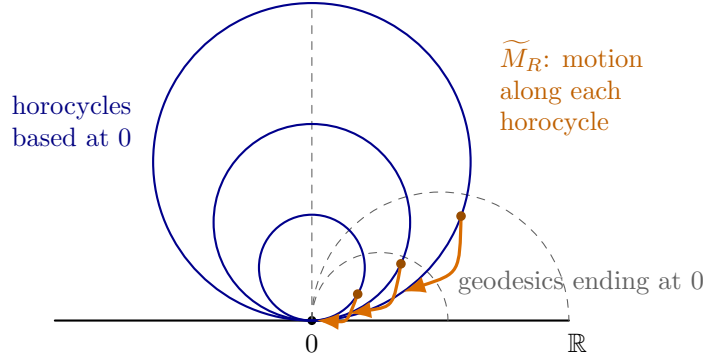
\begin{figure}[htbp]
  \centering
  \begin{tikzpicture}[>=Latex, line join=round, line cap=round, font=\small]
    \draw[thick] (-3.4,0) -- (3.4,0);
    \node[anchor=west] at (3.25,-0.30) {$\mathbb{R}$};
    \fill (0,0) circle (1.7pt);
    \node[below] at (0,-0.05) {$0$};
    \draw[gray!75!black, dashed] (0,0) -- (0,4.2);
    \draw[gray!75!black, dashed] (0,0) arc (180:0:0.9);
    \draw[gray!75!black, dashed] (0,0) arc (180:0:1.7);
    \node[gray!75!black, anchor=west] at (1.8,0.5) {geodesics ending at $0$};
    \foreach \r in {0.7,1.3,2.1} {
      \draw[blue!55!black, thick] (0,\r) circle (\r);
    }
    \node[blue!50!black, anchor=east] at (-2.25,2.6)
      {\begin{tabular}{@{}r@{}}horocycles\\ based at $0$\end{tabular}};
    \coordinate (A0) at ($(0,2.1)+(-20:2.1)$);
    \coordinate (A1) at ($(0,2.1)+(-55:2.1)$);
    \draw[orange!85!black, very thick, ->]
      (A0) .. controls ($(0,2.1)+(-37:2.45)$) .. (A1);
    \fill[orange!60!black] (A0) circle (1.8pt);
    \coordinate (B0) at ($(0,1.3)+(-25:1.3)$);
    \coordinate (B1) at ($(0,1.3)+(-68:1.3)$);
    \draw[orange!85!black, very thick, ->]
      (B0) .. controls ($(0,1.3)+(-46:1.55)$) .. (B1);
    \fill[orange!60!black] (B0) circle (1.8pt);
    \coordinate (C0) at ($(0,0.7)+(-30:0.7)$);
    \coordinate (C1) at ($(0,0.7)+(-85:0.7)$);
    \draw[orange!85!black, very thick, ->]
      (C0) .. controls ($(0,0.7)+(-57:0.86)$) .. (C1);
    \fill[orange!60!black] (C0) circle (1.8pt);
    \node[orange!75!black, anchor=west] at (2.35,3.05)
      {\begin{tabular}{@{}l@{}}$\widetilde M_R$: motion\\ along each\\ horocycle\end{tabular}};
  \end{tikzpicture}
  \caption{The parabolic element $\widetilde M_R$ preserves the horocycles based at $0$. Each circle is tangent to the real axis at the fixed point $0$ and meets orthogonally the geodesics with ideal endpoint $0$ (shown dashed: the vertical half-line and the semicircles meeting the real axis at $0$). The map $\widetilde M_R$ carries each horocycle onto itself, sliding points along it toward the fixed point.}
  \label{fig:horocycle-invariance}
\end{figure}

\FloatBarrier
\section{Higher-dimensional generalization}\label{sec:higher-d}

The construction goes through in arbitrary dimension $n\ge 1$ with the evident replacements. For $R,h>0$, consider the cone in $\mathbb{R}^{n+1}$ with apex at the origin, axis along $e_{n+1}$, height $h$, and base the $n$-ball of radius $R$ in the hyperplane $x_{n+1}=h$, whose rim is the $(n-1)$-sphere of radius $R$ (for $n=1$ this base is the interval $[-R,R]$ with rim $S^0=\{\pm R\}$, and the cone is a pair of segments); identify $\mathbb{R}^n$ with $\{x_{n+1}=0\}$. Set $f_R(0):=0$. For $\mathbf{x}\in\mathbb{R}^n\setminus\{0\}$, draw the line from $\mathbf{x}$ to $(0,h):=(0,\dots,0,h)\in\mathbb{R}^{n+1}$, let $P$ be its intersection with the lateral surface of the cone, and let $f_R(\mathbf{x})$ be the orthogonal projection of $P$ onto $\mathbb{R}^n$.

\begin{proposition}\label{prop:higher-d}
With this construction,
\[
  f_R(\mathbf{x}) \;=\; \frac{\mathbf{x}}{1+|\mathbf{x}|/R}, \qquad |f_R(\mathbf{x})| \;=\; \frac{R\,|\mathbf{x}|}{R+|\mathbf{x}|},
\]
independently of $h$. The map $f_R$ is a homeomorphism $\mathbb{R}^n\to B^n_R:=\{y\in\mathbb{R}^n:|y|<R\}$ and a $C^1$-diffeomorphism with $C^1$ inverse; it is smooth away from the origin but not $C^2$ at the origin. The statements corresponding to Propositions~\ref{prop:lens}, \ref{prop:cross-ratio}, Theorems~\ref{thm:semigroup} and~\ref{thm:group}, Corollary~\ref{cor:iteration} and Proposition~\ref{prop:gen} hold with the same domain qualifications as in dimension two, with rays replacing complex arguments, $\mathrm{O}(n)$-equivariance (that is, $f_R(Q\mathbf{x})=Q\,f_R(\mathbf{x})$ for every orthogonal matrix $Q\in\mathrm{O}(n)$ and every $\mathbf{x}\in\mathbb{R}^n$; see Step~3 of the proof) replacing rotation equivariance, and vector-valued ODEs replacing complex-valued ones.
\end{proposition}

\begin{proof}
\emph{Step 1 (reduction to a $2$-plane).} The case $\mathbf{x}=0$ is trivial as before. For $\mathbf{x}\ne 0$, let $\Pi$ denote the $2$-plane in $\mathbb{R}^{n+1}$ spanned by $\mathbf{x}$ and the vector $e_{n+1}$. Then $\Pi$ contains $0$, $\mathbf{x}$, $(0,h)$, the construction line $L_{\mathbf{x}}$, and the cone axis. The intersection of $\Pi$ with the \emph{lateral surface} of $\mathcal{K}_{R,h}$ consists of two line segments emanating from the origin (the two cone edges in $\Pi$), exactly as in the $2$-dimensional case. Set up coordinates on $\Pi$ so that the origin is $(0,0)$, the cone axis is the positive vertical axis, and $\mathbf{x}$ has coordinates $(|\mathbf{x}|,0)$. The argument now reduces to the planar problem already solved in the proof of Proposition~\ref{prop:p10}, with $|\mathbf{x}|$ in place of $|z|$, yielding $|f_R(\mathbf{x})|=R|\mathbf{x}|/(R+|\mathbf{x}|)$.

\emph{Step 2 (recovering the vector $f_R(\mathbf{x})$).} The orthogonal projection $f_R(\mathbf{x})\in\mathbb{R}^n$ lies on the same ray from the origin as $\mathbf{x}$ inside $\mathbb{R}^n$ (for the same reason as before: $P$ lies in the half-plane of $\Pi$ on the same side as $\mathbf{x}$, and dropping the perpendicular preserves this side). Hence $f_R(\mathbf{x})=|f_R(\mathbf{x})|\cdot \mathbf{x}/|\mathbf{x}|=R\,\mathbf{x}/(R+|\mathbf{x}|)=\mathbf{x}/(1+|\mathbf{x}|/R)$.

\emph{Step 3 (transferring the structural results).} Each of the corresponding results uses only the formula $f_R(z)=z/(1+|z|/R)$, the magnitude formula \eqref{eq:mag}, and equivariance under rotations. The formula has the same shape in any dimension; complex arguments are replaced by rays from the origin, rotation equivariance becomes $\mathrm{O}(n)$-equivariance, and the ODE is read as a vector-valued equation $\dot{\mathbf{x}}=-|\mathbf{x}|\mathbf{x}/R$. Here $\mathrm{O}(n)$-equivariance means that $f_R$ commutes with the orthogonal group $\mathrm{O}(n)=\{Q\in\mathbb{R}^{n\times n}:Q^\top Q=I_n\}$ of all rotations and reflections fixing the origin:
\[
  f_R(Q\mathbf{x})=Q\,f_R(\mathbf{x})\qquad\text{for every }Q\in\mathrm{O}(n)\text{ and }\mathbf{x}\in\mathbb{R}^n.
\]
This is immediate from the radial form $f_R(\mathbf{x})=\tfrac{R}{R+|\mathbf{x}|}\mathbf{x}$: since every $Q\in\mathrm{O}(n)$ preserves the Euclidean norm, $|Q\mathbf{x}|=|\mathbf{x}|$, and is linear, one has $f_R(Q\mathbf{x})=\tfrac{R}{R+|Q\mathbf{x}|}(Q\mathbf{x})=Q\bigl(\tfrac{R}{R+|\mathbf{x}|}\mathbf{x}\bigr)=Q\,f_R(\mathbf{x})$ for $\mathbf{x}\ne 0$, and trivially at $\mathbf{x}=0$. It is the $n$-dimensional counterpart of the rotation equivariance used throughout the planar argument (where the relevant group is $\mathrm{O}(2)$), and it lets each higher-dimensional proof be reduced by an orthogonal change of coordinates to the planar case, exactly as in Step~1. With those replacements, the same proofs go through. (Cross-ratio preservation, in particular, makes sense on any ray from the origin in $\mathbb{R}^n$.) The regularity statement transfers as well: the Jacobian computation of Section~\ref{sec:setup}, including the bound on the $\mathbf{x}\otimes\mathbf{x}/|\mathbf{x}|$ term, is dimension-free, so $f_R$ is globally $C^1$ with $Df_R(0)=I_n$, smooth away from the origin, and not $C^2$ there.
\end{proof}

The map $\mathbf{x}\mapsto \mathbf{x}/(1+|\mathbf{x}|/R)$ is a standard radial homeomorphism from $\mathbb{R}^n$ onto the open ball $B^n_R$. Adjoining one limiting endpoint for each ray gives the closed-ball radial compactification; further collapsing that boundary sphere to a point yields the usual one-point compactification $S^n$ \cite{munkres2000}. Proposition~\ref{prop:higher-d} says that this radial homeomorphism is realized by an elementary single-cone construction in any dimension. Thus the target is the open interval when $n=1$, the open disk when $n=2$, and the open ball in higher dimensions.

\begin{remark}[The Self-Directrix Theorem in higher dimensions]
For $n\ge 2$, the Euclidean theorems of Section~\ref{sec:geometry} extend with the same proof. (For $n=1$ the construction degenerates to the one-dimensional radial interval map, with no conic content.) Let $H=\{x_1=d\}\subset\mathbb{R}^n$ be a hyperplane at distance $d>0$ from $O$. For $\mathbf{x}\in H$ and $w=f_R(\mathbf{x})$, the inverse relation $\mathbf{x}=w/(1-|w|/R)$ gives $w_1=d(1-|w|/R)$, hence $\operatorname{dist}(w,H)=d-w_1=(d/R)\,|w|$ and
\[
  |Ow|=\frac{R}{d}\,\operatorname{dist}(w,H).
\]
Thus $f_R$ carries $H$ onto a patch of the quadric of revolution with focus $O$, directrix hyperplane $H$, eccentricity $R/d$, and semi-latus rectum $R$ in every meridian section: the rotational analogue of the Self-Directrix conic, obtained by revolving that planar conic about the axis through $O$ perpendicular to $H$, with the same trichotomy in $d>R$, $d=R$, $d<R$. The corresponding arc/patch statement in the Confocal--Codirectrix Theorem extends in the same way. Let $S=S(\lambda,e,\mathbf{a})\subset\mathbb{R}^n$, with $\mathbf{a}\in S^{n-1}$ a unit axis direction, be the focal quadric of revolution obtained by rotating the planar focal locus $\mathcal{B}(\lambda,e,\alpha)$ of Section~\ref{sec:confocal} (its planar axis aligned with $\mathbf{a}$) about the axis through $O$ in direction $\mathbf{a}$; this is a conic when $n=2$, a surface when $n=3$, and a hypersurface for $n>3$. Concretely, in spherical coordinates $\mathbf{x}=\rho\,\mathbf{u}$ with $\mathbf{u}$ a unit vector, writing $\psi$ for the angle between $\mathbf{u}$ and $\mathbf{a}$ (so $\cos\psi=\mathbf{u}\cdot\mathbf{a}$),
\[
  S=\Bigl\{\rho\,\mathbf{u}:\ \tfrac{1}{\rho}=\tfrac{1}{\lambda}+\tfrac{e}{\lambda}\cos\psi,\ \tfrac{1}{\rho}>0\Bigr\},
\]
the focal radius depending on $\mathbf{u}$ only through $\psi$. Because $f_R$ is $\mathrm{O}(n)$-equivariant and radial, it fixes $\mathbf{u}$ (hence $\psi$) and acts by the reciprocal-modulus shift $1/\rho\mapsto1/\rho+1/R$ (Proposition~\ref{prop:lens}), which alters only the constant term:
\[
  \frac{1}{\rho'}=\Bigl(\frac{1}{\lambda}+\frac{1}{R}\Bigr)+\frac{e}{\lambda}\cos\psi
  =\frac{1}{\lambda'}+\frac{e'}{\lambda'}\cos\psi,
  \qquad \lambda'=\frac{R\lambda}{R+\lambda},\quad e'=\frac{Re}{R+\lambda},
\]
exactly as in the planar Theorem~\ref{thm:confocal}. Thus $f_R(S)$ is a patch of the focal quadric of revolution $S(\lambda',e',\mathbf{a})$ with the same focus $O$, the same axis $\mathbf{a}$, and the same directrix hyperplane $H_{\mathbf{a},\delta}=\{x\in\mathbb{R}^n:x\cdot\mathbf{a}=\delta\}$ (since $\delta'=\lambda'/e'=\lambda/e=\delta$), and with strictly smaller eccentricity $e'<e$; the meridian sections are precisely the planar arcs of Theorem~\ref{thm:confocal}, and the three cases transfer verbatim. The argument is dimension-independent because it reduces, meridian half-plane by meridian half-plane, to the planar case.
\end{remark}

\section{The pullback metric and the flat-disk completion}\label{sec:metric}

Let $g_E$ denote the standard Euclidean metric on $\mathbb{R}^n$. Since $f_R$ is not smooth at the origin, the pullback formula should not be read as producing a smooth Riemannian metric on all of $\mathbb{R}^n$. It is cleanest to define the global distance directly, by declaring $f_R$ to be an isometry onto the open Euclidean ball $B^n_R$:
\[
  d_R(p,q):=\bigl|f_R(p)-f_R(q)\bigr|.
\]
As defined, $d_R$ is a pullback \emph{chordal} distance. It coincides with the intrinsic length distance it induces: since $B^n_R$ is convex, the straight segment between $f_R(p)$ and $f_R(q)$ stays in $B^n_R$ and already realizes the infimum of path lengths (Step~2 of the proof below makes this precise), so calling $d_R$ a length metric is justified. On $\mathbb{R}^n\setminus\{0\}$, where $f_R$ is smooth, this length metric is represented by a smooth Riemannian tensor, the pullback $f_R^*(g_E|_{B^n_R})$, which is flat in every dimension because it is the pullback of a flat metric by a diffeomorphism. In dimension two, the coordinate formula is
\[
  g_R \;:=\; f_R^*\bigl(g_E|_{B^2_R}\bigr).
\]

\begin{proposition}[Induced length metric, $n=2$]\label{prop:metric}
On $\mathbb{R}^2\setminus\{0\}$, in polar coordinates $(\rho,\theta)$,
\[
  g_R \;=\; \frac{R^4}{(R+\rho)^4}\,d\rho^2 \;+\; \frac{R^2\rho^2}{(R+\rho)^2}\,d\theta^2.
\]
The metric is flat on $\mathbb{R}^2\setminus\{0\}$ and, as a length metric, is globally isometric through $f_R$ to the open Euclidean disk $B^2_R$. The $g_R$-distance from $0$ to a point of Euclidean radius $\rho$ is $R\rho/(R+\rho)$; the $g_R$-diameter of $\mathbb{R}^2$, in the supremal sense, is $2R$ and is not attained by any pair of points (it is attained only in the completion); the total $g_R$-area of $\mathbb{R}^2$ is $\pi R^2$. The metric space $(\mathbb{R}^2,d_R)$ is incomplete, and its metric completion is isometric to the closed Euclidean disk $\overline{B^2_R}$ (with the induced Euclidean length metric).
\end{proposition}

\begin{proof}
\emph{Step 1 (computing the metric).} The Euclidean metric on the target $\mathbb{R}^2$, written in polar coordinates $(s,\theta')$ with $s$ the radial coordinate, is the standard expression
\[
  g_E \;=\; ds^2 + s^2\,d\theta'^2.
\]
The map $f_R$, in polar coordinates, sends $(\rho,\theta)\mapsto(s,\theta')$ where $\theta'=\theta$ (rotation equivariance) and $s=R\rho/(R+\rho)$ (magnitude formula). Differentiating $s$ with respect to $\rho$ by the quotient rule:
\[
  \frac{ds}{d\rho} \;=\; \frac{R(R+\rho) - R\rho\cdot 1}{(R+\rho)^2} \;=\; \frac{R^2}{(R+\rho)^2}.
\]
Hence
\[
  ds \;=\; \frac{R^2}{(R+\rho)^2}\,d\rho \quad\text{and}\quad d\theta' = d\theta.
\]
The pullback $g_R=f_R^*g_E$ is computed by substituting these:
\[
  g_R \;=\; \left(\frac{R^2}{(R+\rho)^2}\right)^{\!2}\!d\rho^2 \;+\; \left(\frac{R\rho}{R+\rho}\right)^{\!2}\!d\theta^2
  \;=\; \frac{R^4}{(R+\rho)^4}\,d\rho^2 \;+\; \frac{R^2\rho^2}{(R+\rho)^2}\,d\theta^2,
\]
as claimed.

\emph{Step 2 (flatness).} On $\mathbb{R}^2\setminus\{0\}$ the map $f_R$ is smooth, so the usual pullback argument applies there: the pullback of the flat Euclidean metric is flat. Globally, the length metric defined by this formula is isometric through $f_R$ to the open Euclidean disk $(B_R^2,g_E|_{B_R^2})$. Here the chordal definition $d_R(p,q)=|f_R(p)-f_R(q)|$ coincides with the induced length (geodesic) metric \cite{burago2001}: because $B_R^2$ is convex, the straight segment between any two of its points stays inside, so the Euclidean chord between $f_R(p)$ and $f_R(q)$ is already the shortest path, and pulling it back by $f_R$ realizes that infimum. This is the sense in which the length space $(\mathbb{R}^2,d_R)$ is isometric to the flat open Euclidean disk, even though the coordinate coefficients of $g_R$ are not those of a smooth Riemannian metric at the origin. (In Cartesian coordinates the tensor $g_R=Df_R^{\!\top}Df_R$ does extend \emph{continuously} to the origin, with value the identity, because $Df_R$ is continuous and $Df_R(0)=I$; but the extension is not smooth there, consistent with $f_R$ being $C^1$ and not $C^2$.)

\emph{Step 3 (distance, area, and completion).} The remaining claims follow from the isometry with the open Euclidean disk. The $g_R$-distance from $0$ to a point of Euclidean radius $\rho$ equals the Euclidean radius of its image, namely $R\rho/(R+\rho)$.
The diameter is the supremal Euclidean diameter of $B_R^2$, namely $2R$; it is not attained because the antipodal points at distance $2R$ lie only on the boundary of the completion. The total area is the Euclidean area of $B_R^2$, namely $\pi R^2$, and the metric completion is isometric to the closed Euclidean disk $\overline{B_R^2}$.

For explicit incompleteness, take any sequence $\rho_n\to\infty$ on a fixed ray, for instance $\rho_n=n$. Its image under $f_R$ converges to the corresponding boundary point of $\overline{B_R^2}$, so the sequence is $g_R$-Cauchy but has no limit in $\mathbb{R}^2$. Equivalently, the radial segment joining the points of radii $\rho_m$ and $\rho_n$ has $g_R$-length
\[
  \left|\frac{R\rho_m}{R+\rho_m}-\frac{R\rho_n}{R+\rho_n}\right|\to0.
\]
The completion adds one boundary point for each ray. The limiting coordinate circles have lengths
\[
  \lim_{\rho\to\infty}\sqrt{(g_R)_{\theta\theta}}\cdot 2\pi
  =\lim_{\rho\to\infty}\frac{R\rho}{R+\rho}\cdot 2\pi=2\pi R.
\]
The metric induced on the added boundary as a subset of the completed disk is the Euclidean chord metric; the number $2\pi R$ is the limiting intrinsic length of the coordinate circles, equivalently the Euclidean circumference of the boundary circle.
\end{proof}

The induced length metric $g_R$ realizes the entire complex plane as a flat disk of finite Euclidean radius $R$ (cf.\ Figure~\ref{fig:global}). This complements the spherical ``compactification'' of $\mathbb{C}$ via stereographic projection \cite{needham1997}: there the boundary at infinity is a single point, and the induced metric is round; here the boundary at infinity is a circle, the limiting coordinate circles have intrinsic length $2\pi R$, and the induced metric is flat.

\section{A characterization theorem}\label{sec:characterization}

The properties identified above suggest that $f_R$ is pinned down by a few axioms, and it admits two complementary characterizations. Theorem~\ref{thm:sd-char} characterizes $f_R$ \emph{geometrically}, assuming no regularity, from the self-directrix property: among radial maps fixing $O$, it is the only one carrying every line to a focus--directrix conic with that line as its directrix. The theorem below characterizes $f_R$ \emph{algebraically and topologically}, from ray preservation, a M\"obius radial action, bounded image, and first-order normalization. The formulation makes explicit the normalizations that are actually needed: the map must preserve each ray, not merely commute with rotations, and it must be the identity to first order at the origin.

\begin{theorem}[Normalized characterization]\label{thm:char}
Let $\Phi:\mathbb{C}\to\mathbb{C}$ be a continuous map satisfying:
\begin{enumerate}[label=\textup{(\roman*)}]
  \item $\Phi(0)=0$;
  \item $\Phi$ preserves each ray from the origin: for every $\rho\ge 0$ and $\theta\in\mathbb{R}$ there is a number $\sigma(\rho)\ge 0$ such that
  \[
    \Phi(\rho e^{i\theta})=\sigma(\rho)e^{i\theta};
  \]
  equivalently, $\Phi$ is rotation equivariant and maps the positive real ray $[0,\infty)$ into itself;
  \item the radial action $\rho\mapsto\sigma(\rho)$ is the restriction to $[0,\infty)$ of a strictly increasing \emph{real} M\"obius transformation (a real projective transformation of $\mathbb{RP}^1$) defined on all of $[0,\infty)$ and fixing $0$;
  \item $|\Phi(z)|\to R$ as $|z|\to\infty$ for some $R\in(0,\infty)$; equivalently, the radial M\"obius action of (iii) has a finite positive limit at infinity, which is then the supremum of $|\Phi|$ by strict monotonicity;
  \item $\Phi$ is tangent to the identity at the origin, in the radial sense
  \[
    \lim_{\rho\to 0^+}\frac{\sigma(\rho)}{\rho}=1.
  \]
\end{enumerate}
Then $\Phi=f_R$.
\end{theorem}

\begin{proof}
By (ii), the whole map is determined by its radial function $\sigma$:
\[
  \Phi(\rho e^{i\theta})=\sigma(\rho)e^{i\theta}.
\]
By (iii), $\sigma$ is the restriction of a real M\"obius transformation fixing $0$, so
\[
  \sigma(\rho)=\frac{a\rho+b}{c\rho+d}, \qquad a,b,c,d\in\mathbb{R},\quad ad-bc\ne 0.
\]
By (iii) it is defined and finite at $\rho=0$, which forces $d\ne0$; then $\sigma(0)=b/d=0$ gives $b=0$. Dividing through by $d$ and renaming $a/d,\,c/d$ as $a,\,\beta$,
\[
  \sigma(\rho)=\frac{a\rho}{1+\beta\rho}.
\]
Here $a\ne0$. Because (iii) requires $\sigma$ to be defined on all of $[0,\infty)$, the denominator $1+\beta\rho$ cannot vanish for any $\rho\ge 0$; this forces $\beta\ge 0$, since $\beta<0$ would place the pole $\rho=-1/\beta$ inside $[0,\infty)$. Strict monotonicity gives
\[
  \sigma'(\rho)=\frac{a}{(1+\beta\rho)^2}>0,
\]
so $a>0$. The finite positive limiting radius in (iv) forces $\beta>0$: if $\beta=0$ then $\sigma(\rho)=a\rho\to\infty$ as $\rho\to\infty$, contradicting (iv). Hence
\[
  R=\lim_{\rho\to\infty}\frac{a\rho}{1+\beta\rho}=\frac{a}{\beta},
\]
so $a=\beta R$. The normalization (v) gives
\[
  1=\lim_{\rho\to0^+}\frac{\sigma(\rho)}{\rho}=a.
\]
Thus $a=1$ and $\beta=1/R$, hence
\[
  \sigma(\rho)=\frac{\rho}{1+\rho/R}=\frac{R\rho}{R+\rho}.
\]
Substituting into the ray-preserving representation in (ii) gives
\[
  \Phi(\rho e^{i\theta})=\frac{\rho}{1+\rho/R}e^{i\theta}=f_R(\rho e^{i\theta}),
\]
so $\Phi=f_R$.
\end{proof}

\begin{remark}[Why the normalizations are necessary]
If condition (ii) is weakened to rotation equivariance alone, a fixed phase factor is possible, and more generally a radius-dependent twist is possible: maps of the form
\[
  \Phi(\rho e^{i\theta})=\sigma(\rho)e^{i(\theta+\alpha(\rho))}
\]
still commute with rotations. Ray preservation, or equivalently rotation equivariance together with preservation of the positive real ray, removes this angular freedom. If condition (v) is omitted, the remaining conditions determine only the one-parameter family
\[
  \sigma_\beta(\rho)=\frac{\beta R\rho}{1+\beta\rho},\qquad \beta>0,
\]
all of which have limiting radius $R$. The first-order normalization at the origin fixes this remaining radial scale.
\end{remark}

\begin{remark}[The higher-dimensional version]
The theorem holds verbatim in $\mathbb{R}^n$ for the map of Section~\ref{sec:higher-d}, with the same proof. One replaces rotation equivariance by $\mathrm{O}(n)$-equivariance and reads condition (ii) as ray preservation in the form $\Phi(\rho\,\mathbf{u})=\sigma(\rho)\,\mathbf{u}$ for every unit vector $\mathbf{u}\in S^{n-1}$ and $\rho\ge 0$; conditions (i), (iii)--(v) are stated in terms of the radial action $\sigma$ exactly as before. Since the entire argument reduces to determining the scalar M\"obius function $\sigma$ on $[0,\infty)$, the dimension plays no role, and the conclusion is $\Phi=f_R$ on $\mathbb{R}^n$.
\end{remark}

\section{Summary}\label{sec:summary}

The cone construction gives the radial homeomorphism
\[
  f_R(z)=\frac{z}{1+|z|/R},
\]
independently of the cone height. Its basic algebra is the lens-equation identity
\[
  \frac{1}{|f_R(z)|}=\frac{1}{|z|}+\frac{1}{R},
\]
which makes the family closed under composition with curvature addition $1/R_3=1/R_1+1/R_2$. Including the inverse maps on their natural domains gives the one-parameter partial group $\phi_\kappa(z)=z/(1+\kappa|z|)$; for fixed $R$, the path $f_R^t=\phi_{t/R}$ is the flow of $\dot z=-|z|z/R$.

Projectively, the restriction of $f_R$ to any ray is a one-dimensional M\"obius transformation: cross-ratios on rays are preserved, and the forward and backward images of a point form a harmonic range with the point and the origin. The main Euclidean result is the Self-Directrix Theorem: every line $\ell$ not through $O$ maps to the open arc, between the two latus-rectum endpoints, of the conic with focus $O$, directrix $\ell$, eccentricity $R/d$, and semi-latus rectum $R$. Thus the single distance parameter $d=\operatorname{dist}(O,\ell)$ determines the ellipse/parabola/hyperbola trichotomy. The Confocal--Codirectrix Theorem extends this from lines to all focal polar loci: $f_R$ fixes the focus and directrix of each such locus and only lowers its eccentricity, by $1/e\mapsto1/e+\delta/R$, with the Self-Directrix Theorem recovered as the limiting line member of each fixed-directrix pencil; its parameter laws also re-emerge chord by chord from Askwith's classical focal-chord harmonic-mean theorem (Section~\ref{sec:focal-chord}). The image of a circle, by contrast, is in general a circular quartic rather than a conic: circles centered at $O$ remain circles, a circle through $O$ maps to the rational inverse of one branch of a conchoid of Nicomedes, and every other circle to a smooth closed curve on a circular quartic of geometric genus~$1$ (Section~\ref{sec:circle-image}); Table~\ref{tab:curve-images} collects these images.

The later results show that the same formula extends to $\mathbb{R}^n$, realizes the plane as a flat open disk of finite diameter and area through the pullback length metric, and is characterized among ray-preserving maps by M\"obius radial action, bounded image, and first-order normalization at the origin. The further constructions and interpretations in Section~\ref{sec:alt-constructions} all reflect the same reciprocal-radius law.

\end{document}